\documentclass[a4paper,10pt]{amsart}
\usepackage{pb-diagram}
\usepackage{xypic}
\usepackage{amscd}
\usepackage{bm}
\usepackage{amsmath,amssymb,amsfonts,amssymb,amscd,bbm,amsthm,mathrsfs,hyperref}
\usepackage{latexsym}
\usepackage{amsfonts}
\usepackage{amsthm}
\usepackage{indentfirst}
 \newtheorem{thm}{Theorem}[section]
 \newtheorem{cor}[thm]{Corollary}
 \newtheorem{lem}[thm]{Lemma}
 \newtheorem{prop}[thm]{Proposition}
 \newtheorem{defn}[thm]{Definition}
 \newtheorem{ex}[thm]{Example}
 \newtheorem{rem}[thm]{Remark}
 
 \newtheorem{quest}[thm]{Question}

 \newcommand{\Hom}{\mathrm{Hom}}

\title{DG Algebra structures on the quantum affine $n$-space $\mathcal{O}_{-1}(k^n)$}
\author{X.-F. Mao}
\address{Department of Mathematics, Shanghai University, Shanghai 200444, China}
\email{xuefengmao@shu.edu.cn}

\author{X.-T. Wang}
\address{Department of Mathematics, Howard University, Washington DC, 20059, USA}
\email{xingting.wang@Howard.edu}

\author{M.-Y.Zhang}
\address{Department of Mathematics, Shanghai University, Shanghai 200444, China}
\email{zmy1023@shu.edu.cn}

\date{}

\begin{document}

\maketitle \def\abstactname{abstact}
\begin{abstract}
Let $\mathcal{A}$ be a connected cochain DG algebra, whose underlying graded algebra $\mathcal{A}^{\#}$ is the quantum affine $n$-space $\mathcal{O}_{-1}(k^n)$. We compute all possible differential structures of $\mathcal{A}$ and show that there exists a one-to-one correspondence between
 $$\{\text{cochain DG algebra}\,\,\mathcal{A}\,|\,\mathcal{A}^{\#}=\mathcal{O}_{-1}(k^n)\}$$ and the $n\times n$ matrices $M_n(k)$.
 For any $M\in M_n(k)$, we write $\mathcal{A}_{\mathcal{O}_{-1}(k^3)}(M)$ for the DG algebra corresponding to it.
  We also study the isomorphism problems of these non-commutative DG algebras.
 For the cases $n\le 3$, we check their homological properties. Unlike the case of $n=2$, we discover that not all of them are Calabi-Yau when $n=3$.
 In spite of this, we recognize those Calabi-Yau ones case by case. In brief, we solve the problem on how to judge whether a given such DG algebra $\mathcal{A}_{\mathcal{O}_{-1}(k^3)}(M)$ is Calabi-Yau.
\end{abstract}

\maketitle

\section*{introduction}
Along this paper,  $k$ will denote an algebraically closed field of characteristic zero. Recall that a cochain DG $k$-algebra is a graded $k$-algebra together with a differential of degree $1$, which satisfies the Leibniz rule. Algebras with additional differential structures provide convenient models for intrinsic and homological information from diverse array of areas ranging from representation theory to symplectic and algebraic geometry. For example, a Gorenstein topological space $X$ in algebraic topology is characterized by the Gorensteinness of the cochain algebra $C^*(X;k)$ of
normalized singular cochains on $X$ (cf. \cite{FHT1,Gam}). And it is well known that the rational homotopy type of a simply connected space of finite type is encoded in its Sullivan model.

In the derived algebraic geometry, a fundamental fact discovered by A. Bondal and M. Van den Bergh is that any quasi-compact and quasi-separated scheme $X$ is affine in the derived sense, i.e. $\mathrm{D}_{Qcoh}(X)$  is equivalent to $\mathrm{D}(\mathcal{A})$ for a suitable DG algebra $\mathcal{A}$ (cf. \cite{BV}).
By \cite[Proposition 3.13]{Lun} and \cite[$\S$ 6.2]{Rou},
the regular property of a quasi-compact and quasi-separated scheme $X$ under some mild conditions is equivalent to the homologically smoothness of the corresponding $\mathcal{A}$. In the smooth case, the triviality of the canonical bundle for the scheme is equivalent to the Calabi-Yau properties of the DG algebra $\mathcal{A}$ (cf.\cite{Kon2}).  Calabi-Yau DG algebras are introduced by Ginzburg in \cite{Gin},  and
have a multitude of connections to  representation theory, mathematical physics and non-commutative algebraic geometry.
Therefore the constructions and studies of Calabi-Yau DG algebras have become tremendous helpful to people working in different areas of mathematics.

 In \cite{HM}, the first  author and J.-W. He give a criterion for a connected cochain DG algebra to be $0$-Calabi-Yau, and prove that a locally finite connected cochain DG algebra is $0$-Calabi-Yau if and only if it is defined by a potential. For a $n$-Calabi-Yau connected cochain DG algebra $\mathcal{A}$, one sees that the full triangulated subcategory $\mathrm{D^b_{lf}}(\mathcal{A})$ of $\mathrm{D}(\mathcal{A})$ containing DG $\mathcal{A}$-modules with finite dimensional total cohomology is a $n$-Calabi-Yau triangulated category (cf. \cite{CV}). The notion of Calabi-Yau triangulated category was introduced by Kontsevich \cite{Kon1} in the late 1990s. Calabi-Yau triangulated  categories appear in string theory,  conformal field theory, Mirror symmetry, integrable system and
 representation theory of quivers and finite-dimensional algebras. Due to the applications of triangulated Calabi-Yau categories in the categorification of Fomin-Zelevinsky's cluster
algebras, they have become popular in representation theory.

Although it is meaningful to
discover some families of Calabi-Yau DG algebras,
it is generally quite complicated to tell whether a given DG algebra is Calabi-Yau, because the properties of a DG algebra are determined by the joint effects of its underlying graded algebra structure and differential structure.
If one considers a DG algebra $\mathcal{A}$ as a living thing, then the underlying graded algebra $\mathcal{A}^{\#}$  and the differential $\partial_{\mathcal{A}}$ are its body and soul, respectively. It is
an efficient way to create meaningful Calabi-Yau DG algebras on some well known regular graded algebras.
In \cite{MHLX}, \cite{MGYC} and \cite{MXYA}, DG down-up algebras, DG polynomial algebras and DG free algebras are introduced and  systematically studied, respectively. Moreover, it is interesting that non-trivial DG down-up algebras, non-trivial DG polynomial algebras and DG free algebras with $2$ degree $1$ variables are all Calabi-Yau DG algebras. These interesting results encourage us to continue the project.

This paper deals with a special family of cochain DG algebras whose underlying graded algebras are the quantum affine $n$-space $\mathcal{O}_{-1}(k^n)$, $n\ge 2$. There are several reasons for us to consider this particular class of DG algebras. Firstly, they are special DG skew polynomial algebras, which are parallel to the case of DG polynomial algebras in \cite{MGYC}. Secondly,
 they can be considered as an intermediate transition family of DG algebras between graded commutative DG algebras and DG free algebras generated in degree $1$ elements.  Thirdly,
the case of $n=3$ coincides with
a family of $3$-dimensional DG Sklyanin algebras in \cite{MWYZ}, where a connected cochain DG algebra $\mathcal{A}$ is called a $3$-dimensional Sklyanin algebra if its underlying graded algebra $\mathcal{A}^{\#}$ is the algebra
\begin{align*}
& S_{a,b,c}=\frac{k\langle x_1,x_2,x_3\rangle}{(f_1,f_2,f_3)},
\begin{cases}
f_1=ax_2x_3+bx_3x_2+cx_1^2\\
f_2=ax_3x_1+bx_1x_3+cx_2^2\\
f_3=ax_1x_2+bx_2x_1+cx_3^2,
\end{cases} (a,b,c)\in \Bbb{P}_k^2-\mathcal{D}, \\
& \text{and}\quad \mathcal{D}:=\{(1,0,0),(0,1,0),(0,0,1)\}\sqcup \{(a,b,c)|a^3=b^3=c^3=1\}.
\end{align*}
By \cite{MWYZ}, we have
$\partial_{\mathcal{A}}=0$ if either $a^2\neq b^2$ or $c\neq 0$. And it is possible for a $3$-dimensional Sklyanin algebra to be non-trivial if $a=-b, c=0$ or $a=b, c=0$. When  $a=-b, c=0$, $\mathcal{A}$ is actually a DG polynomial algebra, which is systematically studied in \cite{MGYC}. For the case $a=b, c=0$, it is proved in \cite{MWYZ} that $\mathcal{A}$ is uniquely determined by a $3\times 3$ matrix $M$ such that
\begin{align*}
\left(
                         \begin{array}{c}
                           \partial_{\mathcal{A}}(x_1)\\
                           \partial_{\mathcal{A}}(x_2)\\
                           \partial_{\mathcal{A}}(x_3)
                         \end{array}
                       \right)=M\left(
                         \begin{array}{c}
                           x_1^2\\
                           x_2^2\\
                           x_3^2
                         \end{array}
                       \right).
\end{align*}
In this paper, we will generalize this result. Beside these, it is proved in \cite{MH} that a connected cochain DG algebra $\mathcal{A}$ is a $0$-Calabi-Yau DG algebra if $\mathcal{A}^{\#}=k\langle x_1,x_2\rangle/(x_1x_2+x_2x_1)$ with $|x_1|=|x_2|=1$. The proof there relies on the classification of the differential of $\mathcal{A}$. Note that $\mathcal{A}^{\#}$ in this case is just the quantum plane $\mathcal{O}_{-1}(k^2)$. This motivates us to consider more general case. For the
  quantum affine $n$-space $\mathcal{O}_{-1}(k^n), n\ge 3$, we want to see what kind cochain DG algebras can be constructed over it.
  We describe all possible cochain DG algebra structures over $\mathcal{O}_{-1}(k^n)$ by the following theorem (see Theorem \ref{diffstr}). \\

\begin{bfseries}
Theorem \ A.
\end{bfseries}
Let $\mathcal{A}$ be a connected cochain DG algebra such that $\mathcal{A}^{\#}$ is the $k$-algebra with degree one generators $x_1,\cdots, x_n$ and relations $x_ix_j=-x_jx_i$, for all $1\le i<j\le n$.  Then  $\partial_{\mathcal{A}}$ is determined by a matrix $M=(m_{ij})_{n\times n}$ such that
\begin{align*}
\left(
                         \begin{array}{c}
                           \partial_{\mathcal{A}}(x_1)\\
                           \partial_{\mathcal{A}}(x_2)\\
                           \vdots  \\
                           \partial_{\mathcal{A}}(x_n)
                         \end{array}
                       \right)=M\left(
                         \begin{array}{c}
                           x_1^2\\
                           x_2^2\\
                           \vdots \\
                           x_n^2
                         \end{array}
                       \right).
\end{align*}
For any $M=(m_{ij})\in M_n(k)$, it is reasonable to define a connected cochain DG algebra $\mathcal{A}_{\mathcal{O}_{-1}(k^n)}(M)$ such that $$[\mathcal{A}_{\mathcal{O}_{-1}(k^n)}(M)]^{\#}=\mathcal{O}_{-1}(k^n)$$  and its differential $\partial_{\mathcal{A}}$ is defined by
\begin{align*}
\left(
                         \begin{array}{c}
                           \partial_{\mathcal{A}}(x_1)\\
                           \partial_{\mathcal{A}}(x_2)\\
                           \vdots  \\
                           \partial_{\mathcal{A}}(x_n)
                         \end{array}
                       \right)=M\left(
                         \begin{array}{c}
                           x_1^2\\
                           x_2^2\\
                           \vdots \\
                           x_n^2
                         \end{array}
                       \right).
\end{align*}
To consider the homological properties of $\mathcal{A}_{\mathcal{O}_{-1}(k^n)}(M)$, it is necessary to study the isomorphism problem.  We have the following theorem (see Theorem \ref{iso}).\\
\begin{bfseries}
Theorem \ B.
\end{bfseries}
Let $M$ and $M'$ be two matrixes in $M_n(k)$. Then $$\mathcal{A}_{\mathcal{O}_{-1}(k^n)}(M)\cong \mathcal{A}_{\mathcal{O}_{-1}(k^n)}(M')$$ if and only if there exists $C=(c_{ij})_{n\times n}\in \mathrm{QPL}_n(k)$ such that
$$M'=C^{-1}M(c_{ij}^2)_{n\times n}.$$

Here $\mathrm{QPL}_n(k)$ is the set of quasi-permutation matrixes in $\mathrm{GL}_n(k)$. One sees that $\mathrm{QPL}_n(k)$ is a subgroup of $\mathrm{GL}_n(k)$ (see Proposition \ref{subgroup}). By Theorem B, one sees that any DG algebra automorphism group of $\mathcal{A}_{\mathcal{O}_{-1}(k^n)}(M)$ is $\mathrm{QPL}_n(k)$'s  subgroup
$$ \{C=(c_{ij})_{n\times n}\in \mathrm{QPL}_n(k)\,|\,M=C^{-1}M(c_{ij}^2)_{n\times n}\}$$
for any $M\in M_n(k)$. Theorem B also indicates that one can define a right group action
  $$\chi: M_n(k) \times \mathrm{QPL}_n(k)\to M_n(k)$$
   of $\mathrm{QPL}_n(k)$ on $M_n(k)$  such that
$\chi[(M, C=(c_{ij})_{n\times n})]=C^{-1}M((c_{ij})^2)_{n\times n}$.
The set of all orbits of this group action is one to one correspondence with
the set of isomorphism classes of DG algebras in
$\{\mathcal{A}_{\mathcal{O}_{-1}(k^n)}(M)|M\in M_{n}(k)\}$.

We want to study various homological properties of $\{\mathcal{A}_{\mathcal{O}_{-1}(k^n)}(M)|M\in M_{n}(k)\}$. For the case $n=2$, we know that each $\mathcal{A}_{\mathcal{O}_{-1}(k^2)}(M)$ is a Koszul and Calabi-Yau DG algebra  by
\cite{MH}.  It is natural for one to ask wether each $\mathcal{A}_{\mathcal{O}_{-1}(k^n)}(M)$ is a Koszul and Calabi-Yau connected cochain DG algebra when $n\ge 3$.

It is worth noting that as $n$ grows large, the classifications, cohomology and homological properties of $\mathcal{A}_{\mathcal{O}_{-1}(k^n)}(M)$ become increasingly difficult to compute and study. This increased complexity is, in large part, due to the irregular increase of the number of cases one needs
to study separately. In this paper, we focus our attentions on the case that $n=3$. It involves further classifications and complicated matrix analysis.

In general, the cohomology graded algebra $H(\mathcal{A})$ of a cochain DG algebra $\mathcal{A}$ usually contains some homological information.
One sees that $\mathcal{A}$ is a Calabi-Yau DG algebra if the trivial DG algebra $(H(\mathcal{A}),0)$ is Calabi-Yau by \cite{MYY}, and it is proved in
 \cite{MH} that
 a connected cochain DG algebra $\mathcal{A}$ is a Koszul Calabi-Yau DG algebra if $H(\mathcal{A})$ belongs to one of the following cases:
\begin{align*}
& (a) H(\mathcal{A})\cong k;  \quad \quad (b) H(\mathcal{A})= k[\lceil z\rceil], z\in \mathrm{ker}(\partial_{\mathcal{A}}^1); \\
& (c) H(\mathcal{A})= \frac{k\langle \lceil z_1\rceil, \lceil z_2\rceil\rangle}{(\lceil z_1\rceil\lceil z_2\rceil +\lceil z_2\rceil \lceil z_1\rceil)}, z_1,z_2\in \mathrm{ker}(\partial_{\mathcal{A}}^1).
\end{align*}
Recently, it is proved in \cite[Proposition 6.5]{MHLX} that a connected cochain DG algebra $\mathcal{A}$ is Calabi-Yau if $H(\mathcal{A})=k[\lceil z_1\rceil, \lceil z_2\rceil]$ where $z_1\in \mathrm{ker}(\partial_{\mathcal{A}}^1)$ and $z_2\in \mathrm{ker}(\partial_{\mathcal{A}}^2)$. In this paper, we show
the following proposition (see Proposition \ref{impcri}).
\\
\begin{bfseries}
Proposition \ A.
\end{bfseries}
Let $\mathcal{A}$ be a connected cochain DG algebra
such that $$H(\mathcal{A})=k\langle \lceil y_1\rceil,\lceil y_2\rceil \rangle/(t_1\lceil y_1\rceil^2+t_2\lceil y_2\rceil^2+t_3(\lceil y_1\rceil \lceil y_2\rceil +\lceil y_2\rceil \lceil y_1\rceil))$$ with $y_1,y_2\in Z^1(\mathcal{A})$ and $(t_1,t_2,t_3)\in \Bbb{P}_k^2-\{(t_1,t_2,t_3)|t_1t_2-t_3^2\neq 0\}$. Then
$\mathcal{A}$ is a Koszul and Calabi-Yau DG algebra.

By the proposition above and the computational results of $H[\mathcal{A}_{\mathcal{O}_{-1}(k^3)}(M)]$ in \cite{MR}, we
get the following two propositions (see Propositions \ref{noncycase} and Proposition \ref{nonregsec}).
\\
\begin{bfseries}
Proposition \ B.
\end{bfseries}
The DG algebra $\mathcal{A}_{\mathcal{O}_{-1}(k^3)}(M)$ is  not homologically smooth but Koszul when
 $$M=\left(
                                 \begin{array}{ccc}
                                   m_{11} & m_{12} & m_{13} \\
                                   l_1m_{11} & l_1m_{12} & l_1m_{13} \\
                                   l_2m_{11} & l_2m_{12} & l_2m_{13} \\
                                 \end{array}
                               \right), m_{12}l_1^2+m_{13}l_2^2\neq m_{11}, l_1l_2\neq 0$$ and $4m_{12}m_{13}l_1^2l_2^2= (m_{12}l_1^2+m_{13}l_2^2-m_{11})^2$. In this case,
                               neither $m_{12}m_{11}< 0$ nor $m_{13}m_{11}< 0$ will occur.
 Furthermore,
\begin{enumerate}
\item if $m_{11}=0$, then $m_{12}l_1=m_{13}l_2$ and
$\mathcal{A}_{\mathcal{O}_{-1}(k^3)}(M)$ is isomorphic to $\mathcal{A}_{\mathcal{O}_{-1}(k^3)}(N)$, where $N=\left(
                                 \begin{array}{ccc}
                                   0 & m_{12} & m_{12} \\
                                   0 & l_1m_{12} & l_1m_{12} \\
                                   0 & l_2\sqrt{m_{12}m_{13}} & l_2\sqrt{m_{12}m_{13}} \\
                                 \end{array}
                               \right);$
\item if $m_{11}m_{12}>0, m_{11}m_{13}$ then $\mathcal{A}_{\mathcal{O}_{-1}(k^3)}(M)$  is isomorphic to $\mathcal{A}_{\mathcal{O}_{-1}(k^3)}(Q)$, where $$Q=\left(
                                 \begin{array}{ccc}
                                   m_{11}\sqrt{m_{12}m_{13}} & m_{11}\sqrt{m_{12}m_{13}} & m_{11}\sqrt{m_{12}m_{13}} \\
                                   l_1m_{12}\sqrt{m_{11}m_{13}} & l_1m_{12}\sqrt{m_{11}m_{13}} & l_1m_{12}\sqrt{m_{11}m_{13}} \\
                                   l_2m_{13}\sqrt{m_{11}m_{12}} & l_2m_{13}\sqrt{m_{11}m_{12}} & l_2m_{13}\sqrt{m_{11}m_{12}} \\
                                 \end{array}
                               \right).$$
\end{enumerate}

Beside this, we have the following interesting proposition.\\
\begin{bfseries}
Proposition \ C.
\end{bfseries}
The DG algebra $\mathcal{A}_{\mathcal{O}_{-1}(k^3)}(M)$ is  not homologically smooth but Koszul when
 $$M=\left(
                                 \begin{array}{ccc}
                                   m_{11} & m_{12} & m_{13} \\
                                   l_1m_{11} & l_1m_{12} & l_1m_{13} \\
                                   l_2m_{11} & l_2m_{12} & l_2m_{13} \\
                                 \end{array}
                               \right), \,\, (m_{11},m_{12},m_{13})\neq 0,\,\, l_1l_2\neq 0,$$
 $m_{12}m_{13}=0$ and $m_{12}l_1^2+m_{13}l_2^2= m_{11}$. Furthermore, $\mathcal{A}_{\mathcal{O}_{-1}(k^3)}(M)\cong \mathcal{A}_{\mathcal{O}_{-1}(k^3)}(N)$, where $$N=\left(
                                 \begin{array}{ccc}
                                   1 & 1 & 0 \\
                                   1 & 1 & 0 \\
                                   1 & 1 & 0\\
                                 \end{array}
                               \right).$$

In this paper, we show each $\mathcal{A}_{\mathcal{O}_{-1}(k^3)}(M)$ is Calabi-Yau but those described in Proposition B and Proposition C.
There are cases of corresponding DG algebras whose Calabi-Yau properties one can't judge from their cohomology.
 For such kind of $\mathcal{A}_{\mathcal{O}_{-1}(k^3)}(M)$,
 we construct the minimal semi-free resolution of $k$ in each case and compute the corresponding Ext-algebras. It involves further classifications and complicated matrix analysis. In our proof, we rely heavily on a result proved in \cite{HM} that
a Koszul connected cochain DG algebra $\mathcal{A}$ is Calabi-Yau if and only if its Ext-algebra is a symmetric Frobenius algebra.
Finally, we reach the following conclusion.
\\
\begin{bfseries}
Theorem \ C.
\end{bfseries}
For any $N\in M_3(k)$, the DG algebra $\mathcal{A}_{\mathcal{O}_{-1}(k^3)}(N)$ is Koszul.
It is not Calabi-Yau if and only if there exists some  $C=(c_{ij})_{3\times 3}\in \mathrm{QPL}_3(k)$ satisfying $N=C^{-1}M(c_{ij}^2)_{3\times 3}$,
where
$$M=\left(
                                 \begin{array}{ccc}
                                   1 & 1 & 0 \\
                                   1 & 1 & 0 \\
                                   1 & 1 & 0 \\
                                 \end{array}
                               \right)
\,\,\text{or}\,\,M=\left(
                                 \begin{array}{ccc}
                                   m_{11} & m_{12} & m_{13} \\
                                   l_1m_{11} & l_1m_{12} & l_1m_{13} \\
                                   l_2m_{11} & l_2m_{12} & l_2m_{13} \\
                                 \end{array}
                               \right)$$ with $m_{12}l_1^2+m_{13}l_2^2\neq m_{11}, l_1l_2\neq 0$ and $4m_{12}m_{13}l_1^2l_2^2= (m_{12}l_1^2+m_{13}l_2^2-m_{11})^2$.

A graded $H(\mathcal{A})$-module $X$ is called realizable if there exists some DG $\mathcal{A}$-module $M$ such that $X=H(M)$.
In DG context, the corresponding realizability problem is worthy of deep research. We refer the reader to see this in \cite{BKS,Hub}.
Now, let us consider a similar problem on quasi-isomorphism of DG algebras.
\begin{quest}\label{biproduct}
Let $\mathcal{A}$ and $\mathcal{A}'$ be two connected cochain DG algebra with $\mathcal{A}^{\#}=\mathcal{A}'^{\#}$. Assume that the graded algebras
$H(\mathcal{A})$ and $H(\mathcal{A}')$ are isomorphic to each other.  Can we conclude that $\mathcal{A}$ is quasi-isomorphic to $\mathcal{A}'$ ?
\end{quest}
From the classifications in Section \ref{caseone}, we can see many counter-examples for Question \ref{biproduct} (See Remark \ref{discover}).
This can be consider as a bi-product of our main results.
\section{preliminaries}
 We assume
that the reader is familiar with basic definitions concerning DG
homologically algebra.  If this is not the case, we refer to \cite{AFH, FHT2, MW1,MW2} for more details on them. We begin by fixing some notations and terminology.
There are some overlaps here in \cite{MHLX,MGYC}.

\subsection{Some conventions}For any $k$-vector space $V$, we write $V^*=\Hom_{k}(V,k)$. Let $\{e_i|i\in I\}$ be a basis of a finite dimensional $k$-vector space $V$.  We denote the dual basis of $V$ by $\{e_i^*|i\in I\}$, i.e., $\{e_i^*|i\in I\}$ is a basis of $V^*$ such that $e_i^*(e_j)=\delta_{i,j}$.
For any graded vector space $W$ and $j\in\Bbb{Z}$,  the $j$-th suspension $\Sigma^j W$ of $W$ is a graded vector space defined by $(\Sigma^j W)^i=W^{i+j}$.

\subsection{Notations on DG algebras}
For any cochain DG algebra $\mathcal{A}$,  we denote $\mathcal{A}\!^{op}$ as its opposite DG
algebra, whose multiplication is defined as
 $a \cdot b = (-1)^{|a|\cdot|b|}ba$ for all
graded elements $a$ and $b$ in $\mathcal{A}$.  A cochain DG algebra $\mathcal{A}$ is called
non-trivial if $\partial_{\mathcal{A}}\neq 0$, and $\mathcal{A}$ is said to be connected if its underlying graded algebra $\mathcal{A}^{\#}$ is a connected  graded algebra.

Given a cochain DG algebra $\mathcal{A}$, we denote by $\mathcal{A}^i$ its $i$-th homogeneous component.  The differential $\partial_{\mathcal{A}}$ is a family of linear maps $\partial_{\mathcal{A}}^i: \mathcal{A}^i\to \mathcal{A}^{i+1}$ with $\partial_{\mathcal{A}}^{i+1}\circ \partial_{\mathcal{A}}^i=0$, for all $i\in \Bbb{Z}$. The cohomology graded algebra of $\mathcal{A}$ is the graded algebra $$H(\mathcal{A})=\bigoplus_{i\in \Bbb{Z}}\frac{Z^i(\mathcal{A})}{B^i(\mathcal{A})},$$
where
$Z^i(\mathcal{A})=\mathrm{ker}(\partial_{\mathcal{A}}^i)$ and $B^i(\mathcal{A})=\mathrm{im}(\partial_{\mathcal{A}}^{i-1})$.
For any cocycle element $z\in Z^i(\mathcal{A})$, we write $\lceil z \rceil$ as the cohomology class in $H(\mathcal{A})$ represented by $z$. One sees that $H^i(\mathcal{A})$ is a connected graded algebra if $\mathcal{A}$ is a connected cochain DG algebra.  For the rest of this paper, we write $\mathcal{A}$ for a
connected cochain DG algebra over a field $k$ if no special assumption is
emphasized.
We denote by $\frak{m}_{\mathcal{A}}$ its maximal DG
ideal $$ \cdots\to 0\to \mathcal{A}^1\stackrel{\partial^1_{\mathcal{A}}}{\to}
\mathcal{A}^2\stackrel{\partial^2_{\mathcal{A}}}{\to} \cdots \stackrel{\partial^{n-1}_{\mathcal{A}}}{\to}
\mathcal{A}^n\stackrel{\partial^n_{\mathcal{A}}}{\to}
 \cdots .$$

 A morphism $f:\mathcal{A}\to \mathcal{A}'$ of DG algebras is a chain map of complexes which respects multiplication and unit; $f$ is said to be a DG algebra isomorphism (resp. quasi-isomorphism) if $f$ (resp. $H(f)$) is an isomorphism.  A DG algebra isomorphism $f$ is called a DG automorphism when $\mathcal{A}'=\mathcal{A}$. The set of all DG algebra automorphisms of $\mathcal{A}$ is a group, denoted by $\mathrm{Aut}_{dg}(\mathcal{A})$.
\subsection{Notations on DG modules} A
left DG module over $\mathcal{A}$ (DG $\mathcal{A}$-module for
short) is a complex $(M,\partial_{M})$ together with a left
multiplication $\mathcal{A}\otimes M \to M$ such that $M$ is a left graded
module over $\mathcal{A}$ and the differential $\partial_{M}$ of $M$ satisfies
the Leibniz rule
\begin{align*}
\partial_{M}(am)=\partial_{\mathcal{A}}(a)m + (-1)^{|a|}a\partial_{M}(m)
\end{align*}
for all graded elements $a \in \mathcal{A}, \, m \in M$. A right DG module
over $\mathcal{A}$ is defined similarly. Right DG modules over $\mathcal{A}$ can be
identified with DG $\mathcal{A}\!^{op}$-modules.
 Clearly,  $k$ has a structure of DG $\mathcal{A}$-module via the augmentation map $$\varepsilon: \mathcal{A}\to \mathcal{A}/\frak{m}_{\mathcal{A}}=k.$$
One sees that the enveloping DG algebra $\mathcal{A}^e = \mathcal{A}\otimes \mathcal{A}\!^{op}$ of $\mathcal{A}$
is also a connected cochain DG algebra with $H(\mathcal{A}^e)\cong H(\mathcal{A})^e$, and $$\frak{m}_{\mathcal{A}^e}=\frak{m}_{\mathcal{A}}\otimes \mathcal{A}^{op} + \mathcal{A}\otimes
\frak{m}_{\mathcal{A}^{op}}.$$

A DG $\mathcal{A}$-module is called DG free, if it is isomorphic to a
direct sum of suspensions of $\mathcal{A}$ (note it is not a free object in
the category of DG modules). Let $Y$ be a graded set, we denote
$\mathcal{A}^{(Y)}$ as the DG free DG module $\oplus_{y\in Y}\mathcal{A} e_y$, where
$|e_y|=|y|$ and $\partial(e_y)=0$. Let $M$ be a DG $\mathcal{A}$-module. A
subset $E$ of $M$ is called a \emph{semibasis} if it is a basis of
$M^{\#}$ over $\mathcal{A}^{\#}$ and has a decomposition $E =
\bigsqcup_{i\ge0}E^i$ as a union of disjoint graded subsets $E^i$
such that
\begin{align*}
\partial(E^0)=0 \,\,\, \textrm {and} \,\,\,\partial(E^u)\subseteq \mathcal{A}
(\bigsqcup_{i<u}E^i)\, \,\,\textrm{for all}\,\, \,u >0.
\end{align*}
A DG $\mathcal{A}$-module $M$ is called semifree if there is a
sequence of DG submodules
\begin{align*}
0=M_{-1}\subset M_{0}\subset\cdots\subset M_{n}\subset\cdots
\end{align*}
such that $M = \cup_{n \ge 0} \,M_{n}$ and that each $M_{n}/M_{n-1}=\mathcal{A}^{(Y)}$
is free on a basis $\{e_y|y\in Y\}$ of cocycles. We usually say that $M_n$ is an extension of $M_{n-1}$ since $M_{n-1}$ is a DG submodule of $M_{n-1}$, $M^{\#}_n=M^{\#}_{n-1}\oplus \mathcal{A}^{(Y)}$ and $\partial_{M_{n}}(e_y)\subseteq M_{n-1}$ for any $y\in Y$. Note that we have the following linearly split short exact sequence of DG $\mathcal{A}$-modules $$0\to M_{n-1}\to M_n\to M_n/M_{n-1}\to 0.$$
It is easy to see that a DG
$A$-module is semifree if and only if it has a semibasis. A semifree resolution of a DG $\mathcal{A}$-module $M$ is a
quasi-isomorphism $\varepsilon:F \to M$, where $F$ is a semifree DG
$\mathcal{A}$-module. Sometimes we call $F$ itself a semifree resolution of
$M$.  A semifree resolution $F$ is called minimal if $\partial_F(F)\subseteq F$.
\subsection{Derived categories}
We write $\mathrm{D}(\mathcal{A})$ for
the derived category of left DG modules over $\mathcal{A}$ (DG $\mathcal{A}$-modules for short).  A DG $\mathcal{A}$-module  $M$ is compact if the functor $\Hom_{\mathrm{D}(A)}(M,-)$ preserves
all coproducts in $\mathrm{D}(\mathcal{A})$. It is worth noticing that a DG $\mathcal{A}$-module is compact if and only if it admits a minimal semi-free resolution with a finite semi-basis (see \cite[Proposition 3.3]{MW1}).
The full subcategory of $\mathrm{D}(\mathcal{A})$ consisting of compact DG $\mathcal{A}$-modules is denoted by $\mathrm{D^c}(\mathcal{A})$.

We write
$\mathrm{D^{b}}(\mathcal{A})$ for the full subcategories of $\mathrm{D}(\mathcal{A})$,
whose objects are cohomologically
bounded.  We say a graded vector
space $M = \oplus_{i\in \Bbb{Z}}M^i$ is locally finite, if
each $M^i$ is finite dimensional. The full subcategory of
$\mathrm{D}(A)$ consisting of DG modules with locally finite
cohomology is denoted by $\mathrm{D_{lf}}(A)$.

\subsection{Definitions of some homological properties}
Let $\mathcal{A}$ be a connected cochain DG algebra.
\begin{enumerate}
\item  If $\dim_{k}H(R\Hom_{\mathcal{A}}(k,\mathcal{A}))=1$, then $\mathcal{A}$ is called Gorenstein (cf. \cite{FHT1});
\item  If ${}_{\mathcal{A}}k$, or equivalently ${}_{\mathcal{A}^e}\mathcal{A}$, has a minimal semi-free resolution with a semi-basis concentrated in degree $0$, then $\mathcal{A}$ is called Koszul (cf. \cite{HW});
\item If ${}_{\mathcal{A}}k$, or equivalently the DG $\mathcal{A}^e$-module $\mathcal{A}$ is compact, then $\mathcal{A}$ is called homologically smooth (cf. \cite[Corollary 2.7]{MW3});
\item If $\mathcal{A}$ is homologically smooth and $R\Hom_{\mathcal{A}^e}(\mathcal{A}, \mathcal{A}^e)\cong
\Sigma^{-n}\mathcal{A}$ in  the derived category $\mathrm{D}((\mathcal{A}^e)^{op})$ of right DG $\mathcal{A}^e$-modules, then $\mathcal{A}$ is called an $n$-Calabi-Yau DG algebra  (cf. \cite{Gin,VdB}).
\end{enumerate}

\begin{rem}\label{methods}
Note that a connected cochain DG algebra $\mathcal{A}$ is Koszul if and only if $H(k\otimes_{\mathcal{A}}^Lk)$ is concentrated in degree $0$ (cf. \cite{HM}). And $\mathcal{A}$ is homologically smooth if and only if $\dim_kH(k\otimes_{\mathcal{A}}^Lk)<\infty$ (cf. \cite{MW3}). By \cite[Theorem 4.2]{HM}, $\mathcal{A}$ is a $0$-Calabi-Yau DG algebra if and only if $H(k\otimes_{\mathcal{A}}^Lk)$ is a symmetric co-Frobenius coalgebra concentrated in degree $0$, or equivalently the Ext-algebra $H(R\Hom_{\mathcal{A}}(k,k))$ is a symmetric Frobenius algebra concentrated in degree $0$.
\end{rem}

Now, let us recall the definition of symmetric Frobenius algebra as follows.
\begin{defn}
 Let $\mathcal{E}$ be a finite dimensional graded
$k$-algebra. We call $\mathcal{E}$ a Frobenius algebra, if there is an
isomorphism of left $\mathcal{E}$-modules: $\Sigma^{u}\mathcal{E}\to \mathcal{E}^*$, or
equivalently there is an isomorphism of right $\mathcal{E}$-modules:
$\Sigma^{u}\mathcal{E}\to \mathcal{E}^*$, where $u=\sup\{i|\mathcal{E}^i\neq 0\}$. A Frobenius algebra $\mathcal{E}$ is called symmetric if
$\Sigma^{u}\mathcal{E}\cong \mathcal{E}^*$ as graded $\mathcal{E}$-bimodules.
\end{defn}

\begin{defn}
Let $\mathcal{C}$ be a finite dimensional graded
coalgebra over a field $k$ and let $\mathcal{C}^*$ be the dual graded algebra.
The graded coalgebra $\mathcal{C}$ is called right (resp. left) co-Frobenius
if there is a monomorphism of right (resp. left) graded $\mathcal{C}^*$-module
from $\Sigma^{u}\mathcal{C}$ into $\mathcal{C}^*$, where $u=\sup\{i|\mathcal{C}^i\neq 0\}$. If $\mathcal{C}$
is both left and right co-Frobenius, then we say $\mathcal{C}$ is
co-Frobenius.
\end{defn}

\begin{rem} Assume that $\mathcal{C}$ is a finite
dimensional graded coalgebra. By \cite[Remark 3.3.12]{NNR}, $\mathcal{C}$ is left co-Frobenius if and
only if the dual algebra $\mathcal{C}^*$ is Frobenius, which is also
equivalent to the condition that $\mathcal{C}$ is right co-Frobenius. We say that $\mathcal{C}$ is symmetric if $\mathcal{C}^*$ is a symmetric Frobenius algebra.
\end{rem}

\section {dg algebra structures}\label{second}
In this section, we study all possible differential structures of a connected
cochain DG algebra, whose underlying graded algebra is the
  quantum affine $n$-space $\mathcal{O}_{-1}(k^n)$.
\begin{thm}\label{diffstr}
 Let $\mathcal{A}$ be a connected cochain DG algebra such that $\mathcal{A}^{\#}$ is the $k$-algebra with degree one generators $x_1,\cdots, x_n$ and relations $x_ix_j=-x_jx_i$, for all $1\le i<j\le n$.  Then  $\partial_{\mathcal{A}}$ is determined by a matrix $M=(m_{ij})_{n\times n}$ such that\\
$$
\left(
                         \begin{array}{c}
                           \partial_{\mathcal{A}}(x_1)\\
                           \partial_{\mathcal{A}}(x_2)\\
                           \vdots  \\
                           \partial_{\mathcal{A}}(x_n)
                         \end{array}
                       \right)=M\left(
                         \begin{array}{c}
                           x_1^2\\
                           x_2^2\\
                           \vdots \\
                           x_n^2
                         \end{array}
                       \right).
$$
\end{thm}
\begin{proof}
Since the differential $\partial_{\mathcal{A}}$ of $\mathcal{A}$ is a $k$-linear map of degree $1$, we may let
$\partial_{\mathcal{A}}(x_i) = \sum\limits_{s=1}^n\sum\limits_{t=s}^nm^i_{s,t}x_sx_t$,
where
$m^i_{s,t}\in k$, for any $1\le s\le t\le n$ and $1\le i\le  n$. We have $0=x_ix_j+x_jx_i\in \mathcal{A}$, for any $1\le i<j\le n$.
Hence
\begin{align*}
0&=\partial_{\mathcal{A}}(x_ix_j+x_jx_i)= \partial_{\mathcal{A}}(x_i)x_j-x_i\partial_{\mathcal{A}}(x_j)+\partial_{\mathcal{A}}(x_j)x_i-x_j\partial_{\mathcal{A}}(x_i)         \\
&=\partial_{\mathcal{A}}(x_i)x_j-x_j\partial_{\mathcal{A}}(x_i) +\partial_{\mathcal{A}}(x_j)x_i-x_i\partial_{\mathcal{A}}(x_j) \\
&=[\sum\limits_{s=1}^n\sum\limits_{t=s}^nm^i_{s,t}x_sx_t]x_j-x_j[\sum\limits_{s=1}^n\sum\limits_{t=s}^nm^i_{s,t}x_sx_t] \\
&+[\sum\limits_{s=1}^n\sum\limits_{t=s}^nm^j_{s,t}x_sx_t]x_i-x_i[\sum\limits_{s=1}^n\sum\limits_{t=s}^nm^j_{s,t}x_sx_t] \\
&\stackrel{(1)}{=}[\sum\limits_{s=1}^n\sum\limits_{t=s+1}^nm^i_{s,t}x_sx_t]x_j-x_j[\sum\limits_{s=1}^n\sum\limits_{t=s+1}^nm^i_{s,t}x_sx_t]\\
&+[\sum\limits_{s=1}^n\sum\limits_{t=s+1}^nm^j_{s,t}x_sx_t]x_i-x_i[\sum\limits_{s=1}^n\sum\limits_{t=s+1}^nm^j_{s,t}x_sx_t]\\
&\stackrel{(2)}{=} [\sum\limits_{s=1}^{j-1}\sum\limits_{t=s+1}^nm^i_{s,t}x_sx_t]x_j+\sum\limits_{t=j+1}^nm^i_{j,t}x_jx_tx_j+[\sum\limits_{s=j+1}^n\sum\limits_{t=s+1}^nm^i_{s,t}x_sx_t]x_j\\
&-x_j[\sum\limits_{s=1}^{j-1}\sum\limits_{t=s+1}^nm^i_{s,t}x_sx_t]-\sum\limits_{t=j+1}^nm^i_{j,t}x_j^2x_t-x_j[\sum\limits_{s=j+1}^{n}\sum\limits_{t=s+1}^nm^i_{s,t}x_sx_t]\\
&+[\sum\limits_{s=1}^{i-1}\sum\limits_{t=s+1}^nm^j_{s,t}x_sx_t]x_i+\sum\limits_{t=i+1}^nm^j_{i,t}x_ix_tx_i+[\sum\limits_{s=i+1}^{n}\sum\limits_{t=s+1}^nm^j_{s,t}x_sx_t]x_i \\
&-x_i[\sum\limits_{s=1}^{i-1}\sum\limits_{t=s+1}^nm^j_{s,t}x_sx_t]-\sum\limits_{t=i+1}^nm^j_{i,t}x_i^2x_t-x_i[\sum\limits_{s=i+1}^{n}\sum\limits_{t=s+1}^nm^j_{s,t}x_sx_t]\\
&\stackrel{(3)}{=} [\sum\limits_{s=1}^{j-1}\sum\limits_{t=s+1}^nm^i_{s,t}x_sx_t]x_j+[\sum\limits_{s=j+1}^n\sum\limits_{t=s+1}^nm^i_{s,t}x_sx_t]x_j\\
&-x_j[\sum\limits_{s=1}^{j-1}\sum\limits_{t=s+1}^nm^i_{s,t}x_sx_t]-x_j[\sum\limits_{s=j+1}^{n}\sum\limits_{t=s+1}^nm^i_{s,t}x_sx_t]-2\sum\limits_{t=j+1}^nm^i_{j,t}x_j^2x_t\\
&+[\sum\limits_{s=1}^{i-1}\sum\limits_{t=s+1}^nm^j_{s,t}x_sx_t]x_i+[\sum\limits_{s=i+1}^{n}\sum\limits_{t=s+1}^nm^j_{s,t}x_sx_t]x_i \\
&-x_i[\sum\limits_{s=1}^{i-1}\sum\limits_{t=s+1}^nm^j_{s,t}x_sx_t]-x_i[\sum\limits_{s=i+1}^{n}\sum\limits_{t=s+1}^nm^j_{s,t}x_sx_t]-2\sum\limits_{t=i+1}^nm^j_{i,t}x_i^2x_t\\
&\stackrel{(4)}{=}[\sum\limits_{s=1\atop s\neq j}^n\sum\limits_{t=s+1\atop t\neq j}^nm^i_{s,t}x_sx_t]x_j+\sum\limits_{s=1}^{j-1}m^i_{s,j}x_sx_j^2 -x_j[\sum\limits_{s=1\atop s\neq j}^n\sum\limits_{t=s+1\atop t\neq j}^nm^i_{s,t}x_sx_t]\\
&-x_j\sum\limits_{s=1}^{j-1}m^i_{s,j}x_sx_j -2\sum\limits_{t=j+1}^nm^i_{j,t}x_j^2x_t + [\sum\limits_{s=1\atop s\neq i}^n\sum\limits_{t=s+1\atop t\neq i}^nm^j_{s,t}x_sx_t]x_i +\sum\limits_{s=1}^{i-1}m^j_{s,i}x_sx_i^2 \\                                            &-x_i[\sum\limits_{s=1\atop s\neq i}^n\sum\limits_{t=s+1\atop t\neq i}^nm^j_{s,t}x_sx_t]-\sum\limits_{s=1}^{i-1}m_{s,i}^jx_ix_sx_i-2\sum\limits_{t=i+1}^nm^j_{i,t}x_i^2x_t\\
&\stackrel{(5)}{=}2\sum\limits_{s=1}^{j-1}m^i_{s,j}x_sx_j^2-2\sum\limits_{t=j+1}^nm^i_{j,t}x_j^2x_t+2\sum\limits_{s=1}^{i-1}m^j_{s,i}x_sx_i^2 -2\sum\limits_{t=i+1}^nm^j_{i,t}x_i^2x_t,
\end{align*}
where the labeled equations are obtained respectively by the following facts
\begin{enumerate}
\item $x_s^2x_j=-x_s x_j x_s=x_j x_s^2$, $x_s^2x_i=-x_s x_i x_s=x_ix_s^2$; \\
\item $\sum\limits_{s=1}^n\sum\limits_{t=s+1}^n=\sum\limits_{s=1}^{j-1}\sum\limits_{t=s+1}^n+ \sum\limits_{s=j}\sum\limits_{t=j+1}^n+\sum\limits_{s=j+1}^n\sum\limits_{t=s+1}^n$,\\
    $\sum\limits_{s=1}^n\sum\limits_{t=s+1}^n=\sum\limits_{s=1}^{i-1}\sum\limits_{t=s+1}^n+ \sum\limits_{s=i}\sum\limits_{t=i+1}^n+\sum\limits_{s=i+1}^n\sum\limits_{t=s+1}^n$;\\
\item $x_jx_tx_j=-x_j^2x_t$, $x_ix_tx_i=-x_i^2x_t$;\\
\item $\sum\limits_{s=1}^{j-1}\sum\limits_{t=s+1}^n +\sum\limits_{s=j+1}^n\sum\limits_{t=s+1}^n=\sum\limits_{s=1\atop s\neq j}^n\sum\limits_{t=s+1\atop t\neq j}^n+\sum\limits_{s=1}^{j-1}\sum\limits_{t=j}$,\\
    $\sum\limits_{s=1}^{i-1}\sum\limits_{t=s+1}^n+\sum\limits_{s=i+1}^{n}\sum\limits_{t=s+1}^n=\sum\limits_{s=1\atop s\neq i}^n\sum\limits_{t=s+1\atop t\neq i}^n+\sum\limits_{s=1}^{i-1}\sum\limits_{t=i}$;\\
\item $x_sx_tx_j=x_jx_sx_t$, when $s,t,j$ are pairwise different.
\end{enumerate}
Since $i\neq j$, the elements \begin{align*}
&x_sx_j^2, s=1,\cdots, j-1,\\
&x_j^2x_t, t=j+1,\cdots n,\\
&x_sx_i^2, s=1,\cdots i-1,\\
&x_i^2x_t, t=i+1,\cdots, n
\end{align*}
in $\mathcal{A}^3$ are linearly independent. Therefore,
\begin{align*}
\begin{cases}
m^i_{s,j}=0, \text{for all}\,\,s\in \{1,2,\cdots, j-1\}\\
m^i_{j,t}=0, \text{for all}\,\,t\in \{j+1,j+2,\cdots, n\} \\
m^j_{s,i}=0, \text{for all}\,\, s\in \{1,2,\cdots, i-1\}\\
m^j_{i,t}=0, \text{for all}\,\, t\in \{i+1,i+2, \cdots, n\}.
\end{cases}
\end{align*}
Hence $\partial_{\mathcal{A}}(x_i)=\sum\limits_{s=1}^nm^i_{s,s}x_s^2$. One sees that
\begin{align*}
\partial_{A}^2(x_i)&=\partial_{\mathcal{A}}(\sum\limits_{s=1}^nm^i_{s,s}x_s^2) \\
                   &=\sum\limits_{s=1}^nm^i_{s,s}[\partial_{A}(x_i)x_i-x_i\partial_{A}(x_i)]\\
                   &=\sum\limits_{s=1}^nm^i_{s,s}[\sum\limits_{t=1}^nm^i_{t,t}x_t^2x_i-x_i\sum\limits_{r=1}^nm^i_{r,r}x_r^2]\\
                   &=\sum\limits_{s=1}^nm^i_{s,s}\sum\limits_{q=1}^nm^i_{q,q}(x_q^2x_i-x_ix_q^2)=0,
\end{align*}
for any $i\in \Bbb{Z}$.
Therefore, $\partial_{\mathcal{A}}$ is uniquely determined by the $n\times n$ matrix $M=(m_{ij})$, where $m_{ij}=m^i_{j,j}, i,j\in \{1,2,\cdots, n\}$.
\end{proof}

By Theorem \ref{diffstr}, one sees that the following definition is reasonable.
\begin{defn}
For any $M=(m_{ij})\in M_n(k)$, we define a connected cochain DG algebra $\mathcal{A}_{\mathcal{O}_{-1}(k^n)}(M)$ such that $$[\mathcal{A}_{\mathcal{O}_{-1}(k^n)}(M)]^{\#}=\mathcal{O}_{-1}(k^n)$$  and its differential $\partial_{\mathcal{A}}$ is defined by
\begin{align*}
\left(
                         \begin{array}{c}
                           \partial_{\mathcal{A}}(x_1)\\
                           \partial_{\mathcal{A}}(x_2)\\
                           \vdots  \\
                           \partial_{\mathcal{A}}(x_n)
                         \end{array}
                       \right)=M\left(
                         \begin{array}{c}
                           x_1^2\\
                           x_2^2\\
                           \vdots \\
                           x_n^2
                         \end{array}
                       \right).
\end{align*}
\end{defn}
\begin{lem}\label{centcocy}
For any $M=(m_{ij})\in M_n(k)$ and $t\in \Bbb{N}$, each $x_i^{2t}$ is a cocycle central element of  $\mathcal{A}_{\mathcal{O}_{-1}(k^n)}(M)$.
\end{lem}
\begin{proof}
For the sake of convenience, we let $\mathcal{A}=\mathcal{A}_{\mathcal{O}_{-1}(k^n)}(M)$.
Since $$x_i^2x_j=x_ix_ix_j=-x_ix_jx_i=x_jx_i^2$$ when $i\neq j$, one sees that $x_i^2$ is a central element of $\mathcal{A}$. This implies that each
$x_i^{2t}$ is a central element of  $\mathcal{A}$.
By Proposition \ref{diffstr}, we have \begin{align*}
\partial_{\mathcal{A}}(x_i^2)&=\partial_{\mathcal{A}}(x_i)x_i-x_i\partial_{\mathcal{A}}(x_i)\\
                             &=\sum\limits_{j=1}^nm_{ij}x_j^2x_i-x_i\sum\limits_{j=1}^nm_{ij}x_j^2\\
                             &=\sum\limits_{j=1}^nm_{ij}(x_j^2x_i-x_ix_j^2)=0.
\end{align*}
Using this, we can inductively prove $\partial_{\mathcal{A}}(x_i^{2t})=0$.
\end{proof}

By Lemma \ref{centcocy}, one sees that the graded ideal $I=(x_1^2,x_2^2,\cdots, x_n^2)$ is a DG ideal of $\mathcal{A}_{\mathcal{O}_{-1}(k^n)}(M)$, and the quotient DG algebra $$\mathcal{A}_{\mathcal{O}_{-1}(k^n)}(M)/I=\bigwedge(x_1,x_2,\cdots, x_n)$$ is the exterior algebra with zero differential. We have the following short exact sequence
$$0\to I \hookrightarrow \mathcal{A}_{\mathcal{O}_{-1}(k^n)}(M) \to \bigwedge(x_1,x_2,\cdots, x_n) \to 0.$$
To some extent,  the DG algebras $\mathcal{A}_{\mathcal{O}_{-1}(k^n)}(M)$ can be considered as an intermediate transition family of DG algebras between the  graded commutative DG algebra $\bigwedge(x_1,x_2,\cdots, x_n)$ and the DG free algebras studied in \cite{MXYA}.

\section{Isomorphism problem}
It is well known in linear algebra that a square matrix is called a permutation matrix if its each row and each column  have only one non-zero element $1$.
In \cite{AJL}, a more general notion is introduced. This is the following definition.
\begin{defn}\label{quasi-perm}
A square matrix is called a quasi-permutation matrix if each row and each column has at most one non-zero element.
\end{defn}
\begin{rem}
By the definition above, a quasi-permutation matrix can be singular. Furthermore, a quasi-permutation matrix is non-singular if and only if each row and each column  has exactly one non-zero element.
\end{rem}

\begin{lem}\label{onlyone}
Let $M=(m_{ij})_{n\times n}$ be a matrix in $\mathrm{GL}_n(k)$.  Then each row and each column of $M$ has only one non-zero element, or equivalently $M$ is a quasi-permutation matrix,  if and only if $m_{ir}m_{jr}=0$, for any  $1\le i<j\le n$ and $r\in \{1,2,\cdots, n\}$.
\end{lem}
\begin{proof}
Obviously, we only need to show the 'if' part.
Since $M\in \mathrm{GL}_n(k)$, each column of $M$ has at least one non-zero element. If the $r$-th column of $M$ has two non-zero elements $m_{i_1r}$ and $m_{i_2r}$, then $m_{i_1r}m_{i_2r}\neq 0$, which contradicts with the assumption. Thus each column of $M$ has only one non-zero element.
Then we conclude that $M$ has $n$ non-zero elements. By the non-singularity of $M$,  we show that each row of $M$ has exactly one non-zero element.
\end{proof}

\begin{prop}\label{subgroup}
The set of quasi-permutation matrixes in $\mathrm{GL}_n(k)$ is a subgroup of the general linear group $\mathrm{GL}_n(k)$.
\end{prop}
\begin{proof}
For any quasi-permutation matrixes $B=(b_{ij})_{n\times n}$ and $D=(d_{ij})_{n\times n}$ in $\mathrm{GL}_n(k)$, there exist
$\sigma, \tau\in \Bbb{S}_{n}$ such that
\begin{align*}
b_{i\sigma(i)}\neq 0, b_{ij}=0,\,\, \text{if}\,\,j\neq \sigma(i) \\
d_{i\tau(i)}\neq 0, d_{ij}=0, \,\,  \text{if}\,\,j\neq \tau(i),
\end{align*}
for any $i\in \{1,2,\cdots,n\}$.
One sees that $B$ and $D$ can be written by
\begin{align*}
B=(E_rE_{r-1}\cdots E_1)\mathrm{diag}(b_{1\sigma(1)},b_{2\sigma(2)},\cdots, b_{n\sigma(n)}),\\
D=(E_s'E_{s-1}'\cdots E_1')\mathrm{diag}(d_{1\tau(1)},d_{2\tau(2)}\cdots, d_{n\tau(n)}),
\end{align*}
where $E_i$ and $E'_j$ are elementary matrixes obtained by swapping two rows of the identity matrix $I_n$.
Hence \begin{align*}
BD^{-1}&=(E_r\cdots E_1)\mathrm{diag}(b_{1\sigma(1)},\cdots, b_{n\sigma(n)})\mathrm{diag}(\frac{1}{d_{1\tau(1)}},\cdots, \frac{1}{d_{n\tau(n)}})(E_1'\cdots E_s')\\
&=(E_r\cdots E_1)\mathrm{diag}(\frac{b_{1\sigma(1)}}{d_{1\tau(1)}},\cdots, \frac{b_{n\sigma(n)}}{d_{n\tau(n)}})(E_1'\cdots E_s').
\end{align*}
So $BD^{-1}$ is obtained from the diagonal matrix $\mathrm{diag}(\frac{b_{1\sigma(1)}}{d_{1\tau(1)}},\cdots, \frac{b_{n\sigma(n)}}{d_{n\tau(n)}})$ by swapping two rows or two columns several times. Then  each row and each column of $BD^{-1}$ has only one non-zero element. It implies that the set of quasi-permutation matrixes in $\mathrm{GL}_n(k)$ is closed under multiplication and taking the inverse, hence it is indeed a subgroup of $\mathrm{GL}_n(k)$.
\end{proof}
By Proposition \ref{subgroup}, we can introduce the following definition.
\begin{defn}\label{qpl}
We use $\mathrm{QPL}_n(k)$ to denote the set of non-singular quasi-permutation $n\times n$ matrixes. By Proposition \ref{subgroup}, $\mathrm{QPL}_n(k)$ is a subgroup of $\mathrm{GL}_n(k)$.
\end{defn}

\begin{thm}\label{iso}
Let $M$ and $M'$ be two matrixes in $M_n(k)$. Then $$\mathcal{A}_{\mathcal{O}_{-1}(k^n)}(M)\cong \mathcal{A}_{\mathcal{O}_{-1}(k^n)}(M')$$ if and only if there exists $C=(c_{ij})_{n\times n}\in \mathrm{QPL}_n(k)$ such that
$$M'=C^{-1}M(c_{ij}^2)_{n\times n}.$$
\end{thm}
\begin{proof}
We write
$\mathcal{A}=\mathcal{A}_{\mathcal{O}_{-1}(k^n)}(M)$ and
$\mathcal{A}'= \mathcal{A}_{\mathcal{O}_{-1}(k^n)}(M')$ for
simplicity. In order to distinguish, we assume that $\mathcal{A}'^{\#}$ is the $k$-algebra with degree one generators $x_1',\cdots, x_n'$ and relations $x_i'x_j'=-x_j'x_i'$ for all $1\le i<j\le n$.

If the DG
algebras $\mathcal{A}\cong \mathcal{A}'$, then there exists an
isomorphism $f: \mathcal{A}\to \mathcal{A}'$ of DG algebras. Since
$f^1: \mathcal{A}^1 \to \mathcal{A}'^1$ is a $k$-linear
isomorphism, we may let $$\left(\begin{array}{c}
                           f(x_1) \\
                           f(x_2) \\
                           \vdots \\
                           f(x_n)
                         \end{array}
                       \right)= C\left(
                         \begin{array}{c}
                           x_1' \\
                           x_2' \\
                           \vdots \\
                           x_n'
                         \end{array}
                       \right)$$
for some  $C=(c_{ij})_{n\times n} \in \mathrm{GL}_n(k)$.
We have
\begin{align*}
0&=f(x_ix_j+x_jx_i) \\
 &=f(x_i)f(x_j)+f(x_j)f(x_i)\\
 &=(\sum\limits_{s=1}^nc_{is}x_s')(\sum\limits_{t=1}^nc_{jt}x_t')+(\sum\limits_{t=1}^nc_{jt}x_t')(\sum\limits_{s=1}^nc_{is}x_s')\\
 &=2\sum\limits_{r=1}^nc_{ir}c_{jr}(x_r')^2,
\end{align*}
for any $1\le i<j\le n$. Since $\mathrm{char}k\neq 2$, we get $c_{ir}c_{jr}=0$ for any  $1\le i<j\le n$ and $r\in \{1,2,\cdots, n\}$.
By Lemma \ref{onlyone}, each row and each column of $C$ have only one non-zero element. Hence $C$ is a quasi-permutation non-singular matrix.
Since $f$ is a chain map, we have $f\circ \partial_{\mathcal{A}}=\partial_{\mathcal{A}'}\circ f$.
For any $i\in \{1,2,\cdots,n\}$, we have
\begin{align*}\label{Eq1}\tag{Eq1}
\partial_{\mathcal{A}'}\circ f(x_i)&=\partial_{\mathcal{A}'}(\sum\limits_{j=1}^nc_{ij}x_j')\\
&=\sum\limits_{j=1}^nc_{ij}(\sum\limits_{l=1}^nm'_{jl}(x_l')^2)\\
&=\sum\limits_{l=1}^n[\sum\limits_{j=1}^nc_{ij}m'_{jl}](x_l')^2
\end{align*}
and
\begin{align*}\label{Eq2}\tag{Eq2}
 f\circ \partial_{\mathcal{A}}(x_i)&=f(\sum\limits_{j=1}^{n}m_{ij}(x_j)^2)\\
&=\sum\limits_{j=1}^{n}m_{ij}[f(x_j)]^2\\
&=\sum\limits_{j=1}^{n}m_{ij}[\sum\limits_{l=1}^nc_{jl}x_l']^2\\
&=\sum\limits_{j=1}^{n}m_{ij}\sum\limits_{l=1}^n(c_{jl})^2(x_l')^2\\
&=\sum\limits_{l=1}^n[\sum\limits_{j=1}^{n}m_{ij}(c_{jl})^2](x_l')^2.
\end{align*}
Hence $\sum\limits_{j=1}^nc_{ij}m'_{jl}=\sum\limits_{j=1}^{n}m_{ij}(c_{jl})^2$ for any $i, l\in \{1,2,\cdots, n\}$.
Then we get $$
CM'=M((c_{ij})^2)_{n\times n}.$$
Since $C\in\mathrm{GL}_n(k)$, we have $M'=C^{-1}M((c_{ij})^2)_{n\times n}$.

 Conversely, if there exists a quasi-permutation matrix $C=(c_{ij})_{n\times n}\in \mathrm{GL}_n(k)$ such that $$M'=C^{-1}M((c_{ij})^2)_{n\times n}.$$ Then we have $$
CM'=M((c_{ij})^2)_{n\times n},$$ which implies that $\sum\limits_{j=1}^nc_{ij}m'_{jl}=\sum\limits_{j=1}^{n}m_{ij}(c_{jl})^2$ for any $i, l\in \{1,2,\cdots, n\}$. Define a linear map $f: \mathcal{A}^1 \to \mathcal{A}'^1$
                        by
 $$\left(\begin{array}{c}
                           f(x_1) \\
                           f(x_2) \\
                           \vdots \\
                           f(x_n)
                         \end{array}
                       \right)= C\left(
                         \begin{array}{c}
                           x_1' \\
                           x_2' \\
                           \vdots \\
                           x_n'
                         \end{array}
                       \right).$$
 Obviously, $f$ is invertible since $C\in \mathrm{GL}_n(k)$. Since $C$ is a quasi-permutation matrix, we have
 \begin{align*}
f(x_i)f(x_j)+f(x_j)f(x_i) &=(\sum\limits_{s=1}^nc_{is}x_s')(\sum\limits_{t=1}^nc_{jt}x_t')+(\sum\limits_{t=1}^nc_{jt}x_t')(\sum\limits_{s=1}^nc_{is}x_s')\\
 &=2\sum\limits_{r=1}^nc_{ir}c_{jr}(x_r')^2=0,
\end{align*}
for any $1\le i<j\le n$. Hence $f: \mathcal{A}^1\to \mathcal{A}'^1$ can be extended to a morphism of graded algebras between $\mathcal{A}^{\#}$ and $\mathcal{A}'^{\#}$. We still denote it by $f$. For any $i\in \{1,2,\cdots,n\}$, we still have (\ref{Eq1}) and (\ref{Eq2}). Since $$CM'=M(c_{ij}^2)_{n\times n},$$ we have  $\sum\limits_{j=1}^nc_{ij}m'_{jl}=\sum\limits_{j=1}^{n}m_{ij}(c_{jl})^2$ for any $i, l\in \{1,2,\cdots, n\}$. This implies $$\partial_{\mathcal{A}'}\circ f(x_i)=f\circ \partial_{\mathcal{A}}(x_i),$$ for any $i\in \{1,2,\cdots, n\}$. Hence, $f:A\to A'$ is an isomorphism of DG algebras.
\end{proof}

\begin{cor}\label{autgroup}
For any $M\in M_n(k)$, we have
$$\mathrm{Aut}_{dg}\mathcal{A}_{\mathcal{O}_{-1}(k^n)}(M)=\{C=(c_{ij})_{n\times n}\in \mathrm{QPL}_n(k)|M=C^{-1}M(c_{ij}^2)_{n\times n}\}. $$
\end{cor}
\begin{proof}
This is immediate from Theorem \ref{iso}.
\end{proof}

\begin{defn}\label{leftaction}{\rm  Theorem \ref{iso} indicates that  we can define a map $$\chi: M_n(k) \times \mathrm{QPL}_n(k)\to M_n(k)$$
       such that
$\chi[(M, C=(c_{ij})_{n\times n})]=C^{-1}M((c_{ij})^2)_{n\times n}$.}
\end{defn}

\begin{prop}
The map $\chi$ defined in Definition \ref{leftaction} is a right
group action of $\mathrm{QPL}_n(k)$ on $M_n(k)$.
\end{prop}
\begin{proof}
Obviously, $I_n$ is the identity element in $\mathrm{QPL}_n(k)$. For any
$M$ in $M_n(k)$, we have  $\chi[(M, I_n)]=I^{-1}_nMI_n=M$.
For any $C=(c_{ij})_{n \times n}$ and  $C'=(c'_{ij})_{n \times n}$ in $\mathrm{QPL}_n(k)$,
we have
\begin{align*}
\chi\{(\chi[(M, C)], C')\}&=\chi[(C^{-1}M((c_{ij})^2)_{n\times n}, C')]\\
                          &=(C')^{-1}C^{-1}M((c_{ij})^2)_{n\times n}((c'_{ij})^2)_{n\times n}
\end{align*}
and
$$\chi[(M, CC')]=(C')^{-1}C^{-1}M((\sum\limits_{l=1}^nc_{il}c'_{lj})^2)_{n\times n}.$$
It remains to show that $((c_{ij})^2)_{n\times n}((c'_{ij})^2)_{n\times n}=((\sum\limits_{l=1}^nc_{il}c'_{lj})^2)_{n\times n}$.
Since $C$, $C'$ $\in$ $\mathrm{QPL}_n(k)$, there exist
$\sigma, \tau\in \Bbb{S}_{n}$ such that
\begin{align*}
c_{i\sigma(i)}\neq 0, c_{ij}=0,\,\, \text{if}\,\,j\neq \sigma(i) \\
c_{i\tau(i)}'\neq 0, c_{ij}'=0, \,\,  \text{if}\,\,j\neq \tau(i),
\end{align*}
for any $i\in \{1,2,\cdots,n\}$.
One sees that $C$ and $C'$ can be written by
\begin{align*}
C=(E_rE_{r-1}\cdots E_1)\mathrm{diag}(c_{1\sigma(1)},c_{2\sigma(2)},\cdots, c_{n\sigma(n)}),\\
C'=\mathrm{diag}(c_{1\tau(1)}',c_{2\tau(2)}'\cdots, c_{n\tau(n)}')(E_1'E_2'\cdots E_s'),
\end{align*}
where $E_i$ and $E'_j$ are elementary matrixes obtained by swapping two rows of the identity matrix $I_n$.
 Then
\begin{align*}
CC'=E_sE_{s-1}\cdots E_1\mathrm{diag}(c_{1\sigma(1)}c_{1\tau(1)}',c_{2\sigma(2)}c_{2\tau(2)}',\cdots, c_{n\sigma(n)}c_{n\tau(n)}')E'_1E'_2\cdots E'_t,
\end{align*}
and hence
\begin{align*}
&((\sum\limits_{l=1}^nc_{il}c_{lj}')^2)_{n\times n}=E_r\cdots E_1\mathrm{diag}((c_{1\sigma(1)})^2(c_{1\tau(1)}')^2,\cdots, (c_{n\sigma(n)})^2(c_{n\tau(n)}')^2)E'_1\cdots E'_s\\
&=E_r\cdots E_1\mathrm{diag}((c_{1\sigma(1)})^2,\cdots, (c_{n\sigma(n)})^2)\mathrm{diag}((c_{1\tau(1)}')^2,\cdots,(c_{n\tau(n)}')^2)E'_1\cdots E'_s\\
&=((c_{ij})^2)_{n\times n}((c'_{ij})^2)_{n\times n},
\end{align*}
which implies $\chi[(M, CC')]=\chi\{(\chi[(M, C)], C')\}$.
Therefore,  the map $\chi$ defined in Definition \ref{leftaction} is a right group action of $\mathrm{QPL}_n(k)$ on $M_n(k)$.
\end{proof}

\begin{cor}\label{orbit}{\rm
In $M_n(k)$, there is a natural equivalence relation $\sim_R$
defined by
$$M\sim_R M' \Leftrightarrow \exists\, C=(c_{ij})_{n\times n}\in \mathrm{QPL}_n(k) \quad\text{such that}\quad M'=\chi(M,C).$$ Hence the set of isomorphism classes of DG algebras in
$\{\mathcal{A}_{\mathcal{O}_{-1}(k^n)}(M)|M\in M_{n}(k)\}$
is the quotient set
$$\{\mathcal{A}_{\mathcal{O}_{-1}(k^n)}(M)|M\in M_{n}(k)\}/\mathrm{QPL}_n(k).$$ }
\end{cor}

\section{some useful lemmas}
The cohomology graded algebra $H(\mathcal{A})$ of a cochain DG algebra $\mathcal{A}$ usually contains a lot of homological informations.
In some cases, it is  possible to detect the  Calabi-Yau properties of $\mathcal{A}$ from  $H(\mathcal{A})$. For example,
It is proved in \cite{MYY}, that $\mathcal{A}$ is a Calabi-Yau DG algebra if the trivial DG algebra $(H(\mathcal{A}),0)$ is Calabi-Yau.
And we have the following lemma.
\begin{lem}\cite[Theorem A]{MH}\label{MH}
Let $\mathcal{A}$ be a connected cochain DG algebra. Then $\mathcal{A}$ is a Koszul Calabi-Yau DG algebra if $H(\mathcal{A})$ belongs to one of the following cases:
\begin{align*}
& (a) H(A)\cong k;  \quad \quad (b) H(A)= k[\lceil z\rceil], z\in \mathrm{ker}(\partial_{A}^1); \\
& (c) H(A)= \frac{k\langle \lceil z_1\rceil, \lceil z_2\rceil\rangle}{(\lceil z_1\rceil\lceil z_2\rceil +\lceil z_2\rceil \lceil z_1\rceil)}, z_1,z_2\in \mathrm{ker}(\partial_{A}^1).
\end{align*}
\end{lem}
In the rest of section, we will give another useful criterion to detect the Calabi-Yau properties of $\mathcal{A}$. For this, we need  the following lemma.

\begin{lem}\label{symmetry}
Let $\mathscr{E}_{\lambda,\mu,\nu}$ be the parameterized commutative algebra
$$\mathscr{E}_{\lambda,\mu,\nu}=\Bigg\{\left(
      \begin{array}{cccc}
        a & 0 & 0 & 0 \\
        b & a & 0 & 0 \\
        c & 0 & a & 0 \\
        d & \lambda b+\nu c & \nu b+\mu c & a \\
      \end{array}
    \right)| a,b,c,d\in k\Bigg\}$$
 under the usual multiplication of matrices.
 Then  $\mathscr{E}_{\lambda,\mu,\nu}$ is a symmetric Frobenius $k$-algebra if and only if $\lambda\mu-\nu^2\neq 0$.
\end{lem}

\begin{proof}
We claim that $\mathscr{E}_{\lambda,\mu,\nu}$ is closed  under the usual multiplication of matrices and the multiplication in $\mathscr{E}_{\lambda,\mu,\nu}$ is commutative.
Indeed, it is straightforward to check
\begin{align*}
& \left(
      \begin{array}{cccc}
        a & 0 & 0 & 0 \\
        b & a & 0 & 0 \\
        c & 0 & a & 0 \\
        d & \lambda b+\nu c & \nu b+\mu c & a \\
      \end{array}
    \right)\left(
      \begin{array}{cccc}
        a' & 0 & 0 & 0 \\
        b' & a' & 0 & 0 \\
        c' & 0 & a' & 0 \\
        d' & \lambda b'+\nu c' & \nu b'+\mu c' & a' \\
      \end{array}
    \right)\\
    =&
\left(
      \begin{array}{cccc}
        a' & 0 & 0 & 0 \\
        b' & a' & 0 & 0 \\
        c' & 0 & a' & 0 \\
        d' & \lambda b'+\nu c' & \nu b'+\mu c' & a' \\
      \end{array}
    \right) \left(
      \begin{array}{cccc}
        a & 0 & 0 & 0 \\
        b & a & 0 & 0 \\
        c & 0 & a & 0 \\
        d & \lambda b+\nu c & \nu b+\mu c & a \\
      \end{array}
    \right)\in \mathscr{E}_{\lambda,\mu,\nu}.
\end{align*}
Clearly, $1_{\mathscr{E}}=\sum\limits_{i=1}^4E_{ii}$. Let $e_1=E_{21}+\nu E_{43}+\lambda E_{42}$, $e_2=E_{31}+\nu E_{42}+\mu E_{43}$ and
$e_3=E_{41}$. Then $\mathscr{E}_{\lambda,\mu,\nu}=k1_{\mathscr{E}}\oplus ke_1\oplus ke_2\oplus ke_3$ as a $k$-vector space and
we have the following multiplication table:
\begin{center}
\begin{tabular}{l|llll}

                                                                & $1_{\mathscr{E}}$ & $e_1$ & $e_2$ & $e_3$ \\
                                                                 \hline
                                              $1_{\mathscr{E}}$ & $1_{\mathscr{E}}$  & $e_1$ & $e_2$  & $e_3$ \\
                                              $e_1$             & $e_1$  & $\lambda e_3$ & $\nu e_3$ & 0 \\
                                              $e_2$             & $e_2$  & $\nu e_3$ & $\mu e_3$ & 0 \\
                                              $e_3$             & $e_3$  & 0 & 0 & 0 \\
                                            \end{tabular}.
\end{center}

If $\lambda \mu-\nu^2\neq 0$, we
define a linear map
 \begin{align*}
 \theta: \mathscr{E}_{\lambda,\mu,\nu}\to \Hom_k(\mathscr{E}_{\lambda,\mu,\nu},k)
 \end{align*}
  by
 $$\left\{   \begin{array}{c}
                                                    1_{\mathscr{E}}   \\
                                                    e_1                \\
                                                     e_2                  \\
                                                     e_3    \\
                                                   \end{array}\right \}  \begin{array}{c}
                                                   \longrightarrow\\
                                                   \longrightarrow\\
                                                   \longrightarrow\\
                                                   \longrightarrow\\
                                                   \end{array}
                                                    \left\{   \begin{array}{c}
                                                     e_3^* \,\, \\
                                                     \lambda e_1^*+\nu e_2^*  \,\,      \\
                                                     \nu e_1^*+\mu e_2^* \,\,  \\
                                                     1_{\mathscr{E}}^* \,\, \\
                                                   \end{array}\right \}.
                                                     $$
 We want to show that $\theta$ is an isomorphism of left $\mathscr{E}_{\lambda,\mu,\nu}$-modules. One sees that $\theta$ is bijective since $\left|
       \begin{array}{cc}
         \lambda & \nu \\
         \nu & \mu \\
       \end{array}
     \right|\neq 0$.  It suffices to prove
 that $\theta$ is $\mathscr{E}_{\lambda,\mu,\nu}$-linear. Since
 \begin{align*}
 \theta(e_1\cdot 1_{\mathscr{E}})=\theta(e_1)=\lambda e_1^*+\nu e_2^* : \left\{   \begin{array}{c}
                                                    1_{\mathscr{E}}   \\
                                                    e_1                \\
                                                     e_2                  \\
                                                     e_3    \\
                                                   \end{array}\right \}  \begin{array}{c}
                                                   \longrightarrow\\
                                                   \longrightarrow\\
                                                   \longrightarrow\\
                                                   \longrightarrow\\
                                                   \end{array}
                                                    \left\{   \begin{array}{c}
                                                     0 \,\, \\
                                                     \lambda \,\,      \\
                                                     \nu \,\,  \\
                                                     0 \,\, \\
                                                   \end{array}\right \}
 \end{align*}
and
\begin{align*}
 e_1\theta(1_{\mathscr{E}})=e_1e_3^*: \left\{   \begin{array}{c}
                                                    1_{\mathscr{E}}   \\
                                                    e_1                \\
                                                     e_2                  \\
                                                     e_3    \\
                                                   \end{array}\right \}  \begin{array}{c}
                                                   \longrightarrow\\
                                                   \longrightarrow\\
                                                   \longrightarrow\\
                                                   \longrightarrow\\
                                                   \end{array}
                                                    \left\{   \begin{array}{c}
                                                     0 \,\, \\
                                                     \lambda \,\,      \\
                                                     \nu \,\,  \\
                                                     0 \,\, \\
                                                   \end{array}\right \},
 \end{align*}
we have $\theta(e_1\cdot 1_{\mathscr{E}})=e_1\theta(1_{\mathscr{E}})$. Similarly, we can show
\begin{align*}
&\theta(e_i\cdot 1_{\mathscr{E}})=e_i\theta(1_{\mathscr{E}}), i=2,3; \\
&\theta (1_{\mathscr{E}}\cdot 1_{\mathscr{E}})=1_{\mathscr{E}}\theta (1_{\mathscr{E}}), \theta (1_{\mathscr{E}} \cdot e_i)=1_{\mathscr{E}} \theta (e_i), i=1,2,3;\\
&\theta(e_i\cdot e_1)=e_i\theta(e_1), i=1,2,3; \\
&\theta(e_i\cdot e_2)=e_i\theta(e_2), i=1,2,3;\\
&\theta(e_i\cdot e_3)=e_i\theta(e_3), i=1,2,3.
\end{align*}
Then $\theta$ is $\mathscr{E}_{\lambda,\mu,\nu}$-linear and hence $\mathscr{E}_{\lambda,\mu,\nu}$ is a commutative Frobenius algebra. Since any commutative Frobenius algebra is symmetric, $\mathscr{E}_{\lambda,\mu,\nu}$ is a commutative symmetric algebra.

Conversely, if $\mathscr{E}_{\lambda,\mu,\nu}$ is a Frobenius algebra, then there exists an isomorphism $\theta: \mathscr{E}_{\lambda,\mu,\nu}\to \Hom_k(\mathscr{E}_{\lambda,\mu,\nu},k)$ of left $\mathscr{E}_{\lambda,\mu,\nu}$-modules. One sees that $(\lambda,\nu,\mu)\neq (0,0,0)$ since $\mathscr{E}_{0,0,0}\cong k[e_1,e_2,e_3]/(e_1^2,e_1e_2,e_1e_3,e_2^2,e_2e_3,e_3^2)$ is not a Frobenius algebra.
We have
\begin{align*}
\left\{   \begin{array}{c}
                                                    \theta (1_{\mathscr{E}})   \\
                                                   \theta(e_1)               \\
                                                   \theta(e_2)                  \\
                                                   \theta(e_3)    \\
                                                   \end{array}\right \}=M \left\{   \begin{array}{c}
                                                    1_{\mathscr{E}}^*   \\
                                                    e_1^*               \\
                                                    e_2^*                  \\
                                                    e_3^*    \\
                                                   \end{array}\right \}
\end{align*}
for some $M=(m_{ij})_{4\times 4}\in \mathrm{GL}_4(k)$. Since $\theta$ is $\mathscr{E}_{\lambda,\mu,\nu}$-linear, we have
\begin{align*}
&\quad\quad\quad \begin{cases}
e_1\theta(e_3)=\theta(e_1e_3)=0 =\theta(e_3e_1)=e_3\theta(e_1)\\
e_1\theta(e_2)=\theta(e_1e_2)=\nu\theta(e_3)=\theta(e_2e_1)=e_2\theta(e_1)\\
e_2\theta(e_3)=\theta(e_2e_3)=0=\theta(e_3e_2)=e_3\theta(e_2)\\
e_1\theta(e_1)=\theta(e_1e_1)=\lambda \theta(e_3), \\
e_2\theta(e_2)=\theta(e_2e_3)=\mu\theta(e_3),\\
 e_3\theta(e_3)=\theta(0),\\
\theta(e_3)=e_3\theta(1_{\mathscr{E}}), \\
\theta(e_2)=e_2\theta(1_{\mathscr{E}}),\\
\theta(e_1)=e_1\theta(1_{\mathscr{E}})
\end{cases} \\
&\Rightarrow\begin{cases}
(m_{42},m_{44}\lambda,m_{44}\nu,0)^T=(0,0,0,0)^T=(m_{24},0,0,0)^T, \\
(m_{32},m_{34}\lambda, m_{34}\nu,0)^T=(m_{41}\nu,m_{42}\nu,m_{43}\nu,m_{44}\nu)^T=(m_{23},m_{24}\nu,m_{24}\mu,0)^T,\\
(m_{43},m_{44}\lambda,m_{44}\mu,0)^T=(0,0,0,0)^T=(m_{34},0,0,0)^T,\\
(m_{22},m_{24}\lambda, m_{24}\nu,0)^T=(m_{41}\lambda, m_{42}\lambda, m_{43}\lambda,m_{44}\lambda)^T,\\
(m_{33},m_{34}\nu,m_{34}\mu,0)^T=(m_{41}\mu,m_{42}\mu,m_{43}\mu,m_{44}\mu)^T,\\
(m_{44},0,0,0)^T=(0,0,0,0)^T,\\
(m_{41},m_{42},m_{43},m_{44})^T=(m_{14},0,0,0)^T,\\
(m_{31},m_{32},m_{33},m_{34})^T=(m_{13},m_{14}\nu, m_{14}\mu,0)^T\\
(m_{21},m_{22},m_{23},m_{24})^T=(m_{12},m_{14}\lambda, m_{14}\nu,0)^T
\end{cases}\\
&\Rightarrow \begin{cases}
m_{42}=m_{44}=m_{24}=0\\
m_{32}=m_{41}\nu=m_{23},m_{34}\lambda=0,m_{43}\nu=0,\\
m_{43}=0=m_{34}\\
m_{22}=m_{41}\lambda\\
m_{33}=m_{41}\mu\\
m_{41}=m_{14}\\
m_{31}=m_{13}, m_{32}=m_{14}\nu, m_{33}=m_{14}\mu \\
m_{21}=m_{12},m_{22}=m_{14}\lambda, m_{23}=m_{14}\nu
\end{cases} \\
&\Rightarrow \begin{cases}
m_{24}=m_{34}=m_{42}=m_{43}=m_{44}=0\\
m_{14}=m_{41}, m_{13}=m_{31}, m_{12}=m_{21}\\
m_{32}=m_{41}\nu=m_{23}\\
m_{22}=m_{14}\lambda, m_{33}=m_{14}\mu.
\end{cases}
\end{align*}
Let $m_{11}=a,m_{12}=b,m_{13}=c,m_{14}=d$. Then $M=\left(
      \begin{array}{cccc}
        a & b & c & d \\
        b & d\lambda  & d\nu & 0 \\
        c & d\nu & d\mu & 0 \\
        d & 0    & 0 & 0 \\
      \end{array}
    \right)$ with $|M|=d^4(\nu^2-\lambda\mu)$ and hence $\nu^2-\lambda\mu\neq 0$.
\end{proof}

\begin{prop}\label{impcri}
Let $\mathcal{A}$ be a connected cochain DG algebra such that $$H(\mathcal{A})=k\langle \lceil y_1\rceil,\lceil y_2\rceil \rangle/(t_1\lceil y_1\rceil^2+t_2\lceil y_2\rceil^2+t_3(\lceil y_1\rceil \lceil y_2\rceil +\lceil y_2\rceil \lceil y_1\rceil))$$ with $y_1,y_2\in Z^1(\mathcal{A})$ and $(t_1,t_2,t_3)\in \Bbb{P}_k^2-\{(t_1,t_2,t_3)|t_1t_2-t_3^2\neq 0\}$. Then
$\mathcal{A}$ is a Koszul and Calabi-Yau DG algebra.
\end{prop}
\begin{proof}
The graded module ${}_{H(A)}k$ has the following minimal graded free resolution:
 $$0\to H(\mathcal{A})e_r\stackrel{d_2}{\to} H(\mathcal{A})\otimes (\bigoplus\limits_{i=1}^2 ke_{y_i}) \stackrel{d_1}{\to} H(\mathcal{A})\stackrel{\varepsilon}{\to} k\to 0,$$
where $\varepsilon, d_1$ and $d_2$ are defined by $\varepsilon |_{H^{\ge 1}(A)}=0$,  $\varepsilon |_{H^{0}(A)}=\mathrm{id}_k$, $d_1(e_{y_i})=y_i$ and $d_2(e_r)=t_1y_1e_1+t_2y_2e_2+t_3y_1e_2+t_3y_2e_1$. Applying the constructing procedure of Eilenberg-Moore resolution, we can construct a minimal semi-free resolution $F$ of the DG $\mathcal{A}$-module $k$. We have
$$F^{\#}=\mathcal{A}^{\#}\oplus [\mathcal{A}^{\#}\otimes (\bigoplus_{i=1}^2 k\Sigma e_{y_i})]\oplus \mathcal{A}^{\#}\Sigma^2e_{r}$$ and
$\partial_{F}$ is defined by
$$\left(
    \begin{array}{c}
      \partial_F(1) \\
      \partial_F(\Sigma e_1) \\
      \partial_F(\Sigma e_2) \\
       \partial_F(\Sigma^2 e_r)\\
    \end{array}
  \right)=\left(
            \begin{array}{cccc}
              0 & 0 & 0 & 0 \\
              y_1 & 0& 0 & 0 \\
              y_2 & 0 & 0 & 0 \\
              0 & t_1y_1+t_3y_2 & t_2y_2+t_3y_1 & 0\\
            \end{array}
          \right)\left(
                   \begin{array}{c}
                     1 \\
                     \Sigma e_1 \\
                     \Sigma e_2 \\
                     \Sigma^2 e_r \\
                   \end{array}
                 \right).$$ Hence $\mathcal{A}$ is a Koszul and homologically smooth DG algebra.
 By the minimality of $F$, we have
\begin{align*}
\Hom_{\mathcal{A}}(F,k)=\{k1^*\oplus [\bigoplus_{i=1}^n k(\Sigma e_{y_i})^*]\oplus k(\Sigma^2 e_r)^*\}.
\end{align*}
So the Ext-algebra $E=H(\Hom_{\mathcal{A}}(F,F))$  is concentrated in degree $0$. On the other hand, $$\Hom_{\mathcal{A}}(F,F)^{\#}\cong\{k1^*\oplus [\bigoplus_{i=1}^n k(\Sigma e_{y_i})^*]\oplus k(\Sigma^2 e_r)^*\}\otimes_{k} F^{\#}$$ is concentrated in degrees $\ge 0$ since $|1^*|=|(\Sigma e_{y_i})^*|=|(\Sigma^2e_{r})^*|=0$ and $F$ is concentrated in degrees $\ge 0$. This implies that $E= Z^0(\Hom_{\mathcal{A}}(F,F))$.
Since $F^{\#}$ is a free graded $\mathcal{A}^{\#}$-module with a basis
$\{1, \Sigma e_{y_1}, \Sigma e_{y_2},\Sigma^2 e_r \}$ concentrated in degree $0$,
  the elements in  $\Hom_{\mathcal{A}}(F,F)^0$ is one to one correspondence with the matrixes in $M_4(k)$. Indeed, any $f\in \Hom_{\mathcal{A}}(F,F)^0$ is uniquely determined by
  a matrix $A_f=(a_{ij})_{4\times 4}\in M_4(k)$ with
$$\left(
                         \begin{array}{c}
                          f(1) \\
                          f(\Sigma e_1)\\
                           f(\Sigma e_2)\\
                            f(\Sigma e_r)
                         \end{array}
                       \right) =      A_f \cdot \left(
                         \begin{array}{c}
                          1 \\
                          \Sigma e_1\\
                           \Sigma e_2\\
                            \Sigma e_r
                         \end{array}
                       \right).  $$
                       And $f\in  Z^0[\Hom_{\mathcal{A}}(F,F)]$ if and only if $\partial_{F}\circ f=f\circ \partial_{F}$, if and only if
\begin{small}
 $$ A_f \left(
            \begin{array}{cccc}
              0 & 0 & 0 & 0 \\
              y_1 & 0& 0 & 0 \\
              y_2 & 0 & 0 & 0 \\
              0 & t_1y_1+t_3y_2 & t_2y_2+t_3y_1 & 0\\
            \end{array}
          \right) = \left(
            \begin{array}{cccc}
              0 & 0 & 0 & 0 \\
              y_1 & 0& 0 & 0 \\
              y_2 & 0 & 0 & 0 \\
              0 & t_1y_1+t_3y_2 & t_2y_2+t_3y_1 & 0\\
            \end{array}
          \right) A_f,$$
 \end{small} which is also equivalent to
                       $$\begin{cases}
                       a_{ij}=0, \forall i<j\\
                       a_{11}=a_{22}=a_{33}= a_{44}\\
                       a_{32}=0 \\
                       a_{42}=a_{21}t_1+a_{31}t_3\\
                       a_{43}=a_{21}t_3+a_{31}t_2
                       \end{cases}$$
by direct computations. Hence the the Ext-algebra $$ E\cong \left\{\left(
                                                             \begin{array}{cccc}
                                                               a & 0 & 0 & 0  \\
                                                               b & a & 0 & 0 \\
                                                               c & 0 & a & 0  \\
                                                               d & t_1b+t_3c & t_3b+t_2c & a
                                                             \end{array}
                                                           \right)
\quad | \quad a,b,c,d\in k \right\} = \mathscr{E}_{t_1,t_2,t_3}.$$ Then $\mathcal{A}$ is homologically smooth and Koszul since $E$ is a finite dimensional algebra concentrated in degree $0$.
Since $t_1t_2-t_3^2\neq 0$, $E$ is a symmetric Frobenius algebra
by Lemma \ref{symmetry}.
Hence $\mathrm{Tor}_{\mathcal{A}}(k_{\mathcal{A}},{}_{\mathcal{A}}k)\cong
E^*$ is a symmetric coalgebra when $t_1t_2-t_3^2\neq 0$.  By Remark \ref{methods},  $\mathcal{A}$ is a Calabi-Yau DG algebra.
\end{proof}
\begin{prop}\label{nonhs}
Let $\mathcal{A}$ be a connected cochain DG algebra such that $$H(\mathcal{A})=k\langle \lceil y_1\rceil,\lceil y_2\rceil \rangle/(t_1\lceil y_1\rceil^2+t_2\lceil y_2\rceil^2+t_3(\lceil y_1\rceil \lceil y_2\rceil +\lceil y_2\rceil \lceil y_1\rceil))$$ with $y_1,y_2\in Z^1(\mathcal{A})$ and $(t_1,t_2,t_3)\in \{(t_1,t_2,t_3)\in \Bbb{P}_k^2 |t_1t_2-t_3^2=0\}$. Then
$\mathcal{A}$ is not homologically smooth but Koszul.
\end{prop}
\begin{proof}
The trivial module ${}_{H(\mathcal{A})}k$ admits a finitely generated linearly  minimal free resolution
$$\cdots \xrightarrow{d_{n+1}} F_{n} \xrightarrow{d_n} \cdots \xrightarrow{d_3}  F_2\xrightarrow{d_2}F_1=H(\mathcal{A})e_x\oplus H(\mathcal{A})e_y
\xrightarrow{d_1} H(\mathcal{A})\xrightarrow{\varepsilon} {}_{H(\mathcal{A})}k\rightarrow 0,$$
where \begin{align*}
& d_1(e_{y_1})=\lceil y_1\rceil, d_1(e_{y_1})=\lceil y_2\rceil;\\
&d_2(e_2)=(t_1\lceil y_1\rceil+\sqrt{t_1t_2}\lceil y_2\rceil)e_{y_1}+(\sqrt{t_1t_2}\lceil y_1\rceil+t_2\lceil y_2\rceil)e_{y_2};\\
&F_{n-1}=H(\mathcal{A})e_{n-1}, d_n(e_n)=(t_1\lceil y_1\rceil +\sqrt{t_1t_2}\lceil y_2\rceil)e_{n-1}, n\ge 3.
\end{align*}
From the free resolution above, we can construct
 an Eilenberg-Moore resolution $F$ of the DG $\mathcal{A}$-module ${}_{\mathcal{A}}k$. By the constructing procedure of Eilenberg-Moore resolution described in \cite[P. 279 - 280]{FHT2}, one sees that $F$ is semi-free with
 $$F^{\#}=\mathcal{A}^{\#}\oplus \mathcal{A}^{\#}\Sigma e_{y_1} \oplus \mathcal{A}\Sigma e_{y_2} \oplus [\bigoplus\limits_{i=2}^{+\infty} \mathcal{A}^{\#}\Sigma^i e_{i}]$$
 and a semi-basis $\{1,\Sigma e_{y_1}, \Sigma e_{y_2}, \Sigma^i e_{i}, i\ge 2\}$. It is easy to see that $\partial_F(F)\subseteq \frak{m}_{\mathcal{A}}F$ since $F$ admits a semi-basis concentrated in degree zero. One sees that $\mathcal{A}$ is Koszul, but not homologically smooth
 since $\{1,\Sigma e_{y_1}, \Sigma e_{y_2}, \Sigma^i e_{i}, i\ge 2\}$ is an infinite set.

\end{proof}

\section{cohomology and Calabi-Yau properties}
From this section, we will do research on homological properties of $\mathcal{A}_{\mathcal{O}_{-1}(k^n)}(M)$. For the case $n=2$, we have the following proposition.
\begin{prop}\label{ntwocase}\cite[Proposition 3.3]{MH} For $M=(m_{ij})_{2\times 2}\in M_2(k)$, we have
\begin{align*}
H[\mathcal{A}_{\mathcal{O}_{-1}(k^2)}(M)]=\begin{cases}
k,\,\,\, \text{if}\,\,\, |M|\neq 0 \\
 k[\lceil x_2\rceil],\,\,\,   \text{if}\,\,\, m_{11}\neq 0 \,\,\, \text{and}\,\,\,  m_{12}=m_{21}=m_{22}=0 \\
 k[\lceil x_1^2\rceil,\lceil x_2\rceil]/(\lceil x_2^2\rceil), \,\,\, \text{if}\,\,\,m_{12}\neq 0, m_{11}=m_{21}=m_{22}=0\\
 k[\lceil x_2\rceil], \,\,\, \text{if}\,\,\, m_{11}\neq 0, m_{12}\neq 0 \,\,\, \text{and}\,\,\, m_{21}=m_{22}=0\\
 k[\lceil m_{21}x_1-m_{11}x_2\rceil],  \,\,\, \text{if}\,\,\, m_{11}\neq 0, m_{21}\neq 0\,\, \text {and}\,\, m_{12}=m_{22}=0\\
 k[\lceil m_{21}x_1-m_{11}x_2\rceil],  \,\,\, \text{if}\,\,\, m_{ij}\neq 0,\forall i,j,  m_{11}^2\neq m_{21}m_{22},|M|=0 \\
 k[\lceil m_{21}x_1-m_{11}x_2\rceil,\lceil x_2^2\rceil]/(\lceil m_{21}x_1-m_{11}x_2\rceil^2), \,\, \text{if}\,\,\, m_{ij}\neq 0,\forall i,j, \\
 \quad\quad\quad\quad \quad\quad\quad\quad \quad\quad\quad\quad \quad\quad\quad  m_{11}^2= m_{21}m_{22},|M|=0.
\end{cases}
\end{align*}
\end{prop}

\begin{rem}
By Proposition \ref{ntwocase} and Lemma \ref{MH}, one sees that $\mathcal{A}_{\mathcal{O}_{-1}(k^2)}(M)$ is a Koszul Calabi-Yau DG algebra in the  following cases:
\begin{enumerate}
\item $|M|\neq 0$;
\item $m_{11}\neq 0$ and  $m_{12}=m_{21}=m_{22}=0$;
\item $m_{11}\neq 0, m_{12}\neq 0$ and $m_{21}=m_{22}=0$;
\item $m_{11}\neq 0, m_{21}\neq 0$ and $m_{12}=m_{22}=0$;
\item $m_{ij}\neq 0,\forall i,j,  m_{11}^2\neq m_{21}m_{22},|M|=0$.
\end{enumerate}
For the other cases, the statement that $\mathcal{A}_{\mathcal{O}_{-1}(k^2)}(M)$ is a Koszul Calabi-Yau DG algebra also holds
by \cite[Theorem C]{MH}.
\end{rem}
It is natural for us to ask whether each $\mathcal{A}_{\mathcal{O}_{-1}(k^3)}(M), M\in M_3(k)$  is also a Koszul Calabi-Yau DG algebra.
 By \cite[Proposition $3.2$]{MYY},
one sees that $\mathcal{A}_{\mathcal{O}_{-1}(k^3)}(M)$ is a  Calabi-Yau DG algebra when $M=0$.
 Note that each $\mathcal{A}_{\mathcal{O}_{-1}(k^3)}(M)$ is actually a $3$-dimensional DG Sklyanin algebra in \cite{MWYZ}. When $M$ is not a zero matrix,
we have the following proposition on $H[\mathcal{A}_{\mathcal{O}_{-1}(k^3)}(M)]$. We refer the reader to \cite[Theorem A]{MR} for detailed computations.
\begin{prop}\cite{MR}\label{cohomology}
Assume that $M$ is a matrix in $M_3(k)$. Then we have the following statements on $H[\mathcal{A}_{\mathcal{O}_{-1}(k^3)}(M)]$.
\begin{enumerate}
\item $H(\mathcal{A}_{\mathcal{O}_{-1}(k^3)}(M))=k$, when  $r(M)=3$;
\item $H[\mathcal{A}_{\mathcal{O}_{-1}(k^3)}(M)]=k[\lceil t_1x +t_2y+t_3z\rceil]$ if $r(M)=2$, $s_1t_1^2+s_2t_2^2+s_3t_3^2\neq 0$, where $k(s_1,s_2,s_3)^T$ and $k(t_1,t_2,t_3)^T$ be the solution spaces of  homogeneous linear equations $MX=0$ and $M^TX=0$, respectively;
\item $H(\mathcal{A}_{\mathcal{O}_{-1}(k^3)}(M))$ is  $$k[\lceil t_1x_1 +t_2x_2+t_3x_3\rceil,\lceil s_1x_1^2+s_2x_2^2+s_3x_3^2\rceil ]/(\lceil t_1x_1 +t_2x_2+t_3x_3\rceil^2),$$
    if $r(M)=2$, $s_1t_1^2+s_2t_2^2+s_3t_3^2=0$, where $k(s_1,s_2,s_3)^T$ and $k(t_1,t_2,t_3)^T$ are the solution spaces of  homogeneous linear equations $MX=0$ and $M^TX=0$, respectively;
\item $H[\mathcal{A}_{\mathcal{O}_{-1}(k^3)}(M)]$ is
$$\frac{k\langle \lceil l_1x_1-x_2\rceil, \lceil l_2x_1-x_3\rceil \rangle}{(m_{12}\lceil l_1x_1-x_2\rceil^2+m_{13}\lceil l_2x_1-x_3\rceil^2-\frac{\lceil l_1x_1-x_2\rceil \lceil l_2x_1-x_3\rceil+\lceil l_2x_1-x_3\rceil \lceil l_1x_1-x_2\rceil}{\frac{2l_1l_2}{m_{12}l_1^2+m_{13}l_2^2-m_{11}}})}$$
 $$\text{when}\quad M=\left(
                                 \begin{array}{ccc}
                                   m_{11} & m_{12} & m_{13} \\
                                   l_1m_{11} & l_1m_{12} & l_1m_{13} \\
                                   l_2m_{11} & l_2m_{12} & l_2m_{13} \\
                                 \end{array}
                               \right),m_{12}l_1^2+m_{13}l_2^2\neq m_{11},l_1l_2\neq 0;$$
\item  $H[\mathcal{A}_{\mathcal{O}_{-1}(k^3)}(M)]=\frac{k\langle \lceil l_1x_1-x_2\rceil, \lceil l_2x_1-x_3\rceil \rangle}{(\lceil l_1x_1-x_2\rceil\lceil l_2x_1-x_3\rceil+\lceil l_2x_1-x_3\rceil \lceil l_1x_1-x_2\rceil)}$, when $$M=\left(
                                 \begin{array}{ccc}
                                   m_{11} & m_{12} & m_{13} \\
                                   l_1m_{11} & l_1m_{12} & l_1m_{13} \\
                                   l_2m_{11} & l_2m_{12} & l_2m_{13} \\
                                 \end{array}
                               \right)$$  with $(m_{11},m_{12},m_{13})\neq 0$, $m_{12}l_1^2+m_{13}l_2^2\neq m_{11}$ and $l_1l_2= 0$;
\item $H[\mathcal{A}_{\mathcal{O}_{-1}(k^3)}(M)]=\frac{k\langle \lceil l_1x_1-x_2\rceil, \lceil l_2x_1-x_3\rceil \rangle}{(m_{12}\lceil l_1x_1-x_2\rceil^2+m_{13}\lceil l_2x_1-x_3\rceil^2)}$, when
    $$M=\left(
                                 \begin{array}{ccc}
                                   m_{11} & m_{12} & m_{13} \\
                                   l_1m_{11} & l_1m_{12} & l_1m_{13} \\
                                   l_2m_{11} & l_2m_{12} & l_2m_{13} \\
                                 \end{array}
                               \right),(m_{11},m_{12},m_{13})\neq 0$$  with $m_{12}l_1^2+m_{13}l_2^2= m_{11}$ and $l_1l_2\neq 0$;
\item $H[\mathcal{A}_{\mathcal{O}_{-1}(k^3)}(M)]=\frac{k\langle \lceil l_1x_1-x_2\rceil, \lceil x_3\rceil,\lceil x_1^2\rceil \rangle}{\left(
                                                                                                       \begin{array}{c}
                                                                                                         m_{12}\lceil l_1x_1-x_2\rceil^2+m_{13}\lceil x_3\rceil^2 \\
                                                                                                         \lceil x_1^2\rceil \lceil l_1x_1-x_2\rceil- \lceil l_1x_1-x_2\rceil \lceil x_1^2\rceil \\
                                                                                                       \lceil x_1^2\rceil \lceil x_3\rceil-\lceil x_3\rceil \lceil x_1^2\rceil   \\
                                                                                                       \lceil l_1x_1-x_2\rceil \lceil x_3\rceil +\lceil x_3\rceil \lceil l_1x_1-x_2\rceil
                                                                                                       \end{array}
                                                                                                     \right)},$ when $$M=\left(
                                 \begin{array}{ccc}
                                   m_{11} & m_{12} & m_{13} \\
                                   l_1m_{11} & l_1m_{12} & l_1m_{13} \\
                                   l_2m_{11} & l_2m_{12} & l_2m_{13} \\
                                 \end{array}
                               \right), (m_{11},m_{12},m_{13})\neq 0, $$ with $m_{12}l_1^2+m_{13}l_2^2=m_{11}$, $l_1\neq 0$ and $l_2=0$;
\item $H[\mathcal{A}_{\mathcal{O}_{-1}(k^3)}(M)]=\frac{k\langle \lceil l_2x_1-x_3\rceil, \lceil x_2\rceil,\lceil x_1^2\rceil \rangle}{\left(
                                                                                                       \begin{array}{c}
                                                                                                         m_{13}\lceil l_2x_1-x_3\rceil^2+m_{12}\lceil x_2\rceil^2 \\
                                                                                                         \lceil x_1^2\rceil \lceil l_2x_1-x_3\rceil- \lceil l_2x_1-x_3\rceil \lceil x_1^2\rceil \\
                                                                                                       \lceil x_1^2\rceil \lceil x_2\rceil-\lceil x_2\rceil \lceil x_1^2\rceil   \\
                                                                                                       \lceil l_2x_1-x_3\rceil \lceil x_2\rceil +\lceil x_2\rceil \lceil l_2x_1-x_3\rceil
                                                                                                       \end{array}
                                                                                                     \right)},$
when $$M=\left(
                                 \begin{array}{ccc}
                                   m_{11} & m_{12} & m_{13} \\
                                   l_1m_{11} & l_1m_{12} & l_1m_{13} \\
                                   l_2m_{11} & l_2m_{12} & l_2m_{13} \\
                                 \end{array}
                               \right), (m_{11},m_{12},m_{13})\neq 0, $$ with
 $m_{12}l_1^2+m_{13}l_2^2=m_{11}$, $l_2\neq 0$ and $l_1= 0$;

   \item $H[\mathcal{A}_{\mathcal{O}_{-1}(k^3)}(M)]=\frac{k\langle \lceil x_3\rceil, \lceil x_2\rceil,\lceil x_1^2\rceil \rangle}{\left(
                                                                                                       \begin{array}{c}
                                                                                                         m_{12}\lceil x_2\rceil^2+m_{13}\lceil x_3\rceil^2 \\
                                                                                                         \lceil x_1^2\rceil \lceil x_3\rceil- \lceil x_3\rceil \lceil x_1^2\rceil \\
                                                                                                       \lceil x_1^2\rceil \lceil x_2\rceil-\lceil x_2\rceil \lceil x_1^2\rceil   \\
                                                                                                       \lceil x_3\rceil \lceil x_2\rceil +\lceil x_2\rceil \lceil x_3\rceil
                                                                                                       \end{array}
                                                                                                     \right)}$ when
                                                                                                      $$M=\left(
                                 \begin{array}{ccc}
                                   m_{11} & m_{12} & m_{13} \\
                                   l_1m_{11} & l_1m_{12} & l_1m_{13} \\
                                   l_2m_{11} & l_2m_{12} & l_2m_{13} \\
                                 \end{array}
                               \right), (m_{11},m_{12},m_{13})\neq 0, $$
                               with $m_{12}l_1^2+m_{13}l_2^2=m_{11}$, $l_1= 0$ and $l_2=0$.
\end{enumerate}
\end{prop}

\begin{rem}\label{redtosimp}
Note that $(4-9)$ in Proposition  \ref{cohomology} don't include all cases for $r(M)=1$. However,  we only need to consider them in the sense of isomorphism.
Indeed, we can see the reasons by applying Theorem \ref{iso} and the following fact.
For any $(a,b,c)\neq (0,0,0)$ and $l_1,l_2\in k$, let $$M=\left(
                                 \begin{array}{ccc}
                                   a & b & c \\
                                   l_1a & l_1b & l_1c \\
                                   l_2a & l_2b & l_2c \\
                                 \end{array}
                               \right), C=\left(
                                                  \begin{array}{ccc}
                                                    0 & 1 & 0 \\
                                                    1 & 0 & 0 \\
                                                    0 & 0 & 1 \\
                                                  \end{array}
                                                \right), C'=\left(
                                                  \begin{array}{ccc}
                                                    0 & 0 & 1 \\
                                                    0 & 1 & 0 \\
                                                    1 & 0 & 0 \\
                                                  \end{array}
                                                \right).$$
                                Then
\begin{align*}
\chi(M,C)&=\left(
                                                  \begin{array}{ccc}
                                                    0 & 1 & 0 \\
                                                    1 & 0 & 0 \\
                                                    0 & 0 & 1 \\
                                                  \end{array}
                                                \right)\left(
                                 \begin{array}{ccc}
                                   a & b & c \\
                                   l_1a & l_1b & l_1c \\
                                   l_2a & l_2b & l_2c \\
                                 \end{array}
                               \right)\left(
                                                  \begin{array}{ccc}
                                                    0 & 1& 0 \\
                                                    1 & 0 & 0 \\
                                                    0 & 0 & 1 \\
                                                  \end{array}
                                                \right)\\
&=\left(
                                 \begin{array}{ccc}
                                   l_1b & l_1a & l_1c \\
                                   b   &   a & c \\
                                   l_2b & l_2a & l_2c \\
                                 \end{array}
                               \right)
\end{align*}
and
\begin{align*}
\chi(M,C')&=\left(
                                                  \begin{array}{ccc}
                                                    0 & 0 & 1 \\
                                                    0 & 1 & 0 \\
                                                    1 & 0 & 0 \\
                                                  \end{array}
                                                \right)\left(
                                 \begin{array}{ccc}
                                   a & b & c \\
                                   l_1a & l_1b & l_1c \\
                                   l_2a & l_2b & l_2c \\
                                 \end{array}
                               \right)\left(
                                                  \begin{array}{ccc}
                                                    0 & 0 & 1 \\
                                                    0 & 1 & 0 \\
                                                    1 & 0 & 0 \\
                                                  \end{array}
                                                \right)\\
&=\left(
                                 \begin{array}{ccc}
                                   l_2c & l_2b & l_2a \\
                                   l_1c & l_1b & l_1a \\
                                   c    & b    & a \\
                                 \end{array}
                               \right).
\end{align*}
\end{rem}

\begin{prop}\label{easycases}
For $M=(m_{ij})_{3\times 3}\in \mathrm{M}_3(k)$,  $\mathcal{A}_{\mathcal{O}_{-1}(k^3)}(M)$ is a Koszul Calabi-Yau DG algebra in the following cases:
\begin{enumerate}
\item  $r(M)=3$;
\item  $r(M)=2$, $s_1t_1^2+s_2t_2^2+s_3t_3^2\neq 0$, where $k(s_1,s_2,s_3)^T$ and $k(t_1,t_2,t_3)^T$ be the solution spaces of  homogeneous linear equations $MX=0$ and $M^TX=0$, respectively;
\item  $$M=\left(
                                 \begin{array}{ccc}
                                   m_{11} & m_{12} & m_{13} \\
                                   l_1m_{11} & l_1m_{12} & l_1m_{13} \\
                                   l_2m_{11} & l_2m_{12} & l_2m_{13} \\
                                 \end{array}
                               \right),(m_{11},m_{12},m_{13})\neq 0, l_1l_2\neq 0,$$  $m_{12}l_1^2+m_{13}l_2^2\neq m_{11}$ and $4m_{12}m_{13}l_1^2l_2^2\neq (m_{12}l_1^2+m_{13}l_2^2-m_{11})^2;$
\item $$M=\left(
                                 \begin{array}{ccc}
                                   m_{11} & m_{12} & m_{13} \\
                                   l_1m_{11} & l_1m_{12} & l_1m_{13} \\
                                   l_2m_{11} & l_2m_{12} & l_2m_{13} \\
                                 \end{array}
                               \right), (m_{11},m_{12},m_{13})\neq 0,$$ with $m_{12}l_1^2+m_{13}l_2^2\neq m_{11}$ and $l_1l_2= 0$;
\item $$M=\left(
                                 \begin{array}{ccc}
                                   m_{11} & m_{12} & m_{13} \\
                                   l_1m_{11} & l_1m_{12} & l_1m_{13} \\
                                   l_2m_{11} & l_2m_{12} & l_2m_{13} \\
                                 \end{array}
                               \right),m_{12}m_{13}\neq 0, l_1l_2\neq 0\,\,\text{and} $$  $m_{12}l_1^2+m_{13}l_2^2= m_{11}$.
\end{enumerate}
\end{prop}

\begin{proof}

By Lemma \ref{MH}, Lemma \ref{impcri} and Proposition \ref{cohomology}, it is easy to check that the statement holds for cases $(1),(2)$ and $(4)$.
For case $(3)$, one sees that
$H[\mathcal{A}_{\mathcal{O}_{-1}(k^3)}(M)]$ is
$$\frac{k\langle \lceil l_1x_1-x_2\rceil, \lceil l_2x_1-x_3\rceil \rangle}{(m_{12}\lceil l_1x_1-x_2\rceil^2+m_{13}\lceil l_2x_1-x_3\rceil^2-\frac{\lceil l_1x_1-x_2\rceil \lceil l_2x_1-x_3\rceil+\lceil l_2x_1-x_3\rceil \lceil l_1x_1-x_2\rceil}{\frac{2l_1l_2}{m_{12}l_1^2+m_{13}l_2^2-m_{11}}})}$$
by Proposition \ref{cohomology}.
Since $4m_{12}m_{13}l_1^2l_2^2\neq (m_{12}l_1^2+m_{13}l_2^2-m_{11})^2$, $\mathcal{A}_{\mathcal{O}_{-1}(k^3)}(M)$ is a Koszul Calabi-Yau DG algebra
by Lemma \ref{impcri}. For case $5$, we have
$$H[\mathcal{A}_{\mathcal{O}_{-1}(k^3)}(M)]=\frac{k\langle \lceil l_1x_1-x_2\rceil, \lceil l_2x_1-x_3\rceil \rangle}{(m_{12}\lceil l_1x_1-x_2\rceil^2+m_{13}\lceil l_2x_1-x_3\rceil^2)}$$
by Proposition \ref{cohomology}. Since $m_{12}m_{13}\neq 0$, $\mathcal{A}_{\mathcal{O}_{-1}(k^3)}(M)$ is a Koszul Calabi-Yau DG algebra
by Lemma \ref{impcri}
\end{proof}

\begin{prop}\label{noncycase}
The DG algebra $\mathcal{A}_{\mathcal{O}_{-1}(k^3)}(M)$ is not homologically smooth when
 $$M=\left(
                                 \begin{array}{ccc}
                                   m_{11} & m_{12} & m_{13} \\
                                   l_1m_{11} & l_1m_{12} & l_1m_{13} \\
                                   l_2m_{11} & l_2m_{12} & l_2m_{13} \\
                                 \end{array}
                               \right), m_{12}l_1^2+m_{13}l_2^2\neq m_{11}, l_1l_2\neq 0$$ and $4m_{12}m_{13}l_1^2l_2^2= (m_{12}l_1^2+m_{13}l_2^2-m_{11})^2$. In this case,
                               neither $m_{12}m_{11}< 0$ nor $m_{13}m_{11}< 0$ will occur.
 Furthermore,
\begin{enumerate}
\item if $m_{11}=0$, then $m_{12}l_1=m_{13}l_2$ and
$\mathcal{A}_{\mathcal{O}_{-1}(k^3)}(M)$ is isomorphic to $\mathcal{A}_{\mathcal{O}_{-1}(k^3)}(N)$, where $N=\left(
                                 \begin{array}{ccc}
                                   0 & m_{12} & m_{12} \\
                                   0 & l_1m_{12} & l_1m_{12} \\
                                   0 & l_2\sqrt{m_{12}m_{13}} & l_2\sqrt{m_{12}m_{13}} \\
                                 \end{array}
                               \right);$
\item if $m_{11}m_{12}>0, m_{11}m_{13}>0$ then $\mathcal{A}_{\mathcal{O}_{-1}(k^3)}(M)$  is isomorphic to $\mathcal{A}_{\mathcal{O}_{-1}(k^3)}(Q)$, where $$Q=\left(
                                 \begin{array}{ccc}
                                   m_{11}\sqrt{m_{12}m_{13}} & m_{11}\sqrt{m_{12}m_{13}} & m_{11}\sqrt{m_{12}m_{13}} \\
                                   l_1m_{12}\sqrt{m_{11}m_{13}} & l_1m_{12}\sqrt{m_{11}m_{13}} & l_1m_{12}\sqrt{m_{11}m_{13}} \\
                                   l_2m_{13}\sqrt{m_{11}m_{12}} & l_2m_{13}\sqrt{m_{11}m_{12}} & l_2m_{13}\sqrt{m_{11}m_{12}} \\
                                 \end{array}
                               \right).$$
\end{enumerate}
\end{prop}
\begin{proof}
In this case,  $H[\mathcal{A}_{\mathcal{O}_{-1}(k^3)}(M)]$ is
$$\frac{k\langle \lceil l_1x_1-x_2\rceil, \lceil l_2x_1-x_3\rceil \rangle}{(m_{12}\lceil l_1x_1-x_2\rceil^2+m_{13}\lceil l_2x_1-x_3\rceil^2-\frac{\lceil l_1x_1-x_2\rceil \lceil l_2x_1-x_3\rceil+\lceil l_2x_1-x_3\rceil \lceil l_1x_1-x_2\rceil}{\frac{2l_1l_2}{m_{12}l_1^2+m_{13}l_2^2-m_{11}}})}$$
by Proposition \ref{cohomology}. Since $4m_{12}m_{13}l_1^2l_2^2= (m_{12}l_1^2+m_{13}l_2^2-m_{11})^2$, $\mathcal{A}_{\mathcal{O}_{-1}(k^3)}(M)$ is not homologically smooth  by Proposition \ref{nonhs}. Clearly,  $m_{12}m_{13}>0$ by the assumption.
If $m_{13}m_{11}<0$, we let
$C=\left(
            \begin{array}{ccc}
              0 & -1 & 0 \\
              1 & 0 & 0 \\
              0 & 0 & 1 \\
            \end{array}
          \right)$.
Then
          \begin{align*}
 &\quad \chi(M,C)\\
         &=\left(
             \begin{array}{ccc}
               0 & 1 & 0 \\
               -1 & 0 & 0 \\
               0 & 0 & 1 \\
             \end{array}
           \right)\left(
                    \begin{array}{ccc}
                      m_{11} & m_{12} & m_{13} \\
                      l_1m_{11} & l_1m_{12} & l_1m_{13} \\
                      l_2m_{11} & l_2m_{12} & l_2m_{13} \\
                    \end{array}
                  \right)\left(
                           \begin{array}{ccc}
                             1 & 0 & 0 \\
                             0 & 0 & 1 \\
                             0 & 1 & 0 \\
                           \end{array}
                         \right)\\
           &=\left(
                           \begin{array}{ccc}
                             l_1m_{12}  & l_1m_{11} & l_1m_{13} \\
                             -m_{12} & -m_{11} & -m_{13} \\
                             l_2m_{12} & l_2m_{11} & l_2m_{13} \\
                           \end{array}
                         \right).
          \end{align*}
Let $M'=\left(
                                 \begin{array}{ccc}
                                   m_{11}' & m_{12}' & m_{13}' \\
                                   l_1'm_{11}' & l_1'm_{12}' & l_1'm_{13}' \\
                                   l_2'm_{11}' & l_2'm_{12}' & l_2m_{13}' \\
                                 \end{array}
                               \right)=\left(
                           \begin{array}{ccc}
                             l_1m_{12}  & l_1m_{11} & l_1m_{13} \\
                             -m_{12} & -m_{11} & -m_{13} \\
                             l_2m_{12} & l_2m_{11} & l_2m_{13} \\
                           \end{array}
                         \right)$. Then
                          $$ \begin{cases}
                          \mathcal{A}_{\mathcal{O}_{-1}(k^3)}(M)\cong \mathcal{A}_{\mathcal{O}_{-1}(k^3)}(M')\\
                          l_1'=-\frac{1}{l_1}, l_2'=\frac{l_2}{l_1},\\
                           m_{11}'=l_1m_{12},m_{12}'=l_1m_{11}, m_{13}'=l_1m_{13}.
                          \end{cases}$$
So $l_1'l_2'=\frac{-l_2}{l_1^2}\neq 0$ and
\begin{align*}
m_{12}'(l_1')^2+m_{13}'(l_2')^2- m_{11}'=\frac{1}{l_1}(m_{11}+m_{13}l_2^2-l_1^2m_{12}).
\end{align*}
When $m_{11}+m_{13}l_2^2-l_1^2m_{12}\neq 0$, then $H[\mathcal{A}_{\mathcal{O}_{-1}(k^3)}(M')]$ is
$$\frac{k\langle \lceil l_1'x_1-x_2\rceil, \lceil l_2'x_1-x_3\rceil \rangle}{(m_{12}'\lceil l_1'x_1-x_2\rceil^2+m_{13}'\lceil l_2'x_1-x_3\rceil^2-\frac{\lceil l_1'x_1-x_2\rceil \lceil l_2'x_1-x_3\rceil+\lceil l_2'x_1-x_3\rceil \lceil l_1'x_1-x_2\rceil}{\frac{2l_1'l_2'}{m_{12}'(l_1')^2+m_{13}'(l_2')^2-m_{11}'}})}$$
by Proposition \ref{cohomology}. Since $m_{12}'m_{13}'=l_1^2m_{11}m_{13}<0$, we have $4m_{12}'m_{13}'(l_1')^2(l_2')^2\neq (m_{12}'(l_1')^2+m_{13}'(l_2')^2-m_{11}')^2$. Then $\mathcal{A}_{\mathcal{O}_{-1}(k^n)}(M')$ is a Koszul Calabi-Yau  connected cochain DG algebra by Proposition \ref{easycases} (3). This is impossible since $\mathcal{A}_{\mathcal{O}_{-1}(k^n)}(M)$ is not homologically smooth.
When $m_{11}+m_{13}l_2^2-l_1^2m_{12}=0$, then $$H[\mathcal{A}_{\mathcal{O}_{-1}(k^3)}(M')]=\frac{k\langle \lceil l_1'x_1-x_2\rceil, \lceil l_2'x_1-x_3\rceil \rangle}{(m_{12}'\lceil l_1'x_1-x_2\rceil^2+m_{13}'\lceil l_2'x_1-x_3\rceil^2)}$$ by Proposition \ref{cohomology}. Since $m_{12}'m_{13}'=l_1^2m_{11}m_{13}\neq 0$, the DG algebra $\mathcal{A}_{\mathcal{O}_{-1}(k^n)}(M')$ is a Koszul Calabi-Yau  connected cochain DG algebra by Lemma \ref{impcri}. This also contradicts with the fact that $\mathcal{A}_{\mathcal{O}_{-1}(k^n)}(M)$ is not homologically smooth.

Therefore, $m_{13}m_{11}< 0$ can't occur. Similarly, we can show that $m_{11}m_{12}<0$ is also impossible.
So we have either $m_{11}=0$ or $m_{11}m_{12}>0,m_{11}m_{13}>0$.

(1) If $m_{11}=0$, then $m_{12}l_1=m_{13}l_2$. Let
$$C=\left(
            \begin{array}{ccc}
              1 & 0 & 0 \\
              0 & 1 & 0 \\
              0 & 0 & \sqrt{\frac{m_{12}}{m_{13}}} \\
            \end{array}
          \right).$$
Then
          \begin{align*}
 &\quad \chi(M,C)\\
         &=\left(
             \begin{array}{ccc}
               1 & 0 & 0 \\
               0 & 1 & 0 \\
               0 & 0 & \sqrt{\frac{m_{13}}{m_{12}}} \\
             \end{array}
           \right)\left(
                    \begin{array}{ccc}
                      0 & m_{12} & m_{13} \\
                      0 & l_1m_{12} & l_1m_{13} \\
                      0 & l_2m_{12} & l_2m_{13} \\
                    \end{array}
                  \right)\left(
                           \begin{array}{ccc}
                             1 & 0 & 0 \\
                             0 & 1 & 0 \\
                             0 & 0 & \frac{m_{12}}{m_{13}} \\
                           \end{array}
                         \right)\\
           &=\left(
                           \begin{array}{ccc}
                             0 & m_{12} & m_{12} \\
                             0 & l_1m_{12} & l_1m_{12} \\
                             0 & l_2\sqrt{m_{12}m_{13}} & l_2\sqrt{m_{12}m_{13}} \\
                           \end{array}
                         \right)=N.
          \end{align*}
By Theorem \ref{iso}, $\mathcal{A}_{\mathcal{O}_{-1}(k^3)}(M)\cong \mathcal{A}_{\mathcal{O}_{-1}(k^3)}(N)$.

(2)If $m_{11}m_{12}>0,m_{11}m_{13}>0$, let
$$D=\left(
            \begin{array}{ccc}
              \sqrt{m_{12}m_{13}} & 0 & 0 \\
              0 & \sqrt{m_{11}m_{13}} & 0 \\
              0 & 0 & \sqrt{m_{11}m_{12}}\\
            \end{array}
          \right).$$
Then
          \begin{align*}
  \chi(M,D)=\left(
                                 \begin{array}{ccc}
                                   m_{11}\sqrt{m_{12}m_{13}} & m_{11}\sqrt{m_{12}m_{13}} & m_{11}\sqrt{m_{12}m_{13}} \\
                                   l_1m_{12}\sqrt{m_{11}m_{13}} & l_1m_{12}\sqrt{m_{11}m_{13}} & l_1m_{12}\sqrt{m_{11}m_{13}} \\
                                   l_2m_{13}\sqrt{m_{11}m_{12}} & l_2m_{13}\sqrt{m_{11}m_{12}} & l_2m_{13}\sqrt{m_{11}m_{12}} \\
                                 \end{array}
                               \right)
       =Q.
       \end{align*}
By Theorem \ref{iso}, $\mathcal{A}_{\mathcal{O}_{-1}(k^3)}(M)\cong \mathcal{A}_{\mathcal{O}_{-1}(k^3)}(Q)$.
\end{proof}

\begin{rem}
Note that the differential of $\mathcal{A}_{\mathcal{O}_{-1}(k^3)}(N)$ in Proposition \ref{noncycase}(1) is defined by
\begin{align*}
\begin{cases}
\partial_{\mathcal{A}}(x_1)=m_{12}(x_2^2+x_3^2)\\
\partial_{\mathcal{A}}(x_2)=l_1m_{12}(x_2^2+x_3^2)\\
\partial_{\mathcal{A}}(x_3)=l_2\sqrt{m_{12}m_{13}}(x_2^2+x_3^2),
\end{cases}
\end{align*}
where $l_1m_{12}=l_2m_{13}, l_1,l_2,m_{12},m_{13}\in k^{\times}$. Let $l_1=l_2=m_{12}=m_{13}=1$. Then $$N=\left(
                           \begin{array}{ccc}
                             0 & 1 & 1 \\
                             0 & 1 & 1 \\
                             0 & 1 & 1 \\
                           \end{array}
                         \right)$$ and we get a simple example $\mathcal{A}_{\mathcal{O}_{-1}(k^3)}(N)$, which is not homologically smooth but Koszul.
Similarly, the differential of $\mathcal{A}_{\mathcal{O}_{-1}(k^3)}(Q)$ in Proposition \ref{noncycase}(2) is defined by
\begin{align*}
\begin{cases}
\partial_{\mathcal{A}}(x_1)= m_{11}\sqrt{m_{12}m_{13}}(x_1^2+x_2^2+x_3^2)\\
\partial_{\mathcal{A}}(x_2)=l_1m_{12}\sqrt{m_{11}m_{13}}(x_1^2+x_2^2+x_3^2)\\
\partial_{\mathcal{A}}(x_3)=l_2m_{13}\sqrt{m_{11}m_{12}}(x_1^2+x_2^2+x_3^2),
\end{cases}
\end{align*}
where $m_{12}m_{13},m_{11}m_{13},m_{11}m_{12}>0, l_1l_2\neq 0$ and $$4m_{12}m_{13}l_1^2l_2^2= (m_{12}l_1^2+m_{13}l_2^2-m_{11})^2.$$ For example, let $l_1=m_{11}=m_{12}=m_{13}=1, l_2=2$, then $$Q=\left(
                           \begin{array}{ccc}
                             1 & 1 & 1 \\
                             1 & 1 & 1 \\
                             2 & 2 & 2 \\
                           \end{array}
                         \right)$$ and $\mathcal{A}_{\mathcal{O}_{-1}(k^3)}(Q)$ is a simple example of DG algebra which is not homologically smooth but Koszul.
\end{rem}
\begin{prop}\label{nonregsec}
The DG algebra $\mathcal{A}_{\mathcal{O}_{-1}(k^3)}(M)$ is not homologically smooth but Koszul when
 $$M=\left(
                                 \begin{array}{ccc}
                                   m_{11} & m_{12} & m_{13} \\
                                   l_1m_{11} & l_1m_{12} & l_1m_{13} \\
                                   l_2m_{11} & l_2m_{12} & l_2m_{13} \\
                                 \end{array}
                               \right),\,\, (m_{11},m_{12},m_{13})\neq 0,\,\, l_1l_2\neq 0,$$
 $m_{12}m_{13}=0$ and $m_{12}l_1^2+m_{13}l_2^2= m_{11}$. Furthermore, $\mathcal{A}_{\mathcal{O}_{-1}(k^3)}(M)\cong \mathcal{A}_{\mathcal{O}_{-1}(k^3)}(N)$, where $$N=\left(
                                 \begin{array}{ccc}
                                   1 & 1 & 0 \\
                                   1 & 1 & 0 \\
                                   1 & 1 & 0\\
                                 \end{array}
                               \right).$$
\end{prop}
\begin{proof}
By Proposition \ref{cohomology}(6), we have $$H(\mathcal{A}_{\mathcal{O}_{-1}(k^3)}(M))=\frac{k\langle \lceil l_1x_1-x_2\rceil, \lceil l_2x_1-x_3\rceil \rangle}{(m_{12}\lceil l_1x_1-x_2\rceil^2+m_{13}\lceil l_2x_1-x_3\rceil^2)}.$$
Since $(m_{11},m_{12},m_{13})\neq 0, m_{12}m_{13}=0$ and $m_{12}l_1^2+m_{13}l_2^2= m_{11}$, we have either $m_{12}=0,m_{13}\neq 0$ or $m_{12}\neq 0, m_{13}=0$. In both cases, the DG algebra $\mathcal{A}_{\mathcal{O}_{-1}(k^3)}(M)$ is
not homologically smooth but Koszul by Proposition \ref{nonhs}. Let $C=\left(
                           \begin{array}{ccc}
                             1 & 0 & 0 \\
                             0 & 0 & 1 \\
                             0 & 1 & 0 \\
                           \end{array}
                         \right)$.  Then
\begin{align*}
  \chi(M,C)=\left(
                                 \begin{array}{ccc}
                                   m_{11} & m_{13} & m_{12} \\
                                   l_2m_{11} & l_2m_{13} & l_2m_{12} \\
                                   l_1m_{11}& l_1m_{13}& l_1m_{12} \\
                                 \end{array}
                               \right).
       \end{align*}
This implies that we might as well assume that $m_{13}=0$ and $m_{12}\neq 0$. Then $m_{11}=m_{12}l_1^2$ and hence
 $$ M=\left(
                                 \begin{array}{ccc}
                                   m_{12}l_1^2 & m_{12} & 0 \\
                                   m_{12}l_1^3 & l_1m_{12} & 0 \\
                                   m_{12}l_1^2l_2 & l_2m_{12} & 0 \\
                                 \end{array}
                               \right). $$   Let $C'= \left(
                           \begin{array}{ccc}
                             \frac{1}{l_1^2m_{12}} & 0 & 0 \\
                             0 & \frac{1}{l_1m_{12}} & 0 \\
                             0 & 0 & \frac{l_1^2m_{12}}{l_2} \\
                           \end{array}
                         \right)$.   Then we have $
  \chi(M,C')=\left(
                                 \begin{array}{ccc}
                                   1 & 1 & 0 \\
                                   1 & 1 & 0 \\
                                   1 & 1 & 0\\
                                 \end{array}
                               \right)$.
Therefore,  $\mathcal{A}_{\mathcal{O}_{-1}(k^3)}(M)\cong \mathcal{A}_{\mathcal{O}_{-1}(k^3)}(N)$ where
$$N=\left(
                                 \begin{array}{ccc}
                                   1 & 1 & 0 \\
                                   1 & 1 & 0 \\
                                   1 & 1 & 0\\
                                 \end{array}
                               \right).$$

\end{proof}

It remains to consider the Calabi-Yau properties of $\mathcal{A}_{\mathcal{O}_{-1}(k^3)}(M)$ in the following $5$ cases:
\begin{enumerate}
\item  Case $1$: $r(M)=2$ and $s_1t_1^2+s_2t_2^2+s_3t_3^2= 0$ where $k(s_1,s_2,s_3)^T$ and $k(t_1,t_2,t_3)^T$ be the solution spaces of  homogeneous linear equations $MX=0$ and $M^TX=0$, respectively;
\item  Case $2$:
 $M=\left(
                                 \begin{array}{ccc}
                                   m_{11} & m_{12} & m_{13} \\
                                   l_1m_{11} & l_1m_{12} & l_1m_{13} \\
                                   l_2m_{11} & l_2m_{12} & l_2m_{13} \\
                                 \end{array}
                               \right)$, $(m_{11},m_{12},m_{13})\neq 0$,  with $m_{12}l_1^2+m_{13}l_2^2=m_{11}$, $l_1\neq 0$ and $l_2=0$;
\item  Case $3$: $M=\left(
                                 \begin{array}{ccc}
                                   m_{11} & m_{12} & m_{13} \\
                                   l_1m_{11} & l_1m_{12} & l_1m_{13} \\
                                   l_2m_{11} & l_2m_{12} & l_2m_{13} \\
                                 \end{array}
                               \right)$, $(m_{11},m_{12},m_{13})\neq 0$, with $m_{12}l_1^2+m_{13}l_2^2=m_{11}$, $l_1=0$ and $l_2\neq 0$;
\item  Case $4$: $M=\left(
                                 \begin{array}{ccc}
                                   m_{11} & m_{12} & m_{13} \\
                                   l_1m_{11} & l_1m_{12} & l_1m_{13} \\
                                   l_2m_{11} & l_2m_{12} & l_2m_{13} \\
                                 \end{array}
                               \right)$, $(m_{11},m_{12},m_{13})\neq 0$, with $m_{12}l_1^2+m_{13}l_2^2=m_{11}$, $l_1=0$ and $l_2=0$.
\end{enumerate}
The proof of each case involves further classifications and complicated analysis. The main ideas of our proof is to construct the minimal semi-free resolution of ${}_{\mathcal{A}}k$ in each case and compute the corresponding Ext-algebras. In the rest of this paper, we will allocate Section \ref{caseone} and Section \ref{casetwo} to discuss the homological properties of $\mathcal{A}_{\mathcal{O}_{-1}(k^3)}(M)$ for Case $1$ and Case $2-4$ separately.

\section{case $1$}\label{caseone}
By Proposition \ref{cohomology},  $H(\mathcal{A}_{\mathcal{O}_{-1}(k^3)}(M))$ in Case $1$  is not homologically smooth since it is
$k[\lceil t_1x_1 +t_2x_2+t_3x_3\rceil,\lceil s_1x_1^2+s_2x_2^2+s_3x_3^2\rceil ]/(\lceil t_1x_1 +t_2x_2+t_3x_3\rceil^2).$ From $H(\mathcal{A}_{\mathcal{O}_{-1}(k^3)}(M))$  we can't judge the  Calabi-Yau properties (resp. homologically smoothness) of $\mathcal{A}_{\mathcal{O}_{-1}(k^3)}(M)$. We turn to construct the minimal semi-free resolution of ${}_{\mathcal{A}}k$.
According to the constructing procedure of \cite[Proposition 2.4]{MW1}, we will construct the resolution as follows.

 Let $F_0=\mathcal{A}$ and $\varepsilon_0=\varepsilon: F_0=\mathcal{A}\to k$.
 Define $F_1$ as an extension of the  DG $\mathcal{A}$-module $F_0$ by $F_1^{\#}=F_0^{\#}\oplus \mathcal{A}^{\#}e_1$ and $\partial_{F_1}(e_1)=t_1x_1+t_2x_2+t_3x_3.$ Since $t_1^2x_1^2+t_2^2x_2^2+t_3^2x_3^2\in B^2(\mathcal{A})$, there exist $\sigma=q_1x_1+q_2x_2+q_3x_3 \in \mathcal{A}^1$ such that $\partial_{\mathcal{A}}(\sigma)=t_1^2x_1^2+t_2x_2^2+t_3x_3^2$.
 Define $F_2$ as an extension of the  DG $\mathcal{A}$-module $F_1$ by $F_2^{\#}=F_1^{\#}\oplus \mathcal{A}^{\#}e_2$ and $\partial_{F_2}(e_2)=(t_1x_1+t_2x_2+t_3x_3)e_1+\sigma.$

 Case $1.1$. If the condition $\textbf{C}1$:$q_1t_1x_1^2+q_2t_2x_2^2+q_3t_3x_3^2\not\in B^2(\mathcal{A})$ holds, then $q_1t_1x_1^2+q_2t_2x_2^2+q_3t_3x_3^2=\partial_{\mathcal{A}}(b)+s_1x_1^2+s_2x_2^2+s_3x_3^2$, for some $b\in \mathcal{A}^1$.
 For any cocycle element $a+a_1e_1+a_2e_2\in Z^1(F_2)$, we have
 \begin{align*}
 0&=\partial_{F_2}(a+a_1e_1+a_2e_2)\\
  &=\partial_{\mathcal{A}}(a)+\partial_{\mathcal{A}}(a_1)e_1-a_1(t_1x_1+t_2x_2+t_3x_3)+\partial_{\mathcal{A}}(a_2)e_2\\
  &\quad -a_2[(t_1x_1+t_2x_2+t_3x_3)e_1+\sigma] \\
  &=\partial_{\mathcal{A}}(a_2)e_2+[\partial_{\mathcal{A}}(a_1)-a_2(t_1x_1+t_2x_2+t_3x_3)]e_1\\
  &\quad +\partial_{\mathcal{A}}(a)-a_1(t_1x_1+t_2x_2+t_3x_3).
 \end{align*}
 This implies that
 \begin{align}\label{8.1}
 \begin{cases}
 \partial_{\mathcal{A}}(a_2)=0\\
 \partial_{\mathcal{A}}(a_1)-a_2(t_1x_1+t_2x_2+t_3x_3)=0\\
 \partial_{\mathcal{A}}(a)-a_1(t_1x_1+t_2x_2+t_3x_3)=0.
 \end{cases}
 \end{align}
Since $q_1t_1x_1^2+q_2t_2x_2^2+q_3t_3x_3^2\not\in B^2(\mathcal{A})$, one can easily check that (\ref{8.1}) implies that
\begin{align*}
\begin{cases}
a_2=0\\
a_1=b_1(t_1x_1+t_2x_2+t_3x_3)\\
a=b_1(q_1x_1+q_2x_2+q_3x_3)+b_0(t_1x_1+t_2x_2+t_3x_3)
\end{cases}
\end{align*}
 for some $b_1,b_0\in k$. So
$a+a_1e_1+a_2e_2=\partial_{F_2}(b_1e_2+b_0e_1)\in B^1(F_2)$. Hence $H^1(F_2)=0$. Furthermore,
 $F_2$ is the minimal semi-free resolution of ${}_{\mathcal{A}}k$ (see A.2. in the appendix). Note that
 $F_2^{\#}=\mathcal{A}^{\#}\oplus \mathcal{A}^{\#}e_1\oplus \mathcal{A}^{\#}e_2$ with a differential $\partial_{F_2}$ defined by
 $$\left(
                         \begin{array}{c}
                          \partial_{F_2}(1) \\
                          \partial_{F_2}(e_1)\\
                          \partial_{F_2}(e_2)
                         \end{array}
                       \right)=\left(
            \begin{array}{ccc}
              0 & 0& 0  \\
              \sum\limits_{i=1}^3t_ix_i & 0 & 0 \\
              \sigma & \sum\limits_{i=1}^3t_ix_i & 0 \\
            \end{array}
          \right)\left(
                         \begin{array}{c}
                          1 \\
                          e_1\\
                          e_2
                         \end{array}
                       \right).$$
  To make the paper more readable, we put the proof in the Appendix.

  We can list the following examples for Case $1.1$.
 \begin{ex}\label{ex8.1}
The DG algebra $\mathcal{A}_{\mathcal{O}_{-1}(k^3)}(M)$ belongs to Case $1.1$, when $M$ is one of the following matrixes:
 \begin{align*}
& (1) \left(
           \begin{array}{ccc}
             1 & 0 & 1 \\
             1 & 1 & 1 \\
             1 & 0 & 1 \\
           \end{array}
         \right), (2) \left(
           \begin{array}{ccc}
             0 & 1 & 0 \\
             1 & 0 & 1 \\
             1 & 1 & 1 \\
           \end{array}
         \right), (3) \left(
           \begin{array}{ccc}
             1 & 0 & 1 \\
             0 & 1 & 0 \\
             1 & 1 & 1 \\
           \end{array}
         \right),\\
         &(4) \left(
           \begin{array}{ccc}
             1 & 0 & 1 \\
             0 & 1 & 1 \\
             1 & 0 & 1 \\
           \end{array}
         \right),
(5) \left(
           \begin{array}{ccc}
             1 & 1 & 1 \\
             0 & 1 & 0 \\
             1 & 1 & 1 \\
           \end{array}
         \right),
(6) \left(
           \begin{array}{ccc}
             0 & 1 & 0 \\
             1 & 1 & 1 \\
             0 & 1 & 0 \\
           \end{array}
         \right),\\
&(7) \left(
           \begin{array}{ccc}
             1 & 0 & 1 \\
             0 & 1 & 0 \\
             1 & 0 & 1 \\
           \end{array}
         \right),
(8)\left(
           \begin{array}{ccc}
             1 & -1 & 0 \\
             1 & 1 & 1 \\
             1 & -1 & 1 \\
           \end{array}
         \right),
(9)\left(
           \begin{array}{ccc}
             1 & 1 & 1 \\
             -1 & 1 & 1 \\
             -1 & 1 & 1 \\
           \end{array}
         \right).
\end{align*}
 \end{ex}
 \begin{rem}
 Note that the matrixes appear in Example \ref{ex8.1} don't include all representatives of isomorphic class of matrices for which the algebra is in Case $1.1$. We list them for the reader to check and calculate specifically. The same comment holds for examples that appear below for other cases in this section.
 \end{rem}

Now, let us study the case that $q_1t_1x_1^2+q_2t_2x_2^2+q_3t_3x_3^2 \in B^2(\mathcal{A})$. We claim that we can divide it into the following two series:
\begin{itemize}
\item Case $1.2.*$,  when the condition $\textbf{C}2$: $(q_1t_1, q_2t_2, q_3t_3)^T=0$ holds;
\item Case $1.3.*$,  when we have $\textbf{C}2'$: $(q_1t_1, q_2t_2, q_3t_3)^T$ and $(t_1^2, t_2^2, t_3^2)^T$ are linearly independent.
\end{itemize}
Indeed, when $0\neq (q_1t_1, q_2t_2, q_3t_3)^T$ and $(t_1^2, t_2^2, t_3^2)^T$ are linearly dependent, we may as well let
$(q_1t_1, q_2t_2, q_3t_3)^T=c(t_1^2, t_2^2, t_3^2)^T$ for some $c\in k^{\times}$.
Let $q_1'=q_1-ct_1,q_2'=q_2-ct_2,q_3'=q_3-ct_3$. Then
\begin{align*}
\begin{cases}
q_1't_1+q_2't_2+q_3't_3=0\\
\partial_{\mathcal{A}}(q_1'x_1+q_2'x_2+q_3'x_3)=t_1^2x_1^2+t_2^2x_2^2+t_3^2x_3^2.
\end{cases}
\end{align*} We can replace $q_1,q_2,q_3$ by $q_1',q_2',q_3'$ in the construction.
\subsection{Case $1.2.*$}
 Since $q_1t_1x_1^2+q_2t_2x_2^2+q_3t_3x_3^2=0$,  we may choose $\tau=r_1x_1+r_2x_2+r_3x_3=0$, equivalently each $r_i=0$,  such that $\partial_{\mathcal{A}}(\tau)=0$. We label it ``Case $1.2.1$" when the conditions $\textbf{C}2$ and $\textbf{C}3$: $q_1^2x_1^2+q_2^2x_2^2+q_3^2x_3^2\not\in B^2(\mathcal{A})$ hold.
We extend $F_2$ to a semi-free DG module $F_3$ with $F_3^{\#}=F_2^{\#}\oplus \mathcal{A}^{\#}e_3$ and
$\partial_{F_3}(e_3)=(t_1x_1+t_2x_2+t_3x_3)e_2+\sigma e_1.$ We claim $H^1(F_3)=0$. Indeed, for any  cocycle element $a+a_1e_1+a_2e_2+a_3e_3\in Z^1(F_3)$, we have
 \begin{align*}
 0&=\partial_{F_3}(a+a_1e_1+a_2e_2+a_3e_3)\\
  &=\partial_{\mathcal{A}}(a)+\partial_{\mathcal{A}}(a_1)e_1-a_1(t_1x_1+t_2x_2+t_3x_3)+\partial_{\mathcal{A}}(a_2)e_2+\partial_{\mathcal{A}}(a_3)e_3\\
  &\quad -a_2[(t_1x_1+t_2x_2+t_3x_3)e_1+\sigma]-a_3[(t_1x_1+t_2x_2+t_3x_3)e_2+\sigma e_1] \\
  &=\partial_{\mathcal{A}}(a_3)e_3+[\partial_{\mathcal{A}}(a_2)-a_3(t_1x_1+t_2x_2+t_3x_3)]e_2\\
  &\quad +[\partial_{\mathcal{A}}(a_1)-a_3\sigma-a_2(t_1x_1+t_2x_2+t_3x_3)]e_1\\
  &\quad +\partial_{\mathcal{A}}(a)-a_1(t_1x_1+t_2x_2+t_3x_3)-a_2\sigma .
 \end{align*}
Then
\begin{align}\label{8.3}
 \begin{cases}
 \partial_{\mathcal{A}}(a_3)=0\\
 \partial_{\mathcal{A}}(a_2)-a_3(t_1x_1+t_2x_2+t_3x_3)=0\\
 \partial_{\mathcal{A}}(a_1)-a_3\sigma-a_2(t_1x_1+t_2x_2+t_3x_3)=0\\
\partial_{\mathcal{A}}(a)-a_1(t_1x_1+t_2x_2+t_3x_3)-a_2\sigma =0.
 \end{cases}
 \end{align}
Since $q_1^2x_1^2+q_2^2x_2^2+q_3^2x_3^2\not\in B^2(\mathcal{A})$,
it is easy to check that
(\ref{8.3}) implies,
\begin{align*}
\begin{cases}
a_3=0\\
a_2=c_2(t_1x_1+t_2x_2+t_3x_3)\\
a_1=c_2(q_1x_1+q_2x_2+q_3x_3)+c_1(t_1x_1+t_2x_2+t_3x_3)\\
a=c_1\sigma +c_0(t_1x_1+t_2x_2+t_3x_3),
\end{cases}
\end{align*}
for some $c_0,c_1$ and $c_2\in k$.  Then
$a+a_1e_1+a_2e_2+a_3e_3 =\partial_{\mathcal{A}}(c_2e_3+c_1e_2+c_0e_1)$.  So $H^1(F_3)=0$. Furthermore, we can show that
 $F_3$ is the minimal semi-free resolution of ${}_{\mathcal{A}}k$(see A.2 in the appendix).
 Note that
 $F_3^{\#}=\mathcal{A}^{\#}\oplus \mathcal{A}^{\#}e_1\oplus \mathcal{A}^{\#}e_2\oplus \mathcal{A}^{\#}e_3$ with a differential $\partial_{F_3}$ defined by
 $$\left(
                         \begin{array}{c}
                          \partial_{F_3}(1) \\
                          \partial_{F_3}(e_1)\\
                          \partial_{F_3}(e_2) \\
                          \partial_{F_3}(e_3)
                         \end{array}
                       \right)=\left(
            \begin{array}{cccc}
              0 & 0& 0 & 0 \\
              \sum\limits_{i=1}^3t_ix_i & 0 & 0 & 0\\
              \sigma & \sum\limits_{i=1}^3t_ix_i & 0 & 0 \\
              2\tau  & \sigma &\sum\limits_{i=1}^3t_ix_i  & 0
            \end{array}
          \right)\left(
                         \begin{array}{c}
                          1 \\
                          e_1\\
                          e_2\\
                          e_3
                         \end{array}
                       \right).$$
    We can list the following examples for Case $1.2.1$.
 \begin{ex}\label{ex8.3}
The DG algebra $\mathcal{A}_{\mathcal{O}_{-1}(k^3)}(M)$ belongs to Case $1.2.1$, when $M$ is one of the following matrixes:
 \begin{align*}
 & (1) \left(
           \begin{array}{ccc}
             1 & 1 & 0 \\
             1 & 0 & 1 \\
             1 & 1 & 0 \\
           \end{array}
         \right),(2) \left(
           \begin{array}{ccc}
             1 & 1 & 1 \\
             0 & 0 & 1 \\
             0 & 0 & 0 \\
           \end{array}
         \right),(3) \left(
           \begin{array}{ccc}
             1 & 0 & 0 \\
             1 & 0 & 1 \\
             1 & 0 & 0 \\
           \end{array}
         \right),\\
 &(4) \left(
           \begin{array}{ccc}
             0 & 0 & 1 \\
             0 & 1 & 0 \\
             0 & 0 & 0 \\
           \end{array}
         \right), (5) \left(
           \begin{array}{ccc}
             0 & 1 & 0 \\
             0 & 0 & 0 \\
             1 & 0 & 1 \\
           \end{array}
         \right), (6) \left(
           \begin{array}{ccc}
             1 & 1 & 0 \\
             0 & 0 & 0 \\
             1 & 0 & 0 \\
           \end{array}
         \right).
 \end{align*}
\end{ex}

If $\textbf{C}2$: $q_1t_1x_1^2+q_2t_2x_2^2+q_3t_3x_3^2=0$ and $\overline{\textbf{C}3}$: $q_1^2x_1^2+q_2^2x_2^2+q_3^2x_3^2\in B^2(\mathcal{A})$ hold, then  things will be different from Case $1.2.1$.
We must proceed our construction. Let $\lambda=u_1x_1+u_2x_2+u_3x_3$ such that $\partial_{\mathcal{A}}(\lambda)=q_1^2x_1^2+q_2^2x_2^2+q_3^2x_3^2$.
We label it ``Case $1.2.2$" when the conditions $\textbf{C}2$, $\overline{\textbf{C}3}$ and $\textbf{C}4$: $u_1t_1x_1^2+u_2t_2x_2^2+u_3t_3x_3^2\not\in B^2\mathcal{A}$ hold. We extend $F_3$ in Case $1.2.1$ to a semi-free DG module $F_4$ with $F_4^{\#}=F_3^{\#}\oplus \mathcal{A}^{\#}e_4$ and
$\partial_{F_4}(e_4)=(t_1x_1+t_2x_2+t_3x_3)e_3+\sigma e_2+\lambda.$ We claim $H^1(F_4)=0$.
Indeed, for any  cocycle element $a+a_1e_1+a_2e_2+a_3e_3+a_4e_4\in Z^1(F_4)$, we have
 \begin{align*}
 0&=\partial_{F_4}(a+a_1e_1+a_2e_2+a_3e_3+a_4e_4)\\
  &=\partial_{\mathcal{A}}(a)+\partial_{\mathcal{A}}(a_1)e_1-a_1(t_1x_1+t_2x_2+t_3x_3)+\partial_{\mathcal{A}}(a_2)e_2\\
  &\quad -a_2[(t_1x_1+t_2x_2+t_3x_3)e_1+\sigma] +\partial_{\mathcal{A}}(a_3)e_3-a_3[(t_1x_1+t_2x_2+t_3x_3)e_2+\sigma e_1]\\
  &\quad +\partial_{\mathcal{A}}(a_4)e_4-a_4[(t_1x_1+t_2x_2+t_3x_3)e_3+\sigma e_2+\lambda] \\
  &=\partial_{\mathcal{A}}(a_4)e_4+[\partial_{\mathcal{A}}(a_3)-a_4(t_1x_1+t_2x_2+t_3x_3)]e_3\\
  &\quad +[\partial_{\mathcal{A}}(a_2)-a_4\sigma-a_3(t_1x_1+t_2x_2+t_3x_3)]e_2\\
  &\quad +[\partial_{\mathcal{A}}(a_1)-a_2(t_1x_1+t_2x_2+t_3x_3)-a_3\sigma]e_1\\
  &\quad +\partial_{\mathcal{A}}(a)-a_1(t_1x_1+t_2x_2+t_3x_3)-a_2\sigma -a_4\lambda.
 \end{align*}
Then
\begin{align}\label{8.4}
 \begin{cases}
 \partial_{\mathcal{A}}(a_4)=0\\
 \partial_{\mathcal{A}}(a_3)-a_4(t_1x_1+t_2x_2+t_3x_3)=0\\
 \partial_{\mathcal{A}}(a_2)-a_4\sigma-a_3(t_1x_1+t_2x_2+t_3x_3)=0\\
 \partial_{\mathcal{A}}(a_1)-a_2(t_1x_1+t_2x_2+t_3x_3)-a_3\sigma=0\\
\partial_{\mathcal{A}}(a)-a_1(t_1x_1+t_2x_2+t_3x_3)-a_2\sigma -a_4\lambda=0.
 \end{cases}
 \end{align}
Since $u_1t_1x_1^2+u_2t_2x_2^2+u_3t_3x_3^2\not\in B^2\mathcal{A}$, (\ref{8.4}) implies
\begin{align*}
\begin{cases}
a_4=0\\
a_3=c_3(t_1x_1+t_2x_2+t_3x_3)\\
a_2=c_3(q_1x_1+q_2x_2+q_3x_3)+c_2(t_1x_1+t_2x_2+t_3x_3)\\
a_1=c_2(q_1x_1+q_2x_2+q_3x_3)+c_1(t_1x_1+t_2x_2+t_3x_3)\\
a=c_1(q_1x_1+q_2x_2+q_3x_3) +c_3(u_1x_1+u_2x_2+u_3x_3)+c_0(t_1x_1+t_2x_2+t_3x_3),
\end{cases}
\end{align*}
for some $c_0,c_1,c_2,c_3\in k$.
 Then
$$a+a_1e_1+a_2e_2+a_3e_3+a_4e_4 =\partial_{\mathcal{A}}(c_3e_4+c_2e_3+c_1e_2+c_0e_1).$$
 Hence $H^1(F_4)=0$. Furthermore,
 $F_4$ is the minimal semi-free resolution of ${}_{\mathcal{A}}k$ (see A.2 in the appendix).
  Note that
 $F_4^{\#}=\mathcal{A}^{\#}\oplus \mathcal{A}^{\#}e_1\oplus \mathcal{A}^{\#}e_2\oplus \mathcal{A}^{\#}e_3\oplus \mathcal{A}^{\#}e_4$ with a differential $\partial_{F_4}$ defined by
 $$\left(
                         \begin{array}{c}
                          \partial_{F_4}(1) \\
                          \partial_{F_4}(e_1)\\
                          \partial_{F_4}(e_2) \\
                          \partial_{F_4}(e_3) \\
                          \partial_{F_4}(e_4)
                         \end{array}
                       \right)=\left(
            \begin{array}{ccccc}
              0 & 0& 0 & 0 &0 \\
              \sum\limits_{i=1}^3t_ix_i & 0 & 0 & 0 &0 \\
              \sigma & \sum\limits_{i=1}^3t_ix_i & 0 & 0 &0\\
              2\tau  & \sigma &\sum\limits_{i=1}^3t_ix_i  & 0 &0\\
              \lambda & 2\tau & \sigma & \sum\limits_{i=1}^3t_ix_i & 0
            \end{array}
          \right)\left(
                         \begin{array}{c}
                          1 \\
                          e_1\\
                          e_2\\
                          e_3\\
                          e_4
                         \end{array}
                       \right).$$
  We can list the following examples for Case $1.2.2$.
 \begin{ex}\label{ex8.4}
The DG algebra $\mathcal{A}_{\mathcal{O}_{-1}(k^3)}(M)$ belongs to Case $1.2.2$, when $M$ is either one of the following matrixes:
 \begin{align*}
 (1) \left(
           \begin{array}{ccc}
             1 & 1 & 1 \\
             1 & 0 & 1 \\
             1 & 1 & 1 \\
           \end{array}
         \right),(2) \left(
           \begin{array}{ccc}
             0 & 1 & 0 \\
             1 & 0 & 1 \\
             0 & 1 & 0 \\
           \end{array}
         \right), (3) \left(
           \begin{array}{ccc}
             0 & 1 & 1 \\
             1 & 0 & 0 \\
             1 & 0 & 0 \\
           \end{array}
         \right).
 \end{align*}
\end{ex}

If the conditions $\textbf{C}2, \overline{\textbf{C}3}$ and $\overline{\textbf{C}4}$: $u_1t_1x_1^2+u_2t_2x_2^2+u_3t_3x_3^2\in B^2(\mathcal{A})$ hold, then   then we must continue the process after the construction of $F_4$ in Case $1.2.2$.
Let $\omega=v_1x_1+v_2x_2+v_3x_3$ such that $\partial_{\mathcal{A}}(\omega)=u_1t_1x_1^2+u_2t_2x_2^2+u_3t_3x_3^2$.
We label it ``Case $1.2.3$" when $\textbf{C}2, \overline{\textbf{C}3}$, $\overline{\textbf{C}4}$ and the condition $\textbf{C}5$:
$$(4v_1t_1+2q_1u_1)x_1^2+(4v_2t_2+2q_2u_2)x_2^2+(4v_3t_3+2q_3u_3)x_3^2\not\in B^2\mathcal{A}$$  hold. We extend $F_4$ in Case $1.2.2$ to a semi-free DG module $F_5$ such that $F_5^{\#}=F_4^{\#}\oplus \mathcal{A}^{\#}e_5$ and
$\partial_{F_5}(e_5)=(t_1x_1+t_2x_2+t_3x_3)e_4+ \sigma e_3+\lambda e_1+2\omega.$ We claim $H^1(F_5)=0$.
Indeed, for any $a+a_1e_1+a_2e_2+a_3e_3+a_4e_4+a_5e_5\in Z^1(F_5)$, we have
 \begin{align*}
 0&=\partial_{F_5}(a+a_1e_1+a_2e_2+a_3e_3+a_4e_4+a_5e_5)\\
  &=\partial_{\mathcal{A}}(a)+\partial_{\mathcal{A}}(a_1)e_1-a_1(t_1x_1+t_2x_2+t_3x_3)+\partial_{\mathcal{A}}(a_2)e_2\\
  &\quad -a_2[(t_1x_1+t_2x_2+t_3x_3)e_1+\sigma] +\partial_{\mathcal{A}}(a_3)e_3-a_3[(t_1x_1+t_2x_2+t_3x_3)e_2+\sigma e_1]\\
  &\quad +\partial_{\mathcal{A}}(a_4)e_4-a_4[(t_1x_1+t_2x_2+t_3x_3)e_3+\sigma e_2+\lambda]+\partial_{\mathcal{A}}(a_5)e_5 \\
  &\quad -a_5[(t_1x_1+t_2x_2+t_3x_3)e_4+ \sigma e_3+\lambda e_1+2\omega ]\\
  &=\partial_{\mathcal{A}}(a_5)e_5 +[\partial_{\mathcal{A}}(a_4)-a_5(t_1x_1+t_2x_2+t_3x_3)]e_4\\
  &\quad +[\partial_{\mathcal{A}}(a_3)-a_5\sigma -a_4(t_1x_1+t_2x_2+t_3x_3)]e_3\\
  &\quad +[\partial_{\mathcal{A}}(a_2)-a_4\sigma-a_3(t_1x_1+t_2x_2+t_3x_3)]e_2\\
  &\quad +[\partial_{\mathcal{A}}(a_1)-a_2(t_1x_1+t_2x_2+t_3x_3)-a_3\sigma-a_5\lambda]e_1\\
  &\quad +\partial_{\mathcal{A}}(a)-a_1(t_1x_1+t_2x_2+t_3x_3)-a_2\sigma -a_4\lambda-2a_5\omega.
 \end{align*}
Then
\begin{align}\label{8.4}
 \begin{cases}
 \partial_{\mathcal{A}}(a_5)=0\\
 \partial_{\mathcal{A}}(a_4)-a_5(t_1x_1+t_2x_2+t_3x_3)=0\\
 \partial_{\mathcal{A}}(a_3)-a_5\sigma-a_4(t_1x_1+t_2x_2+t_3x_3)=0\\
 \partial_{\mathcal{A}}(a_2)-a_4\sigma-a_3(t_1x_1+t_2x_2+t_3x_3)=0\\
\partial_{\mathcal{A}}(a_1)-a_2(t_1x_1+t_2x_2+t_3x_3)-a_3\sigma-a_5\lambda=0\\
\partial_{\mathcal{A}}(a)-a_1(t_1x_1+t_2x_2+t_3x_3)-a_2\sigma -a_4\lambda-2a_5\omega.
 \end{cases}
 \end{align}
Since $(4v_1t_1+2q_1u_1)x_1^2+(4v_2t_2+2q_2u_2)x_2^2+(4v_3t_3+2q_3u_3)x_3^2\not\in B^2(\mathcal{A})$, (\ref{8.4}) implies
\begin{align*}
\begin{cases}
a_5=0\\
a_4=c_4(t_1x_1+t_2x_2+t_3x_3)\\
a_3=c_4(q_1x_1+q_2x_2+q_3x_3)+c_3(t_1x_1+t_2x_2+t_3x_3)\\
a_2=c_3(q_1x_1+q_2x_2+q_3x_3)+c_2(t_1x_1+t_2x_2+t_3x_3)\\
a_1=c_2(q_1x_1+q_2x_2+q_3x_3)+c_4(u_1x_1+u_2x_2+u_3x_3)+c_1(t_1x_1+t_2x_2+t_3x_3)\\
a=c_1(q_1x_1+q_2x_2+q_3x_3) +c_3(u_1x_1+u_2x_2+u_3x_3)+2c_4(v_1x_1+v_2x_2+v_3x_3)\\
\quad + c_0(t_1x_1+t_2x_2+t_3x_3)
\end{cases}
\end{align*}
for some $c_0,c_1,c_2,c_3,c_4\in k$.
 Then
$$a+a_1e_1+a_2e_2+a_3e_3+a_4e_4+a_5e_5 =\partial_{\mathcal{A}}(c_4e_5+c_3e_4+c_2e_3+c_1e_2+c_0e_1).$$
 Hence $H^1(F_5)=0$. Furthermore,
 $F_5$ is the minimal semi-free resolution of ${}_{\mathcal{A}}k$ (see A.2 in the appendix). Note that
 $$F_5^{\#}=\mathcal{A}^{\#}\oplus \mathcal{A}^{\#}e_1\oplus \mathcal{A}^{\#}e_2\oplus \mathcal{A}^{\#}e_3\oplus \mathcal{A}^{\#}e_4\oplus \mathcal{A}^{\#}e_5$$ with a differential defined by
 $$\left(
                         \begin{array}{c}
                          \partial_{F_5}(1) \\
                          \partial_{F_5}(e_1)\\
                          \partial_{F_5}(e_2) \\
                          \partial_{F_5}(e_3) \\
                          \partial_{F_5}(e_4) \\
                          \partial_{F_5}(e_5)
                         \end{array}
                       \right)= \left(
            \begin{array}{cccccc}
              0 & 0& 0 & 0 &0 & 0 \\
              \sum\limits_{i=1}^3t_ix_i & 0 & 0 & 0 &0 &0 \\
              \sigma & \sum\limits_{i=1}^3t_ix_i & 0 & 0 &0 &0 \\
              2\tau  & \sigma &\sum\limits_{i=1}^3t_ix_i  & 0 &0 &0\\
              \lambda & 2\tau & \sigma & \sum\limits_{i=1}^3t_ix_i & 0 & 0\\
              2\omega & \lambda & 2\tau & \sigma & \sum\limits_{i=1}^3t_ix_i & 0
            \end{array}
          \right) \left(
                         \begin{array}{c}
                          1 \\
                          e_1\\
                          e_2\\
                          e_3\\
                          e_4\\
                          e_5
                         \end{array}
                       \right).$$
  We can list the following example for Case $1.2.3$.
 \begin{ex}\label{ex8.5}
The DG algebra $\mathcal{A}_{\mathcal{O}_{-1}(k^3)}(M)$ belongs to Case $1.2.3$, when
 \begin{align*}
 M=\left(
           \begin{array}{ccc}
             1 & 1 & 1 \\
             0 & 0 & 0 \\
             1 & 0 & 1 \\
           \end{array}
         \right).
 \end{align*}
\end{ex}
If the conditions $\textbf{C}2, \overline{\textbf{C}3}$, $\overline{\textbf{C}4}$ and $\overline{\textbf{C}5}$: $$(4v_1t_1+2q_1u_1)x_1^2+(4v_2t_2+2q_2u_2)x_2^2+(4v_3t_3+2q_3u_3)x_3^2\in B^2\mathcal{A}$$ hold,
then we must continue the process after the construction of $F_5$ in Case $1.2.3$. Since $(t_1,t_2,t_3)\neq 0$ and $\partial_{\mathcal{A}}(q_1x_1+q_2x_2+q_3x_3)=t_1^2x_1^2+t_2^2x_2^2+t_3^2x_3^2\neq 0$, we have $(q_1,q_2,q_3)\neq 0$. Hence $q_1t_1x_1^2+q_2t_2x_2^2+q_3t_3x_3^2=0$ implies that there exist one or two nonzero elements in $\{t_1,t_2,t_3\}$. By symmetry, we only need to consider the following three cases:
$$\text{Case A:}\begin{cases}
t_1=0,q_1\neq 0\\
t_2\neq 0, q_2=0\\
t_3\neq 0,q_3=0
\end{cases}\quad  \text{Case B:}\begin{cases}
t_1=0,q_1\neq 0\\
t_2=0, q_2\neq 0\\
t_3\neq 0,q_3=0
\end{cases}\quad \text{Case C:}\begin{cases}
t_1=0,q_1= 0\\
t_2=0, q_2\neq 0\\
t_3\neq 0,q_3=0.
\end{cases}$$

If Case $A$ happens, then $\partial_{\mathcal{A}}(t_2x_2+t_3x_3)=0$, $\partial_{\mathcal{A}}(q_1x_1)=t_2^2x_2^2+t_3^2x_3^2$ and $\partial_{\mathcal{A}}(u_1x_1+u_2x_2+u_3x_3)=q_1^2x_1^2$. So
$B^2(\mathcal{A})=k(t_2^2x_2^2+t_3^2x_3^2)\oplus kx_1^2$. Since $$\partial_{\mathcal{A}}(v_1x_1+v_2x_2+v_3x_3)=u_2t_2x_2^2+u_3t_3x_3^2\in B^2(\mathcal{A}),$$  we have $u_2=lt_2,u_3=lt_3$, for some $l\in k$.
If $u_1=0$, then $$\partial_{\mathcal{A}}(u_2x_2+u_3x_3)=l\partial_{\mathcal{A}}(t_2x_2+t_3x_3)=0$$ which contradicts with $\partial_{\mathcal{A}}(u_2x_2+u_3x_3)=q_1^2x_1^2\neq 0$. Hence $u_1\neq 0$. Then
\begin{align*}
\partial_{\mathcal{A}}(u_1x_1+u_2x_2+u_3x_3)&=u_1\partial_{\mathcal{A}}(x_1)+l\partial_{\mathcal{A}}(t_2x_2+t_3x_3)\\
                                            &=\frac{u_1}{q_1}[t_2^2x_2^2+t_3^2x_3^2],
\end{align*}
which contradicts with the assumption $\partial_{\mathcal{A}}(u_1x_1+u_2x_2+u_3x_3)=q_1^2x_1^2$. Hence Case $A$ is impossible to occur.

If Case $B$ happens,  then we have $\partial_{\mathcal{A}}(x_3)=0$, $\partial_{\mathcal{A}}(q_1x_1+q_2x_2)=t_3^2x_3^2$ and $$\partial_{\mathcal{A}}(u_1x_1+u_2x_2+u_3x_3)=\partial_{\mathcal{A}}(u_1x_1+u_2x_2)=q_1^2x_1^2+q_2^2x_2^2.$$ So $B^2(\mathcal{A})=kx_3^2\oplus k(q_1^2x_1^2+q_2^2x_2^2)$.
We have $\partial_{\mathcal{A}}(v_1x_1+v_2x_2+v_3x_3)=u_3t_3x_3^2$. Since $\partial_{\mathcal{A}}(x_3)=0$, we may choose $v_3=0$.  If $u_3\neq 0$, then $(v_1,v_2)\neq 0$ and $(v_1,v_2)=(\frac{u_3}{t_3}q_1, \frac{u_3}{t_3}q_2)$ since $Z^1(\mathcal{A})=kx_3$. Then
$$(4v_1t_1+2q_1u_1)x_1^2+(4v_2t_2+2q_2u_2)x_2^2+(4v_3t_3+2q_3u_3)x_3^2=2q_1u_1x_1^2+2q_2u_2x_2^2\in B^2\mathcal{A}.$$
So $u_1=cq_1,u_2=cq_2$ for some $c\in k$. But then we have
\begin{align*}
\partial_{\mathcal{A}}(u_1x_1+u_2x_2+u_3x_3) &=\partial_{\mathcal{A}}(cq_1x_1+cq_2x_2)+u_3\partial_{\mathcal{A}}(x_3) \\
&=ct_3^2x_3^2,
\end{align*}
which contradicts with the assumption that $\partial_{\mathcal{A}}(u_1x_1+u_2x_2+u_3x_3)=q_1^2x_1^2+q_2^2x_2^2.$
Hence $u_3=0$. So $u_1t_1x_1^2+u_2t_2x_2^2+u_3t_3x_3^2=0$. We may choose $v_1=v_2=v_3=0$.
Then \begin{align*}
&2q_1u_1x_1^2+2q_2u_2x_2^2\\
&=(4v_1t_1+2q_1u_1)x_1^2+(4v_2t_2+2q_2u_2)x_2^2+(4v_3t_3+2q_3u_3)x_3^2\in B^2\mathcal{A}.
\end{align*}
Since $B^2(\mathcal{A})=kx_3^2\oplus k(q_1^2x_1^2+q_2^2x_2^2)$, we get $u_1=lq_1,u_2=lq_2$ for some $l\in k$. Then
$\partial_{\mathcal{A}}(u_1x_1+u_2x_2)=l\partial_{\mathcal{A}}(q_1x_1+q_2x_2)=lt_3^2x_3^2$, and we reach a contradiction with the assumption that $\partial_{\mathcal{A}}(u_1x_1+u_2x_2)=\partial_{\mathcal{A}}(u_1x_1+u_2x_2+u_3x_3)=q_1^2x_1^2+q_2^2x_2^2.$ Therefore, Case $B$ is also impossible to occur.

If Case $C$ happens, then  we have $\partial_{\mathcal{A}}(x_3)=0$, $\partial_{\mathcal{A}}(q_2x_2)=t_3^2x_3^2$, and $$\partial_{\mathcal{A}}(u_1x_1+u_2x_2+u_3x_3)=q_2^2x_2^2.$$ So $B^2(\mathcal{A})=kx_2^2\oplus kx_3^2$.
Since $\partial_{\mathcal{A}}(x_3)=0$, we may choose $u_3=0$ in construction. We claim that $u_1\neq 0$. Indeed, if $u_1=0$, then $\partial_{\mathcal{A}}(u_2x_2)=\partial_{\mathcal{A}}(u_1x_1+u_2x_2+u_3x_3)=q_2x_2^2$, which contradicts with the assumption that $\partial_{\mathcal{A}}(q_2x_2)=t_3^2x_3^2$. So $u_1\neq 0$. In the construction, we can choose $v_1=v_2=v_3=0$ since $u_1t_1x_1^2+u_2t_2x_2^2+u_3t_3x_3^2=0$. Then $(4v_1t_1+2q_1u_1)x_1^2+(4v_2t_2+2q_2u_2)x_2^2+(4v_3t_3+2q_3u_3)x_3^2=2q_2u_2x_2^2$. We may choose $$w_1=\frac{2u_2u_1}{q_2}, w_2=\frac{2u_2^2}{q_2},w_3=0$$ and $\eta=w_1x_1+w_2x_2+w_3x_3$ such that $\partial_{\mathcal{A}}(\eta)=2q_2u_2x_2^2$. We extend $F_5$ in Case $1.2.3$ to a semi-free DG module $F_6$ with $F_6^{\#}=F_5^{\#}\oplus \mathcal{A}^{\#}e_6$ and
$$\partial_{F_5}(e_6)=(t_3x_3)e_5+ \sigma e_4+\lambda e_2+\eta.$$
It is straightforward for one to show
that $\partial_{F_6}[(t_3x_3)e_6+\sigma e_5+\lambda e_3+\eta e_1]=0.$ So $H^1(F_6)\neq 0$.
We extend $F_6$ to a semi-free DG module $F_7$ with $F_7^{\#}=F_6^{\#}\oplus \mathcal{A}^{\#}e_7$ and
$\partial_{F_7}(e_7)=(t_3x_3)e_6+\sigma e_5+\lambda e_3+\eta e_1$. We claim $H^1(F_7)=0$.
Indeed, for any  cocycle element $a+a_1e_1+a_2e_2+a_3e_3+a_4e_4+a_5e_5+a_6e_6+a_7e_7\in Z^1(F_7)$, we have
 \begin{align*}
 0&=\partial_{F_7}(a+a_1e_1+a_2e_2+a_3e_3+a_4e_4+a_5e_5+a_6e_6+a_7e_7)\\
  &=\partial_{\mathcal{A}}(a)+\partial_{\mathcal{A}}(a_1)e_1-a_1(t_3x_3)+\partial_{\mathcal{A}}(a_2)e_2-a_2[(t_3x_3)e_1+\sigma]+\partial_{\mathcal{A}}(a_3)e_3\\
  &\quad -a_3[(t_3x_3)e_2+\sigma e_1]+\partial_{\mathcal{A}}(a_4)e_4-a_4[(t_3x_3)e_3+\sigma e_2+\lambda]+\partial_{\mathcal{A}}(a_5)e_5\\
  &\quad -a_5[(t_3x_3)e_4+ \sigma e_3+\lambda e_1 ]+\partial_{\mathcal{A}}(a_6)e_6-a_6[(t_3x_3)e_5+ \sigma e_4+\lambda e_2+\eta]  \\
  &\quad +\partial_{\mathcal{A}}(a_7)e_7-a_7[(t_3x_3)e_6+\sigma e_5+\lambda e_3+\eta e_1]\\
  &=\partial_{\mathcal{A}}(a_7)e_7+[\partial_{\mathcal{A}}(a_6)-a_7(t_3x_3)]e_6+[\partial_{\mathcal{A}}(a_5)-a_6(t_3x_3)-a_7\sigma]e_5 \\
  &\quad +[\partial_{\mathcal{A}}(a_4)-a_5(t_3x_3)-a_6\sigma]e_4+[\partial_{\mathcal{A}}(a_3)-a_5\sigma -a_4(t_3x_3)-a_7\lambda]e_3\\
  &\quad +[\partial_{\mathcal{A}}(a_2)-a_4\sigma-a_3(t_3x_3)-a_6\lambda]e_2+[\partial_{\mathcal{A}}(a_1)-a_2(t_3x_3)-a_3\sigma-a_5\lambda]e_1\\
  &\quad +\partial_{\mathcal{A}}(a)-a_1(t_3x_3)-a_2\sigma -a_4\lambda-a_6\eta.
\end{align*}
Then
\begin{align}\label{8.5}
 \begin{cases}
 \partial_{\mathcal{A}}(a_7)=0\\
 \partial_{\mathcal{A}}(a_6)-a_7(t_3x_3)=0\\
 \partial_{\mathcal{A}}(a_5)-a_7\sigma-a_6(t_3x_3)=0\\
 \partial_{\mathcal{A}}(a_4)-a_6\sigma-a_5(t_3x_3)=0\\
\partial_{\mathcal{A}}(a_3)-a_5\sigma -a_4(t_3x_3)-a_7\lambda=0\\
\partial_{\mathcal{A}}(a_2)-a_4\sigma-a_3(t_3x_3)-a_6\lambda=0\\
\partial_{\mathcal{A}}(a_1)-a_2(t_3x_3)-a_3\sigma-a_5\lambda=0\\
\partial_{\mathcal{A}}(a)-a_1(t_3x_3)-a_2\sigma -a_4\lambda-a_6\eta=0.
 \end{cases}
 \end{align}
Since $u_1\neq 0$, we have $u_1^2x_1^2+3u_2^2x_2^2\not\in B^2(\mathcal{A})$. Then (\ref{8.5}) implies that
\begin{align*}
\begin{cases}
a_7=0\\
a_6=c_6t_3x_3\\
a_5=c_6q_2x_2+c_5t_3x_3\\
a_4=c_5q_2x_2+c_4t_3x_3\\
a_3=c_4q_2x_2+c_6(u_1x_1+u_2x_2)+c_3t_3x_3\\
a_2=c_5(u_1x_1+u_2x_2)+c_3q_2x_2+c_2t_3x_3\\
a_1=c_4(u_1x_1+u_2x_2)+c_2q_2x_2+c_6\eta +c_1t_3x_3\\
a=c_3(u_1x_1+u_2x_2)+c_1q_2x_2+c_5\eta +c_0t_3x_3,
\end{cases}
\end{align*}
for some $c_0,c_1,c_2,c_3,c_4,c_5,c_6\in k$.
\begin{align*}
&\quad a+a_1e_1+a_2e_2+a_3e_3+a_4e_4+a_5e_5+a_6e_6+a_7e_7\\
&=\partial_{F_7}(c_6e_7+c_5e_6+c_4e_5+c_3e_4+c_2e_3+c_1e_2+c_0e_1).
\end{align*}
Hence $H^1(F_7)=0$. Furthermore,
 $F_7$ is the minimal semi-free resolution of ${}_{\mathcal{A}}k$ (see A.2 in the appendix).
 Note that $$F_7^{\#}=\mathcal{A}^{\#}\oplus \mathcal{A}^{\#}e_1\oplus \mathcal{A}^{\#}e_2\oplus \mathcal{A}^{\#}e_3\oplus \mathcal{A}^{\#}e_4\oplus \mathcal{A}^{\#}e_5\oplus \mathcal{A}^{\#}e_6\oplus \mathcal{A}^{\#}e_7$$
and
$$\left(
                         \begin{array}{c}
                          \partial_{F_7}(1) \\
                          \partial_{F_7}(e_1)\\
                          \partial_{F_7}(e_2)\\
                          \partial_{F_7}(e_3)\\
                          \partial_{F_7}(e_4)\\
                          \partial_{F_7}(e_5)\\
                          \partial_{F_7}(e_6)\\
                          \partial_{F_7}(e_7)
                         \end{array}
                       \right) =\left(
            \begin{array}{cccccccc}
              0 & 0& 0 & 0 &0 & 0 &0 & 0\\
               t_3x_3  & 0 & 0 & 0 &0 &0&0 & 0 \\
              \sigma &  t_3x_3  & 0 & 0 &0 &0&0 & 0 \\
              0 & \sigma & t_3x_3   & 0 &0 &0&0 & 0\\
              \lambda & 0 & \sigma &  t_3x_3  & 0 & 0&0 & 0\\
              0   & \lambda & 0       & \sigma &  t_3x_3  & 0      &0 & 0\\
             \eta & 0       & \lambda & 0      & \sigma   & t_3x_3 &0 & 0\\
             0    & \eta    & 0       &\lambda &  0       &\sigma  &t_3x_3 & 0
            \end{array}
          \right)\left(
                         \begin{array}{c}
                          1\\
                          e_1\\
                          e_2\\
                          e_3\\
                          e_4\\
                          e_5\\
                          e_6\\
                          e_7
                         \end{array}
                       \right).$$
  For the convenience of future talk, we rename Case $C$ to Case $1.2.4$.
 We can list the following two examples for Case $1.2.4$.
\begin{ex}\label{ex8.5}
The DG algebra $\mathcal{A}_{\mathcal{O}_{-1}(k^3)}(M)$ belongs to Case $1.2.4$, when $M$ is either one of the following matrixes:
 \begin{align*}
 \left(
           \begin{array}{ccc}
             0 & 1 & 0 \\
             0 & 0 & 1 \\
             0 & 0 & 0 \\
           \end{array}
         \right),  \left(
           \begin{array}{ccc}
             0 & 1 & 1 \\
             0 & 0 & 1 \\
             0 & 0 & 0 \\
           \end{array}
         \right).
 \end{align*}
\end{ex}

\subsection{Case $1.3.*$}
 By assumption,  $q_1t_1x_1^2+q_2t_2x_2^2+q_3t_3x_3^2$ and $t_1^2x_1^2+t_2^2x_2^2+t_3^2x_3^2$ constitute a basis of  $B^2(\mathcal{A})$.  Let $\tau=r_1x_1+r_2x_2+r_3x_3\in \mathcal{A}^1$ such that $\partial_{\mathcal{A}}(\tau)=q_1t_1x_1^2+q_2t_2x_2^2+q_3t_3x_3^2$.
 We label it ``Case $1.3.1$" when the conditions $\textbf{C}2'$ and
 $\textbf{C}3'$: $\sum\limits_{i=1}^3(4r_it_i+q_i^2)x_i^2\not\in B^2(\mathcal{A})$ hold.
We extend $F_2$ to a semi-free DG module $F_3$ with $$F_3^{\#}=F_2^{\#}\oplus \mathcal{A}^{\#}e_3\,\, \text{and}\,\,\partial_{F_3}(e_3)=(t_1x_1+t_2x_2+t_3x_3)e_2+\sigma e_1+2\tau.$$ We claim $H^1(F_3)=0$. Indeed,
for any $a+a_1e_1+a_2e_2+a_3e_3\in Z^1(F_3)$, we have
 \begin{align*}
 0&=\partial_{F_3}(a+a_1e_1+a_2e_2+a_3e_3)\\
  &=\partial_{\mathcal{A}}(a)+\partial_{\mathcal{A}}(a_1)e_1-a_1(t_1x_1+t_2x_2+t_3x_3)+\partial_{\mathcal{A}}(a_2)e_2+\partial_{\mathcal{A}}(a_3)e_3\\
  &\quad -a_2[(t_1x_1+t_2x_2+t_3x_3)e_1+\sigma]-a_3[(t_1x_1+t_2x_2+t_3x_3)e_2+\sigma e_1+2\tau] \\
  &=\partial_{\mathcal{A}}(a_3)e_3+[\partial_{\mathcal{A}}(a_2)-a_3(t_1x_1+t_2x_2+t_3x_3)]e_2\\
  &\quad +[\partial_{\mathcal{A}}(a_1)-a_3\sigma-a_2(t_1x_1+t_2x_2+t_3x_3)]e_1\\
  &\quad +\partial_{\mathcal{A}}(a)-a_1(t_1x_1+t_2x_2+t_3x_3)-a_2\sigma -2a_3\tau.
 \end{align*}
Then
\begin{align}\label{8.2}
 \begin{cases}
 \partial_{\mathcal{A}}(a_3)=0\\
 \partial_{\mathcal{A}}(a_2)-a_3(t_1x_1+t_2x_2+t_3x_3)=0\\
 \partial_{\mathcal{A}}(a_1)-a_3\sigma-a_2(t_1x_1+t_2x_2+t_3x_3)=0\\
\partial_{\mathcal{A}}(a)-a_1(t_1x_1+t_2x_2+t_3x_3)-a_2\sigma -2a_3\tau=0.
 \end{cases}
 \end{align}
Since $(4r_1t_1+q_1^2)x_1^2+(4r_2t_2+q_2^2)x_2^2+(4r_3t_3+q_3^2)x_3^2\not\in B^2(\mathcal{A})$, it is easy to check that
(\ref{8.2}) implies ,
\begin{align*}
\begin{cases}
a_3=0\\
a_2=c_2(t_1x_1+t_2x_2+t_3x_3)\\
a_1=c_2(q_1x_1+q_2x_2+q_3x_3)+c_1(t_1x_1+t_2x_2+t_3x_3)\\
a=2c_2\tau+c_1\sigma +c_0(t_1x_1+t_2x_2+t_3x_3),
\end{cases}
\end{align*}
for some $c_0,c_1$ and $c_2\in k$. Then
$a+a_1e_1+a_2e_2+a_3e_3 =\partial_{\mathcal{A}}(c_2e_3+c_1e_2+c_0e_1)$.
 Hence $H^1(F_3)=0$. Furthermore,
 $F_3$ is the minimal semi-free resolution of ${}_{\mathcal{A}}k$ (see A.2 in the appendix). Note that
 $F_3^{\#}=\mathcal{A}^{\#}\oplus \mathcal{A}^{\#}e_1\oplus \mathcal{A}^{\#}e_2\oplus \mathcal{A}^{\#}e_3$ with a differential $\partial_{F_3}$ defined by
 $$\left(
                         \begin{array}{c}
                          \partial_{F_3}(1) \\
                          \partial_{F_3}(e_1)\\
                          \partial_{F_3}(e_2) \\
                          \partial_{F_3}(e_3)
                         \end{array}
                       \right)= \left(
            \begin{array}{cccc}
              0 & 0& 0 & 0 \\
              \sum\limits_{i=1}^3t_ix_i & 0 & 0 & 0\\
              \sigma & \sum\limits_{i=1}^3t_ix_i & 0 & 0 \\
              2\tau  & \sigma &\sum\limits_{i=1}^3t_ix_i  & 0
            \end{array}
          \right) \left(
                         \begin{array}{c}
                          1 \\
                          e_1\\
                          e_2\\
                          e_3
                         \end{array}
                       \right).$$
 We can list the following examples for Case $1.3.1$.
 \begin{ex}\label{ex8.2}
The DG algebra $\mathcal{A}_{\mathcal{O}_{-1}(k^3)}(M)$ belongs to Case $1.3.1$, when $M$ is one of the following matrixes:
 \begin{align*}
& (1) \left(
           \begin{array}{ccc}
             1 & 0 & 1 \\
             1 & 1 & 1 \\
             0 & 1 & 0 \\
           \end{array}
         \right),(2) \left(
           \begin{array}{ccc}
             1 & 0 & 0 \\
             0 & 0 & 1 \\
             1 & 0 & 0 \\
           \end{array}
         \right),(3) \left(
           \begin{array}{ccc}
             1 & 1 & 1 \\
             0 & 1 & 1 \\
             1 & 0 & 0 \\
           \end{array}
         \right).
 \end{align*}
\end{ex}

 We label it ``Case $1.3.2$" when the condition $\textbf{C}2'$ and $\overline{\textbf{C}3'}$: $\sum\limits_{i=1}^3(4r_it_i+q_i^2)x_i^2\in B^2(\mathcal{A})$ hold. Then things will be different from Case $1.3.1$.
We must proceed our process after constructing $F_3$ in Case $1.3.1$. Let $\lambda=u_1x_1+u_2x_2+u_3x_3$ such that $$\partial_{\mathcal{A}}(\lambda)=(4r_1t_1+q_1^2)x_1^2+(4r_2t_2+q_2^2)x_2^2+(4r_3t_3+q_3^2)x_3^2.$$ We extend $F_3$ in Case $1.3.1$ to a semi-free DG module $F_4$ with $F_4^{\#}=F_3^{\#}\oplus \mathcal{A}^{\#}e_4$ and
$\partial_{F_4}(e_4)=(t_1x_1+t_2x_2+t_3x_3)e_3+\sigma e_2+\lambda.$  In order to get a minimal semi-free resolution of $k$, we should proceed our construction by extending $F_4$. For this, we need some analysis first.

\begin{prop}\label{simpleprop}
Let $M=(m_{ij})_{3 \times 3}$  be a matrix which satisfies the the following conditions:

\begin{enumerate}
\item
$r(M)=2$,   $\exists \overrightarrow{s}=\left(
    \begin{array}{ccc}
      s_1 \\
      s_2 \\
      s_3 \\
    \end{array}
  \right)\neq 0 $ and $\overrightarrow{t}=\left(
    \begin{array}{ccc}
      t_1 \\
      t_2 \\
      t_3 \\
    \end{array}
  \right)\neq 0$ such that $M\overrightarrow{s}=0$, $M^T\overrightarrow{t}=0$, and $s_1t^2_1+s_2t^2_2+s_3t^2_3=0$;

\item  $\exists \overrightarrow{q}=\left(
    \begin{array}{ccc}
      q_1 \\
      q_2 \\
      q_3 \\
    \end{array}
  \right)$ such that $M^T\overrightarrow{q}=\left(
    \begin{array}{ccc}
      t^2_1 \\
      t^2_2 \\
      t^2_3 \\
    \end{array}
  \right)$, which is linearly independent from  $\left(
    \begin{array}{ccc}
      q_1t_1 \\
      q_2t_2 \\
      q_3t_3 \\
    \end{array}
  \right)$;

\item $\exists \overrightarrow{r}=\left(
    \begin{array}{ccc}
      r_1 \\
      r_2 \\
      r_3 \\
    \end{array}
  \right)$ such that $M^T\overrightarrow{r}=\left(
    \begin{array}{ccc}
      q_1t_1 \\
      q_2t_2 \\
      q_3t_3 \\
    \end{array}
  \right)$;

\item $\exists \overrightarrow{u}=\left(
    \begin{array}{ccc}
      u_1 \\
      u_2 \\
      u_3 \\
    \end{array}
  \right)$ such that $M^T\overrightarrow{u}=\left(
    \begin{array}{ccc}
      4r_1t_1+q^2_1 \\
      4r_2t_2+q^2_2\\
      4r_3t_3+q^2_3\\
    \end{array}
  \right)$.
\end{enumerate}
  Then $M$ belongs to one of the following $3$ types:
  $$M_1=\left(
         \begin{array}{ccc}
          a&0&\lambda a \\
          b&0&e  \\
          c&0&\lambda c\\
         \end{array}
       \right), M_2=\left(
    \begin{array}{ccc}
      0&b&e\\
      0&a&\lambda a\\
      0&c&\lambda c\\
    \end{array}
  \right),  M_3=\left(
    \begin{array}{ccc}
      a&\lambda a &0\\
      c&\lambda c &0\\
      b&e         &0\\
    \end{array}
  \right),$$
  where $a,c, \lambda \in k^{\times}, e\neq \lambda b$ and $a^2=\lambda c^2$.
\end{prop}
\begin{proof}
The proof of Proposition \ref{simpleprop} concerns tedious computations and complicated matrix analysis. We provide the detailed proof in the appendix.
\end{proof}

\begin{prop}\label{exist}
Assume that $M=(m_{ij})_{3\times 3}$ satisfies all the  conditions in Proposition \ref{simpleprop}. Then there is $\overrightarrow{v}=\left(
         \begin{array}{ccc}
          v_1 \\
          v_2  \\
          v_3\\
         \end{array}
       \right)$ such that $$M^T\overrightarrow{v}=\overrightarrow{ut+2rq}=\left(
         \begin{array}{ccc}
          u_1t_1+2r_1q_1 \\
          u_2t_2+2r_2q_2 \\
          u_3t_3+2r_3q_3 \\
         \end{array}
       \right).$$
Furthermore, $$r(M^T,\overrightarrow{4vt+2uq+4r^2})=3\neq r(M^T)=2,$$ where
$$ \overrightarrow{4vt+2uq+4r^2}=\left(
         \begin{array}{ccc}
         4v_1t_1+ 2u_1q_1+4r_1^2 \\
         4v_2t_2+ 2u_2q_2+4r_2^2 \\
         4v_3t_3+ 2u_3q_3+4r_3^2 \\
         \end{array}
       \right).$$
\end{prop}
\begin{proof}
By Proposition \ref{simpleprop}, $M$ belongs to three different types. Since the proofs for them are similar, we only need to give a detailed proof for the first case.
Let $M=\left(
         \begin{array}{ccc}
          a&0&\lambda a \\
          b&0&e  \\
          c&0&\lambda c\\
         \end{array}
       \right)$ with $a,c, \lambda \in k^{\times}, e\neq \lambda b$ and $a^2=\lambda c^2$. By the proof of Proposition \ref{simpleprop}, we can choose
       \begin{align*}
       &\overrightarrow{t}=\left(
         \begin{array}{ccc}
          c \\
          0  \\
          -a\\
         \end{array}
       \right), \overrightarrow{q}=\left(
         \begin{array}{ccc}
           \frac{c^2e-a^2b}{ea-\lambda ab} \\
           \frac{a^2-\lambda c^2}{e-\lambda b}  \\
          0 \\
         \end{array}
       \right)=\left(
         \begin{array}{ccc}
           \frac{c^2}{a} \\
           0  \\
          0 \\
         \end{array}
       \right),\\
       & \overrightarrow{r}=\left(
         \begin{array}{ccc}
          \frac{c^3e^2-a^2bce}{a^2(e-\lambda b)^2} \\
          \frac{\lambda (a^2bc-c^3e)}{a(e-\lambda b)^2}  \\
          0\\
         \end{array}
       \right)=\left(
         \begin{array}{ccc}
          \frac{ce}{\lambda(e-\lambda b)} \\
          \frac{ca}{\lambda b-e}  \\
          0\\
         \end{array}
       \right).
       \end{align*}
Hence $$\overrightarrow{4rt+q^2}=\left(
         \begin{array}{ccc}
          \frac{(5c^2e-a^2b)(c^2e-a^2b)}{a^2(e-\lambda b)^2} \\
          \frac{(a^2-\lambda c^2)^2}{(e-\lambda b)^2}  \\
          0\\
         \end{array}
       \right)=\left(
         \begin{array}{ccc}
          \frac{(5e-\lambda b)c^2}{\lambda(e-\lambda b)} \\
          0  \\
          0\\
         \end{array}
       \right).$$
       By $M^T\overrightarrow{u}=\overrightarrow{4rt+q^2}$, we can choose
       $\overrightarrow{u}=\left(
         \begin{array}{ccc}
          \frac{ea(5e-\lambda b)}{\lambda^2(e-\lambda b)^2} \\
          \frac{(\lambda b-5e)c^2}{(e-\lambda b)^2}  \\
          0\\
         \end{array}
       \right).$ Then we have
       $$\overrightarrow{ut+2rq}=\left(
         \begin{array}{ccc}
          u_1t_1+2r_1q_1 \\
          u_2t_2+2r_2q_2 \\
          u_3t_3+2r_3q_3 \\
         \end{array}
       \right)=\left(
         \begin{array}{ccc}
          \frac{eac(7e-3\lambda b)}{\lambda^2(e-\lambda b)^2} \\
          0 \\
          0 \\
         \end{array}
       \right).$$
Hence $r(M^T,\overrightarrow{ut+2rq})=r(M^T)=2$ and there exists $\overrightarrow{v}$ such that $M^T\overrightarrow{v}=\overrightarrow{ut+2rq}$.
More precisely, we can choose $\overrightarrow{v}=\left(
         \begin{array}{ccc}
         \frac{e^2c(7e-3\lambda b)}{\lambda^2(e-\lambda b)^3} \\
         \frac{eac(3\lambda b-7e)}{\lambda(e-\lambda b)^3} \\
          0 \\
         \end{array}
       \right).$ Then
       $$ \overrightarrow{4vt+2uq+4r^2}=\left(
         \begin{array}{ccc}
         4v_1t_1+ 2u_1q_1+4r_1^2 \\
         4v_2t_2+ 2u_2q_2+4r_2^2 \\
         4v_3t_3+ 2u_3q_3+4r_3^2 \\
         \end{array}
       \right)==\left(
         \begin{array}{ccc}
         \frac{ec^2(37e^2-22e\lambda b +\lambda^2b^2)}{\lambda^2(e-\lambda b)^3}\\
         \frac{4c^2a^2}{(\lambda b-e)^2} \\
         0 \\
         \end{array}
       \right).$$
We have $r(M^T,\overrightarrow{4vt+2uq+4r^2})=3\neq r(M^T)$.
\end{proof}

 By Proposition \ref{exist},  $(u_1t_1+2r_1q_1)x_1^2+(u_2t_2+2r_2q_2)x_2^2+(u_3t_3+2r_3q_3)x_3^2\in B^2(\mathcal{A})$. There exists $\omega=v_1x_1+v_2x_2+v_3x_3$ such that $\partial_{\mathcal{A}}(\omega)=\sum\limits_{i=1}^3(u_it_i+2r_iq_i)x_i^2$. It is straightforward for one to see that
$$(t_1x_1+t_2x_2+t_3x_3)e_4+\sigma e_3+2\tau e_2+\lambda e_1+2\omega\in Z^1(F_4).$$ We extent $F_4$ to a  semi-free DG module $F_5$ with $F_5^{\#}=F_4^{\#}\oplus \mathcal{A}^{\#}e_5$ and $\partial_{F_5}(e_5)=(t_1x_1+t_2x_2+t_3x_3)e_4+\sigma e_3+2\tau e_2+\lambda e_1+2\omega.$ One sees that $$\sum\limits_{i=1}^3(4v_it_i+2u_iq_i+4r_i^2)x_i^2\not\in B^2(\mathcal{A}).$$
We claim $H^1(F_5)=0$.
Indeed, for any $a+a_1e_1+a_2e_2+a_3e_3+a_4e_4+a_5e_5\in Z^1(F_5)$, we have
 \begin{align*}
 0&=\partial_{F_5}(a+a_1e_1+a_2e_2+a_3e_3+a_4e_4+a_5e_5)\\
  &=\partial_{\mathcal{A}}(a)+\partial_{\mathcal{A}}(a_1)e_1-a_1(t_1x_1+t_2x_2+t_3x_3)+\partial_{\mathcal{A}}(a_2)e_2\\
  &\quad -a_2[(t_1x_1+t_2x_2+t_3x_3)e_1+\sigma] +\partial_{\mathcal{A}}(a_3)e_3\\
  &\quad -a_3[(t_1x_1+t_2x_2+t_3x_3)e_2+\sigma e_1+2\tau]+\partial_{\mathcal{A}}(a_4)e_4\\
  &\quad -a_4[(t_1x_1+t_2x_2+t_3x_3)e_3+\sigma e_2+2\tau e_1+\lambda]+\partial_{\mathcal{A}}(a_5)e_5 \\
  &\quad -a_5[(t_1x_1+t_2x_2+t_3x_3)e_4+ \sigma e_3+2\tau e_2+\lambda e_1+2\omega ]\\
  &=\partial_{\mathcal{A}}(a_5)e_5 +[\partial_{\mathcal{A}}(a_4)-a_5(t_1x_1+t_2x_2+t_3x_3)]e_4\\
  &\quad +[\partial_{\mathcal{A}}(a_3)-a_5\sigma -a_4(t_1x_1+t_2x_2+t_3x_3)]e_3\\
  &\quad +[\partial_{\mathcal{A}}(a_2)-2a_5\tau-a_4\sigma-a_3(t_1x_1+t_2x_2+t_3x_3)]e_2\\
  &\quad +[\partial_{\mathcal{A}}(a_1)-a_2(t_1x_1+t_2x_2+t_3x_3)-a_3\sigma-2a_4\tau-a_5\lambda]e_1\\
  &\quad +\partial_{\mathcal{A}}(a)-a_1(t_1x_1+t_2x_2+t_3x_3)-a_2\sigma -2a_3\tau-a_4\lambda-2a_5\omega.
 \end{align*}
Then
\begin{align}\label{1.3.2}
 \begin{cases}
 \partial_{\mathcal{A}}(a_5)=0\\
 \partial_{\mathcal{A}}(a_4)-a_5(t_1x_1+t_2x_2+t_3x_3)=0\\
 \partial_{\mathcal{A}}(a_3)-a_5\sigma-a_4(t_1x_1+t_2x_2+t_3x_3)=0\\
 \partial_{\mathcal{A}}(a_2)-2a_5\tau-a_4\sigma-a_3(t_1x_1+t_2x_2+t_3x_3)=0\\
\partial_{\mathcal{A}}(a_1)-a_2(t_1x_1+t_2x_2+t_3x_3)-a_3\sigma-2a_4\tau-a_5\lambda=0\\
\partial_{\mathcal{A}}(a)-a_1(t_1x_1+t_2x_2+t_3x_3)-a_2\sigma -2a_3\tau -a_4\lambda-2a_5\omega.
 \end{cases}
 \end{align}
Since $(4v_1t_1+2q_1u_1)x_1^2+(4v_2t_2+2q_2u_2)x_2^2+(4v_3t_3+2q_3u_3)x_3^2\not\in B^2\mathcal{A}$, (\ref{1.3.2}) implies
\begin{align*}
\begin{cases}
a_5=0\\
a_4=c_4(t_1x_1+t_2x_2+t_3x_3)\\
a_3=c_4(q_1x_1+q_2x_2+q_3x_3)+c_3(t_1x_1+t_2x_2+t_3x_3)\\
a_2=c_3(q_1x_1+q_2x_2+q_3x_3)+c_2(t_1x_1+t_2x_2+t_3x_3)\\
a_1=c_2(q_1x_1+q_2x_2+q_3x_3)+c_4(u_1x_1+u_2x_2+u_3x_3)+c_1(t_1x_1+t_2x_2+t_3x_3)\\
a=c_1(q_1x_1+q_2x_2+q_3x_3) +c_3(u_1x_1+u_2x_2+u_3x_3)+2c_4(v_1x_1+v_2x_2+v_3x_3)\\
\quad + c_0(t_1x_1+t_2x_2+t_3x_3)
\end{cases}
\end{align*}
for some $c_0,c_1,c_2,c_3,c_4\in k$.
 Then
$$a+a_1e_1+a_2e_2+a_3e_3+a_4e_4+a_5e_5 =\partial_{\mathcal{A}}(c_4e_5+c_3e_4+c_2e_3+c_1e_2+c_0e_1).$$
 Hence $H^1(F_5)=0$. Furthermore,
 $F_5$ is the minimal semi-free resolution of ${}_{\mathcal{A}}k$ (see A.2 in the appendix).
 Note that $$F_5^{\#}=\mathcal{A}^{\#}\oplus \mathcal{A}^{\#}e_1\oplus \mathcal{A}^{\#}e_2\oplus \mathcal{A}^{\#}e_3\oplus \mathcal{A}^{\#}e_4\oplus \mathcal{A}^{\#}e_5$$ with a differential defined by
 $$\left(
                         \begin{array}{c}
                          \partial_{F_5}(1) \\
                          \partial_{F_5}(e_1)\\
                          \partial_{F_5}(e_2) \\
                          \partial_{F_5}(e_3) \\
                          \partial_{F_5}(e_4) \\
                          \partial_{F_5}(e_5)
                         \end{array}
                       \right)=  \left(
            \begin{array}{cccccc}
              0 & 0& 0 & 0 &0 & 0 \\
              \sum\limits_{i=1}^3t_ix_i & 0 & 0 & 0 &0 &0 \\
              \sigma & \sum\limits_{i=1}^3t_ix_i & 0 & 0 &0 &0 \\
              2\tau  & \sigma &\sum\limits_{i=1}^3t_ix_i  & 0 &0 &0\\
              \lambda & 2\tau & \sigma & \sum\limits_{i=1}^3t_ix_i & 0 & 0\\
              2\omega & \lambda & 2\tau & \sigma & \sum\limits_{i=1}^3t_ix_i & 0
            \end{array}
          \right) \left(
                         \begin{array}{c}
                          1 \\
                          e_1\\
                          e_2\\
                          e_3\\
                          e_4\\
                          e_5
                         \end{array}
                       \right).$$
 We can list the following examples for Case $1.3.2$.
 \begin{ex}\label{ex1.3.2}
The DG algebra $\mathcal{A}_{\mathcal{O}_{-1}(k^3)}(M)$ belongs to Case $1.3.2$, when $M$ is one of the following matrixes:
 \begin{align*}
& (1) \left(
           \begin{array}{ccc}
             1 & 1 & 0 \\
             1 & 1 & 0 \\
             0 & 1 & 0 \\
           \end{array}
         \right),(2) \left(
           \begin{array}{ccc}
             1 & 0 & 1 \\
             0 & 0 & 1 \\
             1 & 0 & 1 \\
           \end{array}
         \right),(3) \left(
           \begin{array}{ccc}
             1 & 0 & 1 \\
             1 & 0 & 0 \\
             1 & 0 & 1 \\
           \end{array}
         \right),(4) \left(
           \begin{array}{ccc}
             1 & 0 & 1 \\
             -1 & 0 & -2 \\
             1 & 0 & 1 \\
           \end{array}
         \right).
 \end{align*}
\end{ex}
In summary, the constructing procedure of the minimal semifree resolution of ${}_{\mathcal{A}}k$ illustrated above can be simply described as follows:
\begin{align*}
&F_7\backslash  \quad \text{end the construction step ---Case 1.2.4}\\
&\cup \quad \mathcal{A}e_7\\
/&F_6/  \quad \text{proceed the construction if \textbf{C}2}, \overline{\textbf{C}3}\,\, \overline{\textbf{C}4}\,\,\text{and}\,\, \overline{\textbf{C}5}\,\,  \text{hold} \\
\mathcal{A}e_6\quad &\cup\\
\backslash&F_5\backslash  \quad \text{terminate if \textbf{C}2}, \overline{\textbf{C}3}\,\, \overline{\textbf{C}4}\,\,\text{and \textbf{C}5  hold---Case 1.2.3}\\
             &\cup \quad \mathcal{A}e_5\\
            /&F_4/  \quad \text{terminate if \textbf{C}2}, \overline{\textbf{C}3}\,\,\text{and \textbf{C}4 hold---Case 1.2.2} \\
           \mathcal{A}e_4\quad &\cup\\
            \backslash &F_3\backslash \quad \text{terminate if \textbf{C}2 and \textbf{C}3 hold---Case 1.2.1} \\
           \uparrow &\cup  \quad \mathcal{A}e_3   \\
  F_0\quad \subset\quad F_1 \quad \subset \quad /& F_2 /\quad \text{terminate if \textbf{C}1 holds---Case 1.1} \\
\mathcal{A}e_3\quad &\cap \downarrow \quad \text{if}\\
\backslash & F_3\backslash \quad \text{terminate if }\textbf{C}2'\,\,\text{and}\,\, \textbf{C}3'\,\, \text{holds---Case 1.3.1}\\
 &\cap  \quad \mathcal{A}e_4\\
 / & F_4 / \quad \text{proceed the construction if}\,\, \textbf{C}2'\,\, \text{and}\,\, \overline{\textbf{C}3'}\,\,\text{hold}\\
\mathcal{A}e_5\quad &\cap  \\
  \backslash & F_5 \quad \text{end the construction step---Case 1.3.2}
\end{align*}
From the minimal semi-free resolution constructed above, we can show the following proposition.
\begin{prop}\label{case1cy}
We have the following table:
\begin{center}
\begin{tabular}{|l|l|}
  \hline
  Subcases & $\mathrm{Ext}$-algeba\\  \hline
  Case $1.1$ &  $\cong k[x]/(x^3)$
   \\
  \hline
  Case $1.2.1$ &  $\cong k[x]/(x^4)$ \\
  \hline
  Case $1.2.2$ & $\cong k[x]/(x^5)$ \\
  \hline
  Case $1.2.3$ & $\cong k[x]/(x^6)$ \\
  \hline
  Case $1.2.4$ & $\cong k[x]/(x^8)$ \\
  \hline
  Case $1.3.1$ & $\cong k[x]/(x^4)$  \\
  \hline
  Case $1.3.2$ & $\cong k[x]/(x^6)$,  \\
  \hline
\end{tabular},
\end{center}
 which contains a complete list of the $\mathrm{Ext}$-algebra of $k$ considered as a module over $\mathcal{A}_{\mathcal{O}_{-1}(k^3)}(M)$ in
 all subcases of Case $1$. From this table, one sees that each DG algebra $\mathcal{A}_{\mathcal{O}_{-1}(k^3)}(M)$ in
  Case $1$ is a Koszul Calabi-Yau DG algebra.
\end{prop}
\begin{proof}
Since the proof is similar to each other, we only need to consider the most complicated subcase: Case $1.2.4$.
In this case, ${}_{\mathcal{A}}k$ admits a minimal semi-free resolution $F=F_7$ with
$$
F^{\#}=\mathcal{A}^{\#}\oplus \mathcal{A}^{\#}e_1\oplus \mathcal{A}^{\#}e_2\oplus \mathcal{A}^{\#}e_3\oplus \mathcal{A}^{\#}e_4\oplus \mathcal{A}^{\#}e_5\oplus \mathcal{A}^{\#}e_6\oplus \mathcal{A}^{\#}e_7$$
and
$$\left(
                         \begin{array}{c}
                          \partial_F(1) \\
                          \partial_F(e_1)\\
                          \partial_F(e_2)\\
                          \partial_F(e_3)\\
                          \partial_F(e_4)\\
                          \partial_F(e_5)\\
                          \partial_F(e_6)\\
                          \partial_F(e_7)
                         \end{array}
                       \right) =\left(
            \begin{array}{cccccccc}
              0 & 0& 0 & 0 &0 & 0 &0 & 0\\
               t_3x_3  & 0 & 0 & 0 &0 &0&0 & 0 \\
              \sigma &  t_3x_3  & 0 & 0 &0 &0&0 & 0 \\
              0 & \sigma & t_3x_3   & 0 &0 &0&0 & 0\\
              \lambda & 0 & \sigma &  t_3x_3  & 0 & 0&0 & 0\\
              0   & \lambda & 0       & \sigma &  t_3x_3  & 0      &0 & 0\\
             \eta & 0       & \lambda & 0      & \sigma   & t_3x_3 &0 & 0\\
             0    & \eta    & 0       &\lambda &  0       &\sigma  &t_3x_3 & 0
            \end{array}
          \right)\left(
                         \begin{array}{c}
                          1\\
                          e_1\\
                          e_2\\
                          e_3\\
                          e_4\\
                          e_5\\
                          e_6\\
                          e_7
                         \end{array}
                       \right).$$
By the minimality of $F$, the Ext-algebra \begin{align*}
  E=H(R\Hom_{\mathcal{A}}(k,k))&=H(\Hom_{\mathcal{A}}(F,k))=\Hom_{\mathcal{A}}(F,k)\\
  &=k1^*\oplus ke_1^*\oplus ke_2^*\oplus ke_3^*\oplus ke_4^*\oplus ke_5^*\oplus ke_6^*\oplus ke_7^*.
  \end{align*}
  So $E$ is concentrated in degree $0$. On the other hand, \begin{align*}
  \Hom_{\mathcal{A}}(F,F)^{\#}& = \Hom_{\mathcal{A}^{\#}}(\mathcal{A}^{\#}\otimes (k\oplus ke_1\oplus ke_2\oplus ke_3), F^{\#})\\
  &\cong \Hom_k(k\oplus (\bigoplus\limits_{i=1}^7ke_i),\Hom_{\mathcal{A}^{\#}}(\mathcal{A}^{\#},F^{\#}))\\
  &\cong \Hom_k(k\oplus (\bigoplus\limits_{i=1}^7ke_i),k)\otimes F^{\#}\\
  &\cong [k1^*\oplus (\bigoplus\limits_{i=1}^7ke_i^*)]\otimes_{k} F^{\#}\\
  \end{align*}
  is concentrated in degree $\ge 0$. This implies that $E= Z^0(\Hom_{\mathcal{A}}(F,F))$.  Since $F^{\#}$ is a free graded $\mathcal{A}^{\#}$-module with a basis
$\{1, e_1, e_2, e_3, e_4,e_5,e_6,e_7\}$ concentrated in degree $0$,
  the elements in  $\Hom_{\mathcal{A}}(F,F)^0$ is one to one correspondence with the matrixes in $M_8(k)$.
 Indeed, any $f\in \Hom_{\mathcal{A}}(F,F)^0$ is uniquely determined by
  a matrix $A_{f}=(a_{ij})_{8\times 8}\in M_8(k)$ with
$$\left(
                         \begin{array}{c}
                          f(1) \\
                          f(e_1)\\
                          f(e_2)\\
                           f(e_3)\\
                            f(e_4)\\
                             f(e_5)\\
                              f(e_6)\\
                               f( e_7)
                         \end{array}
                       \right) =      A_{f} \cdot \left(
                         \begin{array}{c}
                          1 \\
                          e_1\\
                          e_2\\
                          e_3\\
                          e_4\\
                          e_5\\
                          e_6\\
                          e_7
                         \end{array}
                       \right).  $$
 And $f\in  Z^0[\Hom_{\mathcal{A}}(F,F)]$ if and only if $\partial_{F}\circ f=f\circ \partial_{F}$, if and only if
\begin{small}
 \begin{align*}
 &\quad A_f \left(
            \begin{array}{cccccccc}
              0 & 0& 0 & 0 &0 & 0 &0 & 0\\
               t_3x_3  & 0 & 0 & 0 &0 &0&0 & 0 \\
              \sigma &  t_3x_3  & 0 & 0 &0 &0&0 & 0 \\
              0 & \sigma & t_3x_3   & 0 &0 &0&0 & 0\\
              \lambda & 0 & \sigma &  t_3x_3  & 0 & 0&0 & 0\\
              0   & \lambda & 0       & \sigma &  t_3x_3  & 0      &0 & 0\\
             \eta & 0       & \lambda & 0      & \sigma   & t_3x_3 &0 & 0\\
             0    & \eta    & 0       &\lambda &  0       &\sigma  &t_3x_3 & 0
            \end{array}
          \right) \\
 & =  \left(
            \begin{array}{cccccccc}
              0 & 0& 0 & 0 &0 & 0 &0 & 0\\
               t_3x_3  & 0 & 0 & 0 &0 &0&0 & 0 \\
              \sigma &  t_3x_3  & 0 & 0 &0 &0&0 & 0 \\
              0 & \sigma & t_3x_3   & 0 &0 &0&0 & 0\\
              \lambda & 0 & \sigma &  t_3x_3  & 0 & 0&0 & 0\\
              0   & \lambda & 0       & \sigma &  t_3x_3  & 0      &0 & 0\\
             \eta & 0       & \lambda & 0      & \sigma   & t_3x_3 &0 & 0\\
             0    & \eta    & 0       &\lambda &  0       &\sigma  &t_3x_3 & 0
            \end{array}
          \right) A_f
          \end{align*}
 \end{small} which is also equivalent to
                       $$\begin{cases}
                       a_{ij}=0, \forall i<j\\
                       a_{11}=a_{22}=a_{33}=a_{44}=a_{55}=a_{66}=a_{77}=a_{88}\\
                     a_{87}=a_{76}=a_{65}=a_{54}=a_{43}=a_{32}=a_{21}, \\
                     a_{86}=a_{75}=a_{64}=a_{53}=a_{42}=a_{31},\\
                     a_{85}=a_{74}=a_{63}=a_{52}=a_{41},\\
                     a_{84}=a_{73}=a_{62}=a_{51},a_{83}=a_{72}=a_{61}, a_{82}=a_{71}
                       \end{cases}$$
by direct computations. Hence the the Ext-algebra $$ E\cong \left\{\left(
                                                             \begin{array}{cccccccc}
                                                               a & 0 & 0 & 0 &0 &0  &0 &0 \\
                                                               b & a & 0 & 0 &0 &0  &0 &0\\
                                                               c & b & a & 0 &0 &0  &0 &0\\
                                                               d & c & b & a &0 &0  &0 &0\\
                                                               e & d & c & b &a &0  &0 &0\\
                                                               f & e & d & c &b &a  &0 &0\\
                                                               g & f & e & d &c &b  &a &0\\
                                                               h & g & f & e &d &c  &b &a
                                                             \end{array}
                                                           \right)
\quad | \quad a,b,c, d, e, f,g,h\in k \right\}\cong k[x]/(x^8).$$
So $E$ is a  symmetric Frobenius algebra concentrated
in degree $0$. This implies that
$\mathrm{Tor}_{\mathcal{A}}(k_{\mathcal{A}},{}_{\mathcal{A}}k)\cong
E^*$ is a symmetric coalgebra. By Remark \ref{methods}, the DG algebra
$\mathcal{A}$ in Case $1.2.4$  is a Koszul Calabi-Yau DG algebra.

Similarly, we can get the $\mathrm{Ext}$-algebra of $k$ considered as a module over $\mathcal{A}_{\mathcal{O}_{-1}(k^3)}(M)$ in
 other subcases.
\end{proof}

\begin{rem}\label{discover}
As a biproduct of Proposition \ref{case1cy},
we get counter-examples for Question \ref{biproduct}, since the $\mathrm{Ext}$-algebras of two quasi-isomorphic connected cochain DG algebras should be isomorphic to each other.
\end{rem}

\section{case $2$, case $3$ and case $4$}\label{casetwo}
In this section, we study homological properties of $\mathcal{A}_{\mathcal{O}_{-1}(k^3)}(M)$ for case $2$, case $3$ and case $4$.
By Theorem \ref{iso}, we have the following lemmas on its isomorphism classes.
\begin{lem}\label{rank1-1}
Let $$M=\left(
                                 \begin{array}{ccc}
                                   m_{11} & m_{12} & m_{13} \\
                                   l_1m_{11} & l_1m_{12} & l_1m_{13} \\
                                   l_2m_{11} & l_2m_{12} & l_2m_{13} \\
                                 \end{array}
                               \right), (m_{11},m_{12},m_{13})\neq 0,$$ $m_{12}l_1^2+m_{13}l_2^2=m_{11}$, $l_1=l_2=0$.  Then
\begin{enumerate}
\item $\mathcal{A}_{\mathcal{O}_{-1}(k^3)}(M)\cong \mathcal{A}_{\mathcal{O}_{-1}(k^3)}(E_{12}+E_{13}) $ if $M_{12}\neq 0$ and $M_{13}\neq 0$;
\item $\mathcal{A}_{\mathcal{O}_{-1}(k^3)}(M)\cong \mathcal{A}_{\mathcal{O}_{-1}(k^3)}(E_{12}) $ if $M_{13}=0, M_{12}\neq 0$ or $M_{12}\neq 0, M_{13}=0$.
\end{enumerate}
\end{lem}

\begin{proof}
(1) By the assumption, we have $m_{11}=0$ and $M=\left(
                                 \begin{array}{ccc}
                                   0 & m_{12} & m_{13} \\
                                   0 & 0 & 0 \\
                                   0 & 0 & 0 \\
                                 \end{array}
                               \right)$. Let $C=\left(
                                                  \begin{array}{ccc}
                                                    1 & 0 & 0 \\
                                                    0 & \sqrt{\frac{1}{m_{12}}} & 0 \\
                                                    0 & 0 & \sqrt{\frac{1}{m_{13}}} \\
                                                  \end{array}
                                                \right)$.
                                Then $C\in \mathrm{Gl}_3(k)$, and
\begin{align*}
\chi(M,C)&=\left(
                                                  \begin{array}{ccc}
                                                    1 & 0 & 0 \\
                                                    0 & \sqrt{m_{12}} & 0 \\
                                                    0 & 0 & \sqrt{m_{13}} \\
                                                  \end{array}
                                                \right)\left(
                                 \begin{array}{ccc}
                                   0 & m_{12} & m_{13} \\
                                   0 & 0 & 0 \\
                                   0 & 0 & 0 \\
                                 \end{array}
                               \right)\left(
                                                  \begin{array}{ccc}
                                                    1 & 0 & 0 \\
                                                    0 & \frac{1}{m_{12}} & 0 \\
                                                    0 & 0 & \frac{1}{m_{13}} \\
                                                  \end{array}
                                                \right)\\
&=\left(
                                 \begin{array}{ccc}
                                   0 & 1 & 1 \\
                                   0 & 0 & 0 \\
                                   0 & 0 & 0 \\
                                 \end{array}
                               \right)=E_{12}+E_{13}.
\end{align*}
By Theorem \ref{iso}, $\mathcal{A}_{\mathcal{O}_{-1}(k^3)}(M)\cong \mathcal{A}_{\mathcal{O}_{-1}(k^3)}(E_{12}+E_{13})$.

(2)If $m_{12}\neq 0$ and $m_{13}=0$, then $M=m_{12}E_{12}$.  Let $C'=\left(
                                                  \begin{array}{ccc}
                                                    1 & 0 & 0 \\
                                                    0 & \sqrt{\frac{1}{m_{12}}} & 0 \\
                                                    0 & 0 & 1 \\
                                                  \end{array}
                                                \right)$.
Then
\begin{align*}
\quad\chi(M,C')&=\left(
                                                  \begin{array}{ccc}
                                                    1 & 0 & 0 \\
                                                    0 & \sqrt{m_{12}} & 0 \\
                                                    0 & 0 & 1 \\
                                                  \end{array}
                                                \right)\left(
                                 \begin{array}{ccc}
                                   0 & m_{12} & 0 \\
                                   0 & 0 & 0 \\
                                   0 & 0 & 0 \\
                                 \end{array}
                               \right)\left(
                                                  \begin{array}{ccc}
                                                    1 & 0 & 0 \\
                                                    0 & \frac{1}{m_{12}} & 0 \\
                                                    0 & 0 &1 \\
                                                  \end{array}
                                                \right)\\
&=\left(
                                 \begin{array}{ccc}
                                   0 & 1 & 0 \\
                                   0 & 0 & 0 \\
                                   0 & 0 & 0 \\
                                 \end{array}
                               \right)=E_{12}.
\end{align*}

If $m_{12}=0$ and $m_{13}\neq 0$, then $M=m_{13}E_{13}$.  Let $C''=\left(
                                                  \begin{array}{ccc}
                                                    1 & 0 & 0 \\
                                                    0 & 1 & 0 \\
                                                    0 & 0 & \sqrt{\frac{1}{m_{13}}} \\
                                                  \end{array}
                                                \right)$.
Then
\begin{align*}
\quad\chi(M,C'')&=\left(
                                                  \begin{array}{ccc}
                                                    1 & 0 & 0 \\
                                                    0 & 1 & 0 \\
                                                    0 & 0 & \sqrt{m_{13}} \\
                                                  \end{array}
                                                \right)\left(
                                 \begin{array}{ccc}
                                   0 & 0 & m_{13} \\
                                   0 & 0 & 0 \\
                                   0 & 0 & 0 \\
                                 \end{array}
                               \right)\left(
                                                  \begin{array}{ccc}
                                                    1 & 0 & 0 \\
                                                    0 & 1 & 0 \\
                                                    0 & 0 & \frac{1}{m_{13}} \\
                                                  \end{array}
                                                \right)\\
&=\left(
                                 \begin{array}{ccc}
                                   0 & 0 & 1 \\
                                   0 & 0 & 0 \\
                                   0 & 0 & 0 \\
                                 \end{array}
                               \right)=E_{13}.
\end{align*}
On the other hand, let $Q=\left(
                            \begin{array}{ccc}
                              1 & 0 & 0 \\
                              0 & 0 & 1 \\
                              0 & 1 & 0 \\
                            \end{array}
                          \right)$.  Then $\chi(E_{12}, Q)=E_{13}$ and so $\mathcal{A}_{\mathcal{O}_{-1}(k^3)}(E_{12})\cong \mathcal{A}_{\mathcal{O}_{-1}(k^3)}(E_{13})$ by Theorem \ref{iso}.
\end{proof}

\begin{lem}\label{rank1-2}
Let $$M=\left(
                                 \begin{array}{ccc}
                                   m_{11} & m_{12} & m_{13} \\
                                   l_1m_{11} & l_1m_{12} & l_1m_{13} \\
                                   l_2m_{11} & l_2m_{12} & l_2m_{13} \\
                                 \end{array}
                               \right), (m_{11},m_{12},m_{13})\neq 0,$$  $m_{12}l_1^2+m_{13}l_2^2=m_{11}$, $l_1\neq 0$ and $l_2= 0$.
Then
\begin{enumerate}
\item $\mathcal{A}_{\mathcal{O}_{-1}(k^3)}(M)\cong \mathcal{A}_{\mathcal{O}_{-1}(k^3)}(E_{11}+E_{12}+E_{13}+E_{21}+E_{22}+E_{23})$ if $m_{12}\neq 0$ and $m_{13}\neq 0$;
\item $\mathcal{A}_{\mathcal{O}_{-1}(k^3)}(M)\cong \mathcal{A}_{\mathcal{O}_{-1}(k^3)}(E_{13}+E_{23})$ if $m_{12}=0$ and $m_{13}\neq 0$;
\item $\mathcal{A}_{\mathcal{O}_{-1}(k^3)}(M)\cong \mathcal{A}_{\mathcal{O}_{-1}(k^3)}(E_{11}+E_{12}+E_{21}+E_{22})$ if $m_{12}\neq 0$ and $m_{13}=0$.
\end{enumerate}
\end{lem}
\begin{proof}
(1)By the assumption, $m_{11}=m_{12}l_1^2$ and $M=\left(
                                 \begin{array}{ccc}
                                   m_{12}l_1^2 & m_{12} & m_{13} \\
                                   m_{12}l_1^3 & l_1m_{12} & l_1m_{13} \\
                                   0 & 0 & 0 \\
                                 \end{array}
                               \right)$.
Let $C=\left(
         \begin{array}{ccc}
           \frac{1}{m_{12}l_1^2} & 0 & 0 \\
           0 & \frac{1}{m_{12}l_1} & 0 \\
           0 & 0 & \sqrt{\frac{1}{m_{12}l_1^2m_{13}}} \\
         \end{array}
       \right).$ Then
       \begin{align*}
       &\quad\chi(M,C)=C^{-1}M(c_{ij}^2)\\
       &= \left(
         \begin{array}{ccc}
           m_{12}l_1^2 & 0 & 0 \\
           0 & m_{12}l_1 & 0 \\
           0 & 0 & \sqrt{m_{12}l_1^2m_{13}} \\
         \end{array}
       \right)M\left(
         \begin{array}{ccc}
           \frac{1}{m_{12}^2l_1^4} & 0 & 0 \\
           0 & \frac{1}{m_{12}^2l_1^2} & 0 \\
           0 & 0 & \frac{1}{m_{12}l_1^2m_{13}} \\
         \end{array}
       \right)\\
  &=\left(
      \begin{array}{ccc}
        1 & 1 & 1 \\
        1 & 1 & 1 \\
        0 & 0 & 0 \\
      \end{array}
    \right)=E_{11}+E_{12}+E_{13}+E_{21}+E_{22}+E_{23}).
\end{align*} So $\mathcal{A}_{\mathcal{O}_{-1}(k^3)}(M)\cong \mathcal{A}_{\mathcal{O}_{-1}(k^3)}(E_{11}+E_{12}+E_{13}+E_{21}+E_{22}+E_{23}$ by Theorem \ref{iso}.

(2)By the assumption, $m_{11}=m_{12}l_1^2=0$ and $M=\left(
                                 \begin{array}{ccc}
                                   0 & 0 & m_{13} \\
                                   0 & 0 & l_1m_{13} \\
                                   0 & 0 & 0 \\
                                 \end{array}
                               \right)$.
Let $C'=\left(
         \begin{array}{ccc}
           \frac{1}{l_1^2} & 0 & 0 \\
           0 & \frac{1}{l_1} & 0 \\
           0 & 0 & \sqrt{\frac{1}{l_1^2m_{13}}} \\
         \end{array}
       \right).$
Then
        \begin{align*}
 \chi(M,C') &= \left(
         \begin{array}{ccc}
           l_1^2 & 0 & 0 \\
           0 & l_1 & 0 \\
           0 & 0 & \sqrt{l_1^2m_{13}} \\
         \end{array}
       \right) \left(
                                 \begin{array}{ccc}
                                   0 & 0 & m_{13} \\
                                   0 & 0 & l_1m_{13} \\
                                   0 & 0 & 0 \\
                                 \end{array}
                               \right)  \left(
         \begin{array}{ccc}
           \frac{1}{l_1^4} & 0 & 0 \\
           0 & \frac{1}{l_1^2} & 0 \\
           0 & 0 & \frac{1}{l_1^2m_{13}} \\
         \end{array}
       \right)\\
       &=\left(
           \begin{array}{ccc}
             0 & 0 & 1 \\
             0 & 0 & 1 \\
             0 & 0 & 0 \\
           \end{array}
         \right)=E_{13}+E_{23}.
       \end{align*}
Thus $\mathcal{A}_{\mathcal{O}_{-1}(k^3)}(M)\cong \mathcal{A}_{\mathcal{O}_{-1}(k^3)}(E_{13}+E_{23})$ by Theorem \ref{iso}.

(3)By assumptions, $m_{11}=m_{12}l_1^2\neq 0$ and $m_{13}=0$ $M=\left(
                                 \begin{array}{ccc}
                                   m_{12}l_1^2 & m_{12} & 0 \\
                                   m_{12}l_1^3 & l_1m_{12} & 0 \\
                                   0 & 0 & 0 \\
                                 \end{array}
                               \right)$.
Let $C''=\left(
         \begin{array}{ccc}
           \frac{1}{m_{12}l_1^2} & 0 & 0 \\
           0 & \frac{1}{m_{12}l_1} & 0 \\
           0 & 0 & \sqrt{\frac{1}{m_{12}l_1^2}} \\
         \end{array}
       \right).$
Then
        \begin{align*}
& \quad\chi(M,C'')\\
&= \begin{small}\left(
         \begin{array}{ccc}
           m_{12}l_1^2 & 0 & 0 \\
           0 & m_{12}l_1 & 0 \\
           0 & 0 & \sqrt{m_{12}l_1^2} \\
         \end{array}
       \right) \left(
                                 \begin{array}{ccc}
                                   m_{12}l_1^2 & m_{12} & 0 \\
                                   m_{12}l_1^3 & l_1m_{12} & 0 \\
                                   0 & 0 & 0 \\
                                 \end{array}
                               \right)  \left(
         \begin{array}{ccc}
           \frac{1}{m_{12}^2l_1^4} & 0 & 0 \\
           0 & \frac{1}{m_{12}^2l_1^2} & 0 \\
           0 & 0 & \frac{1}{m_{12}l_1^2} \\
         \end{array}
       \right)\end{small} \\
       &=\left(
           \begin{array}{ccc}
             1 & 1 & 0 \\
             1 & 1 & 0 \\
             0 & 0 & 0 \\
           \end{array}
         \right)=E_{11}+E_{12}+E_{21}+E_{22}.
       \end{align*}
Therefore,  $\mathcal{A}_{\mathcal{O}_{-1}(k^3)}(M)\cong \mathcal{A}_{\mathcal{O}_{-1}(k^3)}(E_{11}+E_{12}+E_{21}+E_{22})$ by Theorem \ref{iso}.
\end{proof}
By a similar proof, we can show the following proposition.
\begin{lem}\label{rank1-3}Let $$M=\left(
                                 \begin{array}{ccc}
                                   m_{11} & m_{12} & m_{13} \\
                                   l_1m_{11} & l_1m_{12} & l_1m_{13} \\
                                   l_2m_{11} & l_2m_{12} & l_2m_{13} \\
                                 \end{array}
                               \right), (m_{11},m_{12},m_{13})\neq 0,$$ $m_{12}l_1^2+m_{13}l_2^2=m_{11}$,  $l_1=0$ and $l_2\neq 0$. Then
\begin{enumerate}
\item $\mathcal{A}_{\mathcal{O}_{-1}(k^3)}(M)\cong \mathcal{A}_{\mathcal{O}_{-1}(k^3)}(E_{11}+E_{12}+E_{13}+E_{31}+E_{32}+E_{33})$ if $m_{12}\neq 0$ and $m_{13}\neq 0$;
\item $\mathcal{A}_{\mathcal{O}_{-1}(k^3)}(M)\cong \mathcal{A}_{\mathcal{O}_{-1}(k^3)}(E_{11}+E_{13}+E_{31}+E_{33})$ if $m_{12}=0$ and $m_{13}\neq 0$;
\item $\mathcal{A}_{\mathcal{O}_{-1}(k^3)}(M)\cong \mathcal{A}_{\mathcal{O}_{-1}(k^3)}(E_{11}+E_{12}+E_{31}+E_{32})$ if $m_{12}\neq 0$ and $m_{13}=0$.
\end{enumerate}
\end{lem}
\begin{proof}
(1)By the assumption, $m_{11}=m_{13}l_2^2$ and $M=\left(
                                 \begin{array}{ccc}
                                   m_{13}l_2^2 & m_{12} & m_{13} \\
                                   0 & 0 & 0 \\
                                   m_{13}l_2^3 & l_2m_{12} & l_2m_{13} \\
                                 \end{array}
                               \right)$.
Let $C=\left(
         \begin{array}{ccc}
           \frac{1}{m_{13}l_2^2} & 0 & 0 \\
           0 & \sqrt{\frac{1}{m_{12}m_{13}l_2^2}} & 0 \\
           0 & 0 & \frac{1}{m_{13}l_2} \\
         \end{array}
       \right).$ Then
       \begin{align*}
       &\quad\chi(M,C)=C^{-1}M(c_{ij}^2)\\
       &= \left(
         \begin{array}{ccc}
           m_{13}l_2^2 & 0 & 0 \\
           0 & \sqrt{m_{12}m_{13}l_2^2} & 0 \\
           0 & 0 &  m_{13}l_2 \\
         \end{array}
       \right)M\left(
         \begin{array}{ccc}
           \frac{1}{m_{13}^2l_2^4} & 0 & 0 \\
           0 & \frac{1}{m_{12}m_{13}l_2^2} & 0 \\
           0 & 0 & \frac{1}{m_{13}^2l_2^2} \\
         \end{array}
       \right)\\
  &=\left(
      \begin{array}{ccc}
        1 & 1 & 1 \\
        0 & 0 & 0 \\
        1 & 1 & 1 \\
      \end{array}
    \right)=E_{11}+E_{12}+E_{13}+E_{31}+E_{32}+E_{33}).
\end{align*} So $\mathcal{A}_{\mathcal{O}_{-1}(k^3)}(M)\cong \mathcal{A}_{\mathcal{O}_{-1}(k^3)}(E_{11}+E_{12}+E_{13}+E_{31}+E_{32}+E_{33})$ by Theorem \ref{iso}.

(2)By the assumption, $m_{11}=m_{13}l_2^2$ and $M=\left(
                                 \begin{array}{ccc}
                                   m_{13}l_2^2 & 0 & m_{13} \\
                                   0 & 0 & 0 \\
                                   m_{13}l_2^3 & 0 & m_{13}l_2 \\
                                 \end{array}
                               \right)$.
Let $C'=\left(
         \begin{array}{ccc}
           \frac{1}{m_{13}l_2^2} & 0 & 0 \\
           0 & 1 & 0 \\
           0 & 0 & \frac{1}{m_{13}l_2} \\
         \end{array}
       \right).$
Then
        \begin{align*}
 \chi(M,C') &= \left(
         \begin{array}{ccc}
           m_{13}l_2^2 & 0 & 0 \\
           0 & 1 & 0 \\
           0 & 0 & m_{13}l_2 \\
         \end{array}
       \right) \left(
                                 \begin{array}{ccc}
                                   m_{13}l_2^2 & 0 & m_{13} \\
                                   0 & 0 & 0\\
                                   m_{13}l_2^3 & 0 & m_{13}l_2 \\
                                 \end{array}
                               \right)  \left(
         \begin{array}{ccc}
           \frac{1}{m_{13}^2l_2^4} & 0 & 0 \\
           0 & 1 & 0 \\
           0 & 0 & \frac{1}{l_2^2m_{13}^2} \\
         \end{array}
       \right)\\
       &=\left(
           \begin{array}{ccc}
             1 & 0 & 1 \\
             0 & 0 & 0 \\
             1 & 0 & 1 \\
           \end{array}
         \right)=E_{11}+E_{13}+E_{31}+E_{33}.
       \end{align*}
Thus $\mathcal{A}_{\mathcal{O}_{-1}(k^3)}(M)\cong \mathcal{A}_{\mathcal{O}_{-1}(k^3)}(E_{13}+E_{23})$ by Theorem \ref{iso}.

(3)By assumptions, $m_{11}=m_{13}l_2^2=0$, $m_{13}=0$ and $M=\left(
                                 \begin{array}{ccc}
                                   0 & m_{12} & 0 \\
                                   0 & 0 & 0 \\
                                   0 & l_2m_{12} & 0 \\
                                 \end{array}
                               \right)$.
Let $C''=\left(
         \begin{array}{ccc}
           m_{12} & 0 & 0 \\
           0 & 1 & 0 \\
           0 & 0 & l_2m_{12}\\
         \end{array}
       \right).$
Then
        \begin{align*}
& \quad\chi(M,C'')\\
&= \begin{small}\left(
         \begin{array}{ccc}
          \frac{1}{m_{12}} & 0 & 0 \\
           0 & 1 & 0 \\
           0 & 0 & \frac{1}{m_{12}l_2} \\
         \end{array}
       \right)   \left(
         \begin{array}{ccc}
           0 & m_{12} & 0 \\
           0 & 0 & 0 \\
           0 & l_2m_{12} & 0 \\
         \end{array}
       \right)\left(
         \begin{array}{ccc}
          m_{12}^2 & 0 & 0 \\
           0 & 1 & 0 \\
           0 & 0 & m_{12}^2l_2^2 \\
         \end{array}
       \right) \end{small} \\
       &=\left(
           \begin{array}{ccc}
             0 & 1 & 0 \\
             0 & 0 & 0 \\
             0 & 1 & 0 \\
           \end{array}
         \right)=E_{12}+E_{32}.
       \end{align*}
Therefore,  $\mathcal{A}_{\mathcal{O}_{-1}(k^3)}(M)\cong \mathcal{A}_{\mathcal{O}_{-1}(k^3)}(E_{12}+E_{32})$ by Theorem \ref{iso}.

\end{proof}

\begin{lem}\label{rank1-4}
We have $\mathcal{A}_{\mathcal{O}_{-1}(k^3)}(M)\cong \mathcal{A}_{\mathcal{O}_{-1}(k^3)}(N)$ if $M$ and $N$ belong to the following cases:

\begin{enumerate}
\item $M=E_{12}+E_{32}, N=E_{13}+E_{23}$;
\item $M=E_{11}+E_{13}+E_{31}+E_{33}, N=E_{11}+E_{12}+E_{21}+E_{22}$;
\item $M=E_{11}+E_{12}+E_{13}+E_{21}+E_{22}+E_{23}, N=E_{11}+E_{12}+E_{13}+E_{31}+E_{32}+E_{33}$.
\end{enumerate}
\end{lem}
\begin{proof}
(1)Let $C=\left(
            \begin{array}{ccc}
              1 & 0 & 0 \\
              0 & 0 & 1 \\
              0 & 1 & 0 \\
            \end{array}
          \right)$.
Then
          \begin{align*}
 &\quad \chi(E_{12}+E_{32},C)\\
         &=\left(
             \begin{array}{ccc}
               1 & 0 & 0 \\
               0 & 0 & 1 \\
               0 & 1 & 0 \\
             \end{array}
           \right)\left(
                    \begin{array}{ccc}
                      0 & 1 & 0 \\
                      0 & 0 & 0 \\
                      0 & 1 & 0 \\
                    \end{array}
                  \right)\left(
                           \begin{array}{ccc}
                             1 & 0 & 0 \\
                             0 & 0 & 1 \\
                             0 & 1 & 0 \\
                           \end{array}
                         \right)\\
           &=\left(
                           \begin{array}{ccc}
                             0 & 0 & 1 \\
                             0 & 0 & 1 \\
                             0 & 0 & 0 \\
                           \end{array}
                         \right)=E_{13}+E_{23},
          \end{align*}
  \begin{align*}
 &\quad \chi(E_{11}+E_{13}+E_{31}+E_{33},C)\\
         &=\left(
             \begin{array}{ccc}
               1 & 0 & 0 \\
               0 & 0 & 1 \\
               0 & 1 & 0 \\
             \end{array}
           \right)\left(
                    \begin{array}{ccc}
                      1 & 0 & 1 \\
                      0 & 0 & 0 \\
                      1 & 0 & 1 \\
                    \end{array}
                  \right)\left(
                           \begin{array}{ccc}
                             1 & 0 & 0 \\
                             0 & 0 & 1 \\
                             0 & 1 & 0 \\
                           \end{array}
                         \right)\\
           &=\left(
                           \begin{array}{ccc}
                             1 & 1 & 0 \\
                             1 & 1 & 0 \\
                             0 & 0 & 0 \\
                           \end{array}
                         \right)=E_{11}+E_{12}+E_{21}+E_{22}
          \end{align*}
and
          \begin{align*}
         &\quad \chi(E_{11}+E_{12}+E_{13}+E_{31}+E_{32}+E_{33},C)\\
         &=\left(
             \begin{array}{ccc}
               1 & 0 & 0 \\
               0 & 0 & 1 \\
               0 & 1 & 0 \\
             \end{array}
           \right)\left(
                    \begin{array}{ccc}
                      1 & 1 & 1 \\
                      0 & 0 & 0 \\
                      1 & 1 & 1 \\
                    \end{array}
                  \right)\left(
                           \begin{array}{ccc}
                             1 & 0 & 0 \\
                             0 & 0 & 1 \\
                             0 & 1 & 0 \\
                           \end{array}
                         \right)\\
           &=\left(
                           \begin{array}{ccc}
                             1 & 1 & 1 \\
                             1 & 1 & 1 \\
                             0 & 0 & 0 \\
                           \end{array}
                         \right)=E_{11}+E_{12}+E_{13}+E_{21}+E_{22}+E_{23}.
          \end{align*}
By Theorem \ref{iso}, we finish the proof.
\end{proof}
\begin{rem}\label{simpsix}
By Lemma \ref{rank1-1}, Lemma \ref{rank1-2}, Lemma \ref{rank1-3} and Lemma \ref{rank1-4}, it remains to study the homological properties of
$\mathcal{A}_{\mathcal{O}_{-1}(k^3)}(M)$ when $M$ belongs to one of the following six specific matrixes:
\begin{align*}
M_1=\left(
                    \begin{array}{ccc}
                      0 & 1 & 1 \\
                      0 & 0 & 0 \\
                      0 & 0 & 0 \\
                    \end{array}
                  \right), M_2=\left(
                    \begin{array}{ccc}
                      0 & 1 & 0 \\
                      0 & 0 & 0 \\
                      0 & 0 & 0 \\
                    \end{array}
                  \right),M_3=\left(
                    \begin{array}{ccc}
                      1 & 1 & 1 \\
                      1 & 1 & 1 \\
                      0 & 0 & 0 \\
                    \end{array}
                  \right), \\
M_4=\left(
                    \begin{array}{ccc}
                      0 & 1 & 0 \\
                      0 & 0 & 0 \\
                      0 & 1 & 0 \\
                    \end{array}
                  \right), M_5=\left(
                    \begin{array}{ccc}
                      1 & 1 & 0 \\
                      1 & 1 & 0 \\
                      0 & 0 & 0 \\
                    \end{array}
                  \right),M_6=\left(
                    \begin{array}{ccc}
                      1 & 1 & 0 \\
                      0 & 0 & 0 \\
                      1 & 1 & 0 \\
                    \end{array}
                  \right).
\end{align*}
\end{rem}

\begin{prop}\label{sixcases}
For any $i\in \{1,2,3,4,5,6\}$,
the connected cochain DG algebra $\mathcal{A}_{\mathcal{O}_{-1}(k^3)}(M_i)$ is a Koszul Calabi-Yau DG algebra.
\end{prop}
\begin{proof}
For briefness, we denote $\mathcal{A}_i=\mathcal{A}_{\mathcal{O}_{-1}(k^3)}(M_i), i=1,2,\cdots, 6$. We will prove one by one that each $\mathcal{A}_i$ is a Koszul Calabi-Yau DG algebra.

(1)We have $\partial_{\mathcal{A}_1}(x_1)=x_1^2+x_2^2, \partial_{\mathcal{A}_1}(x_2)=\partial_{\mathcal{A}_1}(x_3)=0$. According to the constructing procedure of the minimal semi-free resolution in \cite[Proposition 2.4]{MW1}, we get a minimal semi-free resolution $f_1: F_1\stackrel{\simeq}{\to} k$, where $F_1$ is a semi-free DG $\mathcal{A}_1$-module such that $$F_1^{\#}=\mathcal{A}_1\oplus \mathcal{A}_1e_{1}\oplus \mathcal{A}_1 e_2 \oplus \mathcal{A}_1 e_3\oplus \mathcal{A}_1 e_4\oplus \mathcal{A}_1 e_5\oplus \mathcal{A}_1 e_6\oplus \mathcal{A}_1 e_7,$$  with a differential $\partial_F$ defined by \begin{align*}
\left(
                         \begin{array}{c}
                           \partial_{F_1}(1)\\
                           \partial_{F_1} (e_1)\\
                            \partial_{F_1}(e_2)\\
                            \partial_{F_1}(e_3)\\
                             \partial_{F_1}(e_4)\\
                              \partial_{F_1}(e_5)\\
                               \partial_{F_1}(e_6)\\
                                \partial_{F_1}(e_7)
                         \end{array}
                       \right)=\left(
                                 \begin{array}{cccccccc}
                                   0   & 0   & 0   & 0   & 0   & 0  & 0  & 0\\
                                   x_2 & 0   & 0   & 0   & 0   & 0  & 0  & 0\\
                                   x_3 & 0   & 0   & 0   & 0   & 0  & 0  & 0\\
                                   0   & x_3 & x_2 & 0   & 0   & 0  & 0  & 0\\
                                   x_1 & x_2 & x_3 & 0   & 0   & 0  & 0  & 0\\
                                   0   &  0  & x_1 & x_2 & x_3 & 0  & 0  & 0\\
                                   0   & x_1 &  0  & x_3 & x_2 & 0  & 0  & 0\\
                                   0   & 0   &  0  & x_1 & 0   & x_2& x_3& 0
                                 \end{array}
                               \right)
                       \left(
                         \begin{array}{c}
                           1\\
                           e_1\\
                           e_2\\
                           e_3\\
                           e_4\\
                           e_5\\
                           e_6\\
                           e_7
                            \end{array}
                       \right).
\end{align*}
Let $$ D_1=\left(
                                 \begin{array}{cccccccc}
                                   0   & 0   & 0   & 0   & 0   & 0  & 0  & 0\\
                                   x_2 & 0   & 0   & 0   & 0   & 0  & 0  & 0\\
                                   x_3 & 0   & 0   & 0   & 0   & 0  & 0  & 0\\
                                   0   & x_3 & x_2 & 0   & 0   & 0  & 0  & 0\\
                                   x_1 & x_2 & x_3 & 0   & 0   & 0  & 0  & 0\\
                                   0   &  0  & x_1 & x_2 & x_3 & 0  & 0  & 0\\
                                   0   & x_1 &  0  & x_3 & x_2 & 0  & 0  & 0\\
                                   0   & 0   &  0  & x_1 & 0   & x_2& x_3& 0
                                 \end{array}
                               \right).$$
By the minimality of $F_1$, we have \begin{align*}
H(\Hom_{\mathcal{A}_1}(F_1,k))&=\Hom_{\mathcal{A}_1}(F_1,k)\\
&= k 1^*\oplus [\bigoplus\limits_{i=1}^7 k(e_i)^*].
\end{align*}
So the Ext-algebra $E_1=H(\Hom_{\mathcal{A}_1}(F_1,F_1))$  is concentrated in degree $0$. On the other hand, $$\Hom_{\mathcal{A}_1}(F_1,F_1)^{\#}\cong \{k 1^*\oplus \oplus [\bigoplus\limits_{i=1}^7 k(e_i)^*]\}\otimes_{k} F_1^{\#}$$ is concentrated in degree $\ge 0$. This implies that $E_1= Z^0(\Hom_{\mathcal{A}_1}(F_1,F_1))$.
Since $F_1^{\#}$ is a free graded $\mathcal{A}_1^{\#}$-module with a basis
$\{1, e_1,e_2, e_3,e_4,e_5,e_6, e_7 \}$ concentrated in degree $0$,
  the elements in  $\Hom_{\mathcal{A}_1}(F_1,F_1)^0$ is one to one correspondence with the matrixes in $M_8(k)$. Indeed, any $f_1\in \Hom_{\mathcal{A}_1}(F_1,F_1)^0$ is uniquely determined by
  a matrix $A_{f_1}=(a_{ij})_{8\times 8}\in M_8(k)$ with
$$\left(
                         \begin{array}{c}
                          f_1(1) \\
                          f_1(e_1)\\
                          f_1(e_2)\\
                           f_1(e_3)\\
                            f_1(e_4)\\
                             f_1(e_5)\\
                              f_1(e_6)\\
                               f_1( e_7)
                         \end{array}
                       \right) =      A_{f_1} \cdot \left(
                         \begin{array}{c}
                          1 \\
                          e_1\\
                          e_2\\
                          e_3\\
                          e_4\\
                          e_5\\
                          e_6\\
                          e_7
                         \end{array}
                       \right).  $$
                       And $f_1\in  Z^0[\Hom_{\mathcal{A}}(F_1,F_1)]$ if and only if $\partial_{F_1}\circ f_1=f_1\circ \partial_{F_1}$, if and only if
 $ A_{f_1} D_1 = D_1 A_{f_1},$ which is also equivalent to
                       $$\begin{cases}
                       a_{ij}=0, \forall i<j\\
                       a_{11}=a_{22}=a_{33}= a_{44}=a_{55}=a_{66}= a_{77}=a_{88}\\
                       a_{21}=a_{43}=a_{52}=a_{64}=a_{75}=a_{86}\\
                       a_{31}=a_{42}=a_{53}=a_{65}=a_{74}=a_{87}\\
                       a_{41}=a_{62}=a_{73}=a_{85}\\
                       a_{51}=a_{63}=a_{72}=a_{84}\\
                       a_{61}=a_{82}, a_{71}=a_{83}\\
                       a_{32}=a_{54}=a_{76}=0
                       \end{cases}$$
by direct computations.
Hence the algebra $$ E_1\cong \left\{\left(
                                                             \begin{array}{cccccccc}
                                                               a & 0 & 0 & 0 & 0 & 0 & 0 & 0 \\
                                                               b & a & 0 & 0 & 0 & 0 & 0 & 0 \\
                                                               c & 0 & a & 0 & 0 & 0 & 0 & 0 \\
                                                               d & c & b & a & 0 & 0 & 0 & 0 \\
                                                               e & b & c & 0 & a & 0 & 0 & 0 \\
                                                               f & d & e & b & c & a & 0 & 0 \\
                                                               g & e & d & c & b & 0 & a & 0 \\
                                                               h & f & g & e & d & b & c & a \\
                                                             \end{array}
                                                           \right)
\quad | \quad a,b,c,d,e,f,g,h\in k \right\} = \mathcal{E}_1.$$
Set \begin{align*}
&\xi_1=\sum\limits_{i=1}^8E_{ii}, \\
& \xi_2=E_{21}+E_{43}+E_{52}+E_{64}+E_{75}+E_{86},\\
&\xi_3 =E_{31}+E_{42}+E_{53}+E_{65}+E_{74}+E_{87},\\
& \xi_4=E_{41}+E_{62}+E_{73}+E_{85},\\
&\xi_5=E_{51}+E_{63}+E_{72}+E_{84},\\
& \xi_6=E_{61}+E_{82},\\
&\xi_7=E_{71}+E_{83},\\
&\xi_8 =E_{81}.
\end{align*}
One sees that $\{\xi_i|i=1,2,\cdots, 8\}$ is a $k$-linear basis of $\mathcal{E}_1$ and
\begin{align*}
\begin{cases}
\xi_i\xi_1=\xi_1\xi_i, \forall i=1,2,\cdots, 8;\\
\xi_2^2=\xi_5, \xi_2\xi_3=\xi_3\xi_2=\xi_4,\xi_2\xi_4=\xi_4\xi_2=\xi_6,\xi_2\xi_5=\xi_5\xi_2=\xi_7, \\
\xi_2\xi_6=\xi_6\xi_2=\xi_8, \xi_2\xi_7=\xi_7\xi_2=0, \xi_2\xi_8=\xi_8\xi_2=0;\\
\xi_3^2=\xi_5, \xi_3\xi_4=\xi_4\xi_3=\xi_7,\xi_3\xi_5=\xi_5\xi_3=\xi_6,\\
\xi_3\xi_6=\xi_6\xi_3=0,\xi_3\xi_7=\xi_7\xi_3=\xi_8, \xi_3\xi_8=\xi_8\xi_3=0;\\
\xi_4\xi_5=\xi_5\xi_4=\xi_8,\xi_4^2=0,\xi_4\xi_i=\xi_i\xi_4=0,\forall i\in  \{6,7,8\};\\
\xi_j\xi_i=\xi_i\xi_j=0, \forall i, j\in \{5,6,7,8\}.
\end{cases}
\end{align*}
It is easy to check that the map
 \begin{align*}
 \varepsilon_1: \mathcal{E}_{1}\to \Hom_k(\mathcal{E}_{1},k)
 \end{align*}
 defined by
 $$\varepsilon_1: \quad
\begin{aligned}
\xi_1 & \to  \xi_8^*\\
\xi_2 & \to  \xi_6^*\\
\xi_3 & \to  \xi_7^*\\
\xi_4 & \to  \xi_5^*     \\
\xi_5 & \to  \xi_4^*    \\
\xi_6 & \to  \xi_2^*   \\
\xi_7 & \to  \xi_3^*    \\
\xi_8 & \to  \xi_1^*.\\
\end{aligned}
$$
is an isomorphism of left $\mathcal{E}_{1}$-modules. Thus  $\mathcal{E}_{1}$ is a commutative Frobenius algebra. Actually,
the morphism $\theta_1: \mathcal{E}_1\to k[x,y]/(x^2-y^2, x^4)$ of $k$-algebras defined by
$$\theta_1: \quad
\begin{aligned}
\xi_1 & \to 1\\
\xi_2 & \to \bar{x}\\
\xi_3 & \to \bar{y}\\
\xi_4 & \to \bar{x}\bar{y}\\
\xi_5 & \to \bar{x}^2\\
\xi_6 & \to \bar{x}^2\bar{y}\\
\xi_7 & \to \bar{x}^3\\
\xi_8 & \to \bar{x}^3\bar{y}.\\
\end{aligned}
$$
is an isomorphism.  Hence
$E_1$ is a symmetric Frobenius algebra concentrated
in degree $0$. This implies that
$\mathrm{Tor}_{\mathcal{A}_1}(k_{\mathcal{A}_1},{}_{\mathcal{A}_1}k)\cong
E^*$ is a symmetric coalgebra. By Remark \ref{methods},
$\mathcal{A}_1$ is a Koszul Calabi-Yau DG algebra.

(2)We have $\partial_{\mathcal{A}_2}(x_1)=x_2^2, \partial_{\mathcal{A}_2}(x_2)=\partial_{\mathcal{A}_2}(x_3)=0$. According to the constructing procedure of the minimal semi-free resolution in \cite[Proposition 2.4]{MW1}, we get a minimal semi-free resolution $f_2: F_2\stackrel{\simeq}{\to} k$, where $F_2$ is a semi-free DG $\mathcal{A}_2$-module such that $$F_2^{\#}=\mathcal{A}_2\oplus \mathcal{A}_2e_{1}\oplus \mathcal{A}_2 e_2 \oplus \mathcal{A}_2 e_3\oplus \mathcal{A}_2 e_4,$$  with a differential $\partial_{F_2}$ defined by \begin{align*}
\left(
                         \begin{array}{c}
                           \partial_{F_2}(1)\\
                           \partial_{F_2} (e_1)\\
                            \partial_{F_2}(e_2)\\
                            \partial_{F_2}(e_3)\\
                             \partial_{F_2}(e_4)
                         \end{array}
                       \right)=\left(
                                 \begin{array}{ccccc}
                                   0   & 0   & 0   & 0   & 0 \\
                                   x_2 & 0   & 0   & 0   & 0  \\
                                   x_3 & 0   & 0   & 0   & 0  \\
                                   x_1 & x_2 & 0   & 0   & 0  \\
                                   0   & x_1 & 0   & x_2 & 0
                                 \end{array}
                               \right)
                       \left(
                         \begin{array}{c}
                           1\\
                           e_1\\
                           e_2\\
                           e_3\\
                           e_4
                            \end{array}
                       \right).
\end{align*}
Let $$D_2=\left(
                                 \begin{array}{ccccc}
                                   0   & 0   & 0   & 0   & 0 \\
                                   x_2 & 0   & 0   & 0   & 0  \\
                                   x_3 & 0   & 0   & 0   & 0  \\
                                   x_1 & x_2 & 0   & 0   & 0  \\
                                   0   & x_1 & 0   & x_2 & 0
                                 \end{array}
                               \right).$$
By the minimality of $F_2$, we have \begin{align*}
H(\Hom_{\mathcal{A}_2}(F_2,k))&=\Hom_{\mathcal{A}_2}(F_2,k)\\
&= k 1^*\oplus [\bigoplus\limits_{i=1}^4 k(e_i)^*].
\end{align*}
So the Ext-algebra $E_2=H(\Hom_{\mathcal{A}_2}(F_2,F_2))$  is concentrated in degree $0$. On the other hand, $$\Hom_{\mathcal{A}_2}(F_2,F_2)^{\#}\cong \{k 1^*\oplus \oplus [\bigoplus\limits_{i=1}^4 k(e_i)^*]\}\otimes_{k} F_2^{\#}$$ is concentrated in degree $\ge 0$. This implies that $E_2= Z^0(\Hom_{\mathcal{A}_2}(F_2,F_2))$.
Since $F_2^{\#}$ is a free graded $\mathcal{A}_2^{\#}$-module with a basis
$\{1, e_1,e_2, e_3,e_4 \}$ concentrated in degree $0$,
  the elements in  $\Hom_{\mathcal{A}_2}(F_2,F_2)^0$ is one to one correspondence with the matrixes in $M_5(k)$. Indeed, any $f_2\in \Hom_{\mathcal{A}_2}(F_2,F_2)^0$ is uniquely determined by
  a matrix $A_{f_2}=(a_{ij})_{5\times 5}\in M_5(k)$ with
$$\left(
                         \begin{array}{c}
                          f_2(1) \\
                          f_2(e_1)\\
                          f_2(e_2)\\
                           f_2(e_3)\\
                            f_2(e_4)
                         \end{array}
                       \right) =      A_{f_2} \cdot \left(
                         \begin{array}{c}
                          1 \\
                          e_1\\
                          e_2\\
                          e_3\\
                          e_4
                         \end{array}
                       \right).  $$
                       And $f_2\in  Z^0[\Hom_{\mathcal{A}_2}(F_2,F_2)]$ if and only if $\partial_{F_2}\circ f_2=f_2\circ \partial_{F_2}$, if and only if
 $ A_{f_2} D_2 = D_2 A_{f_2},$ which is also equivalent to
                       $$\begin{cases}
                       a_{ij}=0, \forall i<j\\
                       a_{11}=a_{22}=a_{33}= a_{44}=a_{55}\\
                       a_{21}=a_{42}=a_{54}\\
                       a_{32}=a_{43}=a_{53}=0\\
                       a_{41}=a_{52}
                       \end{cases}$$
by direct computations. Hence the algebra $$ E_2\cong \left\{\left(
                                                             \begin{array}{ccccc}
                                                               a & 0 & 0 & 0 & 0 \\
                                                               b & a & 0 & 0 & 0  \\
                                                               c & 0 & a & 0 & 0  \\
                                                               d & b & 0 & a & 0 \\
                                                               e & d & 0 & b & a
                                                             \end{array}
                                                           \right)
\quad | \quad a,b,c,d,e\in k \right\} = \mathcal{E}_2.$$
Set \begin{align*}
&\xi_1=\sum\limits_{i=1}^5E_{ii}, \\
&\xi_2=E_{21}+E_{42}+E_{54},\\
&\xi_3 =E_{31},\\
& \xi_4=E_{41}+E_{52},\\
&\xi_5=E_{51}.
\end{align*}
One sees that $\{\xi_i|i=1,2,\cdots, 5\}$ is a $k$-linear basis of $\mathcal{E}_2$ and
\begin{align*}
\begin{cases}
\xi_i\xi_1=\xi_1\xi_i, \forall i=1,2,\cdots, 5;\\
\xi_2^2=\xi_4, \xi_2\xi_3=\xi_3\xi_2=0,\xi_2\xi_4=\xi_4\xi_2=\xi_5,\xi_2\xi_5=\xi_5\xi_2=0, \\
\xi_3^2=0, \xi_3\xi_4=\xi_4\xi_3=0,\xi_3\xi_5=\xi_5\xi_3=0,\\
\xi_4^2=0,\xi_4\xi_5=\xi_5\xi_4=0,\xi_5^2=0.
\end{cases}
\end{align*}
It is easy to check that the map
 \begin{align*}
 \varepsilon_2: \mathcal{E}_{2}\to \Hom_k(\mathcal{E}_{2},k)
 \end{align*}
 defined by
 $$\varepsilon_2: \quad
\begin{aligned}
\xi_1 & \to  \xi_5^*\\
\xi_2 & \to  \xi_4^*\\
\xi_3 & \to  \xi_3^*\\
\xi_4 & \to  \xi_2^*     \\
\xi_5 & \to  \xi_1^*.
\end{aligned}
$$
is an isomorphism of left $\mathcal{E}_{2}$-modules. Thus  $\mathcal{E}_{2}$ is a commutative Frobenius algebra. Actually,
the morphism $\theta_2: \mathcal{E}_2\to k[x,y]/(x^4,xy, y^2)$ of $k$-algebras defined by
$$\theta_2: \quad
\begin{aligned}
\xi_1 & \to 1\\
\xi_2 & \to \bar{x}\\
\xi_3 & \to \bar{y}\\
\xi_4 & \to \bar{x}^2\\
\xi_5 & \to \bar{x}^3.
\end{aligned}
$$
is an isomorphism.  Hence
$E_2$ is a symmetric Frobenius algebra concentrated
in degree $0$. This implies that
$\mathrm{Tor}_{\mathcal{A}_2}(k_{\mathcal{A}_2},{}_{\mathcal{A}_2}k)\cong
E^*_2$ is a symmetric coalgebra. By Remark \ref{methods},
$\mathcal{A}_2$ is a Koszul Calabi-Yau DG algebra.

(3)We have $\partial_{\mathcal{A}_3}(x_1)=\sum\limits_{i=1}^3x_i^2=\partial_{\mathcal{A}_3}(x_2),\partial_{\mathcal{A}_3}(x_3)=0$. By the constructing procedure of the minimal semi-free resolution in \cite[Proposition 2.4]{MW1}, we get a minimal semi-free resolution $f_3: F_3\stackrel{\simeq}{\to} k$, where $F_3$ is a semi-free DG $\mathcal{A}_3$-module such that $$F_3^{\#}=\mathcal{A}_3\oplus \mathcal{A}_3e_{1}\oplus \mathcal{A}_3 e_2 \oplus \mathcal{A}_2 e_3,$$  with a differential $\partial_{F_3}$ defined by \begin{align*}
\left(
                         \begin{array}{c}
                           \partial_{F_3}(1)\\
                           \partial_{F_3} (e_1)\\
                            \partial_{F_3}(e_2)\\
                            \partial_{F_3}(e_3)
                         \end{array}
                       \right)=\left(
                                 \begin{array}{cccc}
                                   0       & 0       & 0      & 0\\
                                   x_1-x_2 & 0       & 0      & 0\\
                                   x_3     & 0       & 0      & 0\\
                                   x_1     & x_1-x_2 & x_3    & 0
                                 \end{array}
                               \right)
                       \left(
                         \begin{array}{c}
                           1\\
                           e_1\\
                           e_2\\
                           e_3
                            \end{array}
                       \right).
\end{align*}
Let $$D_3=\left(
                                 \begin{array}{cccc}
                                   0       & 0       & 0      & 0\\
                                   x_1-x_2 & 0       & 0      & 0\\
                                   x_3     & 0       & 0      & 0\\
                                   x_1     & x_1-x_2 & x_3    & 0
                                 \end{array}
                               \right).$$
By the minimality of $F_3$, we have \begin{align*}
H(\Hom_{\mathcal{A}_3}(F_3,k))&=\Hom_{\mathcal{A}_3}(F_3,k)\\
&= k 1^*\oplus [\bigoplus\limits_{i=1}^3 k(e_i)^*].
\end{align*}
So the Ext-algebra $E_3=H(\Hom_{\mathcal{A}_3}(F_3,F_3))$  is concentrated in degree $0$. On the other hand, $$\Hom_{\mathcal{A}_3}(F_3,F_3)^{\#}\cong \{k 1^*\oplus \oplus [\bigoplus\limits_{i=1}^3 k(e_i)^*]\}\otimes_{k} F_3^{\#}$$ is concentrated in degree $\ge 0$. This implies that $E_3= Z^0(\Hom_{\mathcal{A}_2}(F_3,F_3))$.
Since $F_3^{\#}$ is a free graded $\mathcal{A}_3^{\#}$-module with a basis
$\{1, e_1,e_2, e_3\}$ concentrated in degree $0$,
  the elements in  $\Hom_{\mathcal{A}_3}(F_3,F_3)^0$ is one to one correspondence with the matrixes in $M_4(k)$. Indeed, any $f_3\in \Hom_{\mathcal{A}_3}(F_3,F_3)^0$ is uniquely determined by
  a matrix $A_{f_3}=(a_{ij})_{4\times 4}\in M_4(k)$ with
$$\left(
                         \begin{array}{c}
                          f_2(1) \\
                          f_2(e_1)\\
                          f_2(e_2)\\
                           f_2(e_3)
                         \end{array}
                       \right) =      A_{f_3} \cdot \left(
                         \begin{array}{c}
                          1 \\
                          e_1\\
                          e_2\\
                          e_3
                         \end{array}
                       \right).  $$
                       And $f_3\in  Z^0[\Hom_{\mathcal{A}_3}(F_3,F_3)]$ if and only if $\partial_{F_3}\circ f_3=f_3\circ \partial_{F_3}$, if and only if
 $ A_{f_3} D_3 = D_3 A_{f_3},$ which is also equivalent to
                       $$\begin{cases}
                       a_{ij}=0, \forall i<j\\
                       a_{11}=a_{22}=a_{33}= a_{44}\\
                       a_{21}=a_{42}\\
                       a_{32}=0\\
                       a_{43}=a_{31}
                       \end{cases}$$
by direct computations. Hence the algebra $$ E_3\cong \left\{\left(
                                                             \begin{array}{cccc}
                                                               a & 0 & 0 & 0  \\
                                                               b & a & 0 & 0   \\
                                                               c & 0 & a & 0   \\
                                                               d & b & c & a
                                                             \end{array}
                                                           \right)
\quad | \quad a,b,c,d\in k \right\} = \mathcal{E}_3.$$
Set \begin{align*}
&\xi_1=\sum\limits_{i=1}^4E_{ii}, \\
&\xi_2=E_{21}+E_{42},\\
&\xi_3 =E_{31}+E_{43},\\
& \xi_4=E_{41}.
\end{align*}
One sees that $\{\xi_i|i=1,2,\cdots, 4\}$ is a $k$-linear basis of $\mathcal{E}_3$ and
\begin{align*}
\begin{cases}
\xi_i\xi_1=\xi_1\xi_i, \forall i=1,2,\cdots, 4;\\
\xi_2^2=\xi_4, \xi_2\xi_3=\xi_3\xi_2=0,\xi_2\xi_4=\xi_4\xi_2=0, \\
\xi_3^2=\xi_4, \xi_3\xi_4=\xi_4\xi_3=0, \xi_4^2=0.
\end{cases}
\end{align*}
It is easy to check that the map
 \begin{align*}
 \varepsilon_3: \mathcal{E}_{3}\to \Hom_k(\mathcal{E}_{3},k)
 \end{align*}
 defined by
 $$\varepsilon_3: \quad
\begin{aligned}
\xi_1 & \to  \xi_4^*\\
\xi_2 & \to  \xi_2^*\\
\xi_3 & \to  \xi_3^*\\
\xi_4 & \to  \xi_1^*.
\end{aligned}
$$
is an isomorphism of left $\mathcal{E}_{3}$-modules. Thus  $\mathcal{E}_{3}$ is a commutative Frobenius algebra. Actually,
the morphism $\theta_3: \mathcal{E}_3\to k[x,y]/(x^3,xy, y^3,x^2-y^2)$ of $k$-algebras defined by
$$\theta_3: \quad
\begin{aligned}
\xi_1 & \to 1\\
\xi_2 & \to \bar{x}\\
\xi_3 & \to \bar{y}\\
\xi_4 & \to \bar{x}^2.
\end{aligned}
$$
is an isomorphism.  Hence
$E_3$ is a symmetric Frobenius algebra concentrated
in degree $0$. This implies that
$\mathrm{Tor}_{\mathcal{A}_3}(k_{\mathcal{A}_3},{}_{\mathcal{A}_3}k)\cong
E^*_3$ is a symmetric coalgebra. By Remark \ref{methods},
$\mathcal{A}_3$ is a Koszul Calabi-Yau DG algebra.

(4)We have $\partial_{\mathcal{A}_4}(x_1)=x_2^2=\partial_{\mathcal{A}_4}(x_3),\partial_{\mathcal{A}_4}(x_2)=0$. According to the constructing procedure of the minimal semi-free resolution in \cite[Proposition 2.4]{MW1}, we get a minimal semi-free resolution $f_4: F_4\stackrel{\simeq}{\to} k$, where $F_4$ is a semi-free DG $\mathcal{A}_4$-module such that $$F_4^{\#}=\mathcal{A}_4\oplus \mathcal{A}_4e_{1}\oplus \mathcal{A}_4 e_2 \oplus \mathcal{A}_4 e_3\oplus \mathcal{A}_4 e_4,$$  with a differential $\partial_{F_4}$ defined by \begin{align*}
\left(
                         \begin{array}{c}
                           \partial_{F_4}(1)\\
                           \partial_{F_4} (e_1)\\
                            \partial_{F_4}(e_2)\\
                            \partial_{F_4}(e_3)\\
                            \partial_{F_4}(e_4)
                         \end{array}
                       \right)=\left(
                                 \begin{array}{ccccc}
                                   0       & 0       & 0      & 0  & 0\\
                                 x_2       & 0       & 0      & 0  & 0\\
                                 x_1-x_3   & 0       & 0      & 0  & 0\\
                                   x_1     & x_2     & 0      & 0  & 0 \\
                                   0       & x_1     & 0      & x_2& 0
                                 \end{array}
                               \right)
                       \left(
                         \begin{array}{c}
                           1\\
                           e_1\\
                           e_2\\
                           e_3\\
                           e_4
                            \end{array}
                       \right).\end{align*}
Let $$D_4=\left(
                                 \begin{array}{ccccc}
                                   0       & 0       & 0      & 0  & 0\\
                                 x_2       & 0       & 0      & 0  & 0\\
                                 x_1-x_3   & 0       & 0      & 0  & 0\\
                                   x_1     & x_2     & 0      & 0  & 0 \\
                                   0       & x_1     & 0      & x_2& 0
                                 \end{array}
                               \right).$$
By the minimality of $F_4$, we have \begin{align*}
H(\Hom_{\mathcal{A}_4}(F_4,k))&=\Hom_{\mathcal{A}_4}(F_4,k)\\
&= k 1^*\oplus [\bigoplus\limits_{i=1}^4 k(e_i)^*].
\end{align*}
So the Ext-algebra $E_4=H(\Hom_{\mathcal{A}_4}(F_4,F_4))$  is concentrated in degree $0$. On the other hand, $$\Hom_{\mathcal{A}_4}(F_4,F_4)^{\#}\cong \{k 1^*\oplus \oplus [\bigoplus\limits_{i=1}^4 k(e_i)^*]\}\otimes_{k} F_4^{\#}$$ is concentrated in degree $\ge 0$. This implies that $E_4= Z^0(\Hom_{\mathcal{A}_4}(F_4,F_4))$.
Since $F_4^{\#}$ is a free graded $\mathcal{A}_4^{\#}$-module with a basis
$\{1, e_1,e_2, e_3,e_4\}$ concentrated in degree $0$,
  the elements in  $\Hom_{\mathcal{A}_4}(F_4,F_4)^0$ is one to one correspondence with the matrixes in $M_5(k)$. Indeed, any $f_4\in \Hom_{\mathcal{A}_4}(F_4,F_4)^0$ is uniquely determined by
  a matrix $A_{f_4}=(a_{ij})_{5\times 5}\in M_5(k)$ with
$$\left(
                         \begin{array}{c}
                          f_4(1) \\
                          f_4(e_1)\\
                          f_4(e_2)\\
                          f_4(e_3)\\
                          f_4(e_4)
                         \end{array}
                       \right) =      A_{f_4} \cdot \left(
                         \begin{array}{c}
                          1 \\
                          e_1\\
                          e_2\\
                          e_3\\
                          e_4
                         \end{array}
                       \right).  $$
                       And $f_4\in  Z^0[\Hom_{\mathcal{A}_4}(F_4,F_4)]$ if and only if $\partial_{F_4}\circ f_4=f_4\circ \partial_{F_4}$, if and only if
 $ A_{f_4} D_4 = D_4 A_{f_4},$ which is also equivalent to
                       $$\begin{cases}
                       a_{ij}=0, \forall i<j\\
                       a_{11}=a_{22}=a_{33}= a_{44}=a_{55}\\
                       a_{21}=a_{42}=a_{54}\\
                       a_{32}=a_{43}=a_{53}=0\\
                       a_{41}=a_{52}
                       \end{cases}$$
by direct computations. Hence the algebra $$ E_4\cong \left\{\left(
                                                             \begin{array}{ccccc}
                                                               a & 0 & 0 & 0  & 0\\
                                                               b & a & 0 & 0  & 0\\
                                                               c & 0 & a & 0  & 0\\
                                                               d & b & 0 & a  & 0 \\
                                                               e & d & 0 & b  & 0
                                                             \end{array}
                                                           \right)
\quad | \quad a,b,c,d,e\in k \right\} = \mathcal{E}_4.$$
Set \begin{align*}
&\xi_1=\sum\limits_{i=1}^5E_{ii}, \\
&\xi_2=E_{21}+E_{42}+E_{54},\\
&\xi_3 =E_{31},\\
& \xi_4=E_{41}+E_{52},\\
&\xi_5=E_{51}.
\end{align*}
One sees that $\{\xi_i|i=1,2,\cdots, 5\}$ is a $k$-linear basis of $\mathcal{E}_4$ and
\begin{align*}
\begin{cases}
\xi_i\xi_1=\xi_1\xi_i, \forall i=1,2,\cdots, 5;\\
\xi_2^2=\xi_4, \xi_2\xi_3=\xi_3\xi_2=0,\xi_2\xi_4=\xi_4\xi_2=\xi_5,\xi_2\xi_5=\xi_5\xi_2=0, \\
\xi_3^2=0, \xi_3\xi_4=\xi_4\xi_3=0,\xi_3\xi_5=\xi_5\xi_3=0,\\
\xi_4^2=0,\xi_4\xi_5=\xi_5\xi_4=0,\xi_5^2=0.
\end{cases}
\end{align*}
It is easy to check that the map
 \begin{align*}
 \varepsilon_2: \mathcal{E}_{4}\to \Hom_k(\mathcal{E}_{4},k)
 \end{align*}
 defined by
 $$\varepsilon_2: \quad
\begin{aligned}
\xi_1 & \to  \xi_5^*\\
\xi_2 & \to  \xi_4^*\\
\xi_3 & \to  \xi_3^*\\
\xi_4 & \to  \xi_2^*     \\
\xi_5 & \to  \xi_1^*
\end{aligned}
$$
is an isomorphism of left $\mathcal{E}_{4}$-modules. Thus  $\mathcal{E}_{4}$ is a commutative Frobenius algebra. Actually,
the morphism $\theta_2: \mathcal{E}_4\to k[x,y]/(x^4,xy, y^2)$ of $k$-algebras defined by
$$\theta_4: \quad
\begin{aligned}
\xi_1 & \to 1\\
\xi_2 & \to \bar{x}\\
\xi_3 & \to \bar{y}\\
\xi_4 & \to \bar{x}^2\\
\xi_5 & \to \bar{x}^3.
\end{aligned}
$$
is an isomorphism.  Hence
$E_4$ is a symmetric Frobenius algebra concentrated
in degree $0$. This implies that
$\mathrm{Tor}_{\mathcal{A}_4}(k_{\mathcal{A}_4},{}_{\mathcal{A}_4}k)\cong
E^*_4$ is a symmetric coalgebra. By Remark \ref{methods},
$\mathcal{A}_4$ is a Koszul Calabi-Yau DG algebra.

(5)We have $\partial_{\mathcal{A}_5}(x_1)=x_1^2+x_2^2=\partial_{\mathcal{A}_5}(x_2),\partial_{\mathcal{A}_5}(x_3)=0$. By the constructing procedure of the minimal semi-free resolution in \cite[Proposition 2.4]{MW1}, we get a minimal semi-free resolution $f_5: F_5\stackrel{\simeq}{\to} k$, where $F_5$ is a semi-free DG $\mathcal{A}_5$-module such that $$F_5^{\#}=\mathcal{A}_5\oplus \mathcal{A}_5e_{1}\oplus \mathcal{A}_5 e_2 \oplus \mathcal{A}_5 e_3,$$  with a differential $\partial_{F_5}$ defined by \begin{align*}
\left(
                         \begin{array}{c}
                           \partial_{F_5}(1)\\
                           \partial_{F_5} (e_1)\\
                            \partial_{F_5}(e_2)\\
                            \partial_{F_5}(e_3)
                         \end{array}
                       \right)=\left(
                                 \begin{array}{cccc}
                                   0       & 0       & 0      & 0\\
                                   x_3 & 0       & 0      & 0\\
                                   x_1-x_2     & 0       & 0      & 0\\
                                   0     & x_1-x_2 & x_3    & 0
                                 \end{array}
                               \right)
                       \left(
                         \begin{array}{c}
                           1\\
                           e_1\\
                           e_2\\
                           e_3
                            \end{array}
                       \right).
\end{align*}
Let $$D_5=\left(
                                 \begin{array}{cccc}
                                   0       & 0       & 0      & 0\\
                                   x_3 & 0       & 0      & 0\\
                                   x_1-x_2     & 0       & 0      & 0\\
                                   0     & x_1-x_2 & x_3    & 0
                                 \end{array}
                               \right).$$
By the minimality of $F_5$, we have \begin{align*}
H(\Hom_{\mathcal{A}_5}(F_5,k))&=\Hom_{\mathcal{A}_5}(F_5,k)\\
&= k 1^*\oplus [\bigoplus\limits_{i=1}^3 k(e_i)^*].
\end{align*}
So the Ext-algebra $E_5=H(\Hom_{\mathcal{A}_5}(F_5,F_5))$  is concentrated in degree $0$. On the other hand, $$\Hom_{\mathcal{A}_5}(F_5,F_5)^{\#}\cong \{k 1^*\oplus \oplus [\bigoplus\limits_{i=1}^3 k(e_i)^*]\}\otimes_{k} F_5^{\#}$$ is concentrated in degree $\ge 0$. This implies that $E_5= Z^0(\Hom_{\mathcal{A}_5}(F_5,F_5))$.
Since $F_5^{\#}$ is a free graded $\mathcal{A}_5^{\#}$-module with a basis
$\{1, e_1,e_2, e_3\}$ concentrated in degree $0$,
  the elements in  $\Hom_{\mathcal{A}_5}(F_5,F_5)^0$ is one to one correspondence with the matrixes in $M_4(k)$. Indeed, any $f_5\in \Hom_{\mathcal{A}_5}(F_5,F_5)^0$ is uniquely determined by
  a matrix $A_{f_5}=(a_{ij})_{4\times 4}\in M_4(k)$ with
$$\left(
                         \begin{array}{c}
                          f_2(1) \\
                          f_2(e_1)\\
                          f_2(e_2)\\
                           f_2(e_3)
                         \end{array}
                       \right) =      A_{f_5} \cdot \left(
                         \begin{array}{c}
                          1 \\
                          e_1\\
                          e_2\\
                          e_3
                         \end{array}
                       \right).  $$
                       And $f_5\in  Z^0[\Hom_{\mathcal{A}_5}(F_5,F_5)]$ if and only if $\partial_{F_5}\circ f_5=f_5\circ \partial_{F_5}$, if and only if
 $ A_{f_5} D_5 = D_5 A_{f_5},$ which is also equivalent to
                       $$\begin{cases}
                       a_{ij}=0, \forall i<j\\
                       a_{11}=a_{22}=a_{33}= a_{44}\\
                       a_{21}=a_{43}\\
                       a_{32}=0\\
                       a_{31}=a_{42}
                       \end{cases}$$
by direct computations. Hence the algebra $$ E_5\cong \left\{\left(
                                                             \begin{array}{cccc}
                                                               a & 0 & 0 & 0  \\
                                                               b & a & 0 & 0   \\
                                                               c & 0 & a & 0   \\
                                                               d & c & b & a
                                                             \end{array}
                                                           \right)
\quad | \quad a,b,c,d\in k \right\} = \mathcal{E}_5.$$
Set \begin{align*}
&\xi_1=\sum\limits_{i=1}^4E_{ii}, \\
&\xi_2=E_{21}+E_{43},\\
&\xi_3 =E_{31}+E_{42},\\
& \xi_4=E_{41}.
\end{align*}
One sees that $\{\xi_i|i=1,2,\cdots, 4\}$ is a $k$-linear basis of $\mathcal{E}_5$ and
\begin{align*}
\begin{cases}
\xi_i\xi_1=\xi_1\xi_i, \forall i=1,2,\cdots, 4;\\
\xi_2^2=0, \xi_2\xi_3=\xi_3\xi_2=\xi_4,\xi_2\xi_4=\xi_4\xi_2=0, \\
\xi_3^2=0, \xi_3\xi_4=\xi_4\xi_3=0, \xi_4^2=0.
\end{cases}
\end{align*}
It is easy to check that the map
 \begin{align*}
 \varepsilon_5: \mathcal{E}_{5}\to \Hom_k(\mathcal{E}_{5},k)
 \end{align*}
 defined by
 $$\varepsilon_5: \quad
\begin{aligned}
\xi_1 & \to  \xi_4^*\\
\xi_2 & \to  \xi_2^*\\
\xi_3 & \to  \xi_3^*\\
\xi_4 & \to  \xi_1^*.
\end{aligned}
$$
is an isomorphism of left $\mathcal{E}_{5}$-modules. Thus  $\mathcal{E}_{5}$ is a commutative Frobenius algebra. Actually,
the morphism $\theta_5: \mathcal{E}_5\to k[x]/(x^4)$ of $k$-algebras defined by
$$\theta_5: \quad
\begin{aligned}
\xi_1 & \to 1\\
\xi_2 & \to \bar{x}\\
\xi_3 & \to \bar{x}^2\\
\xi_4 & \to \bar{x}^3.
\end{aligned}
$$
is an isomorphism.  Hence
$E_5$ is a symmetric Frobenius algebra concentrated
in degree $0$. This implies that
$\mathrm{Tor}_{\mathcal{A}_5}(k_{\mathcal{A}_5},{}_{\mathcal{A}_5}k)\cong
E^*_5$ is a symmetric coalgebra. By Remark \ref{methods},
$\mathcal{A}_5$ is a Koszul Calabi-Yau DG algebra.

(6)We have $\partial_{\mathcal{A}_6}(x_1)=x_1^2+x_2^2=\partial_{\mathcal{A}_6}(x_3),\partial_{\mathcal{A}_6}(x_2)=0$. According to the constructing procedure of the minimal semi-free resolution in \cite[Proposition 2.4]{MW1}, we get a minimal semi-free resolution $f_6: F_6\stackrel{\simeq}{\to} k$, where $F_6$ is a semi-free DG $\mathcal{A}_6$-module such that $$F_6^{\#}=\mathcal{A}_6\oplus \mathcal{A}_6e_{1}\oplus \mathcal{A}_6 e_2 \oplus \mathcal{A}_6 e_3,$$  with a differential $\partial_{F_6}$ defined by \begin{align*}
\left(
                         \begin{array}{c}
                           \partial_{F_6}(1)\\
                           \partial_{F_6} (e_1)\\
                            \partial_{F_6}(e_2)\\
                            \partial_{F_6}(e_3)
                         \end{array}
                       \right)=\left(
                                 \begin{array}{cccc}
                                   0       & 0       & 0      & 0\\
                                   x_2 & 0       & 0      & 0\\
                                   x_1-x_3     & 0       & 0      & 0\\
                                   0     & x_1-x_3 & x_2   & 0
                                 \end{array}
                               \right)
                       \left(
                         \begin{array}{c}
                           1\\
                           e_1\\
                           e_2\\
                           e_3
                            \end{array}
                       \right).
\end{align*}
Let $$D_6=\left(
                                 \begin{array}{cccc}
                                   0       & 0       & 0      & 0\\
                                   x_2 & 0       & 0      & 0\\
                                   x_1-x_3     & 0       & 0      & 0\\
                                   0     & x_1-x_3 & x_2   & 0
                                 \end{array}
                               \right).$$
By the minimality of $F_6$, we have \begin{align*}
H(\Hom_{\mathcal{A}_6}(F_6,k))&=\Hom_{\mathcal{A}_6}(F_6,k)\\
&= k 1^*\oplus [\bigoplus\limits_{i=1}^3 k(e_i)^*].
\end{align*}
So the Ext-algebra $E_6=H(\Hom_{\mathcal{A}_6}(F_6,F_6))$  is concentrated in degree $0$. On the other hand, $$\Hom_{\mathcal{A}_6}(F_6,F_6)^{\#}\cong \{k 1^*\oplus \oplus [\bigoplus\limits_{i=1}^3 k(e_i)^*]\}\otimes_{k} F_6^{\#}$$ is concentrated in degree $\ge 0$. This implies that $E_6= Z^0(\Hom_{\mathcal{A}_6}(F_6,F_6))$.
Since $F_6^{\#}$ is a free graded $\mathcal{A}_6^{\#}$-module with a basis
$\{1, e_1,e_2, e_3\}$ concentrated in degree $0$,
  the elements in  $\Hom_{\mathcal{A}_6}(F_6,F_6)^0$ is one to one correspondence with the matrixes in $M_4(k)$. Indeed, any $f_6\in \Hom_{\mathcal{A}_6}(F_6,F_6)^0$ is uniquely determined by
  a matrix $A_{f_6}=(a_{ij})_{4\times 4}\in M_4(k)$ with
$$\left(
                         \begin{array}{c}
                          f_2(1) \\
                          f_2(e_1)\\
                          f_2(e_2)\\
                           f_2(e_3)
                         \end{array}
                       \right) =      A_{f_6} \cdot \left(
                         \begin{array}{c}
                          1 \\
                          e_1\\
                          e_2\\
                          e_3
                         \end{array}
                       \right).  $$
                       And $f_6\in  Z^0[\Hom_{\mathcal{A}_6}(F_6,F_6)]$ if and only if $\partial_{F_6}\circ f_6=f_6\circ \partial_{F_6}$, if and only if
 $ A_{f_6} D_6 = D_6 A_{f_6},$ which is also equivalent to
                       $$\begin{cases}
                       a_{ij}=0, \forall i<j\\
                       a_{11}=a_{22}=a_{33}= a_{44}\\
                       a_{21}=a_{43}\\
                       a_{32}=0\\
                       a_{31}=a_{42}
                       \end{cases}$$
by direct computations. Hence the algebra $$ E_6\cong \left\{\left(
                                                             \begin{array}{cccc}
                                                               a & 0 & 0 & 0  \\
                                                               b & a & 0 & 0   \\
                                                               c & 0 & a & 0   \\
                                                               d & c & b & a
                                                             \end{array}
                                                           \right)
\quad | \quad a,b,c,d\in k \right\} = \mathcal{E}_6.$$
Set \begin{align*}
&\xi_1=\sum\limits_{i=1}^4E_{ii}, \\
&\xi_2=E_{21}+E_{43},\\
&\xi_3 =E_{31}+E_{42},\\
& \xi_4=E_{41}.
\end{align*}
One sees that $\{\xi_i|i=1,2,\cdots, 4\}$ is a $k$-linear basis of $\mathcal{E}_6$ and
\begin{align*}
\begin{cases}
\xi_i\xi_1=\xi_1\xi_i, \forall i=1,2,\cdots, 4;\\
\xi_2^2=0, \xi_2\xi_3=\xi_3\xi_2=\xi_4,\xi_2\xi_4=\xi_4\xi_2=0, \\
\xi_3^2=0, \xi_3\xi_4=\xi_4\xi_3=0, \xi_4^2=0.
\end{cases}
\end{align*}
It is easy to check that the map
 \begin{align*}
 \varepsilon_6: \mathcal{E}_{5}\to \Hom_k(\mathcal{E}_{6},k)
 \end{align*}
 defined by
 $$\varepsilon_5: \quad
\begin{aligned}
\xi_1 & \to  \xi_4^*\\
\xi_2 & \to  \xi_2^*\\
\xi_3 & \to  \xi_3^*\\
\xi_4 & \to  \xi_1^*.
\end{aligned}
$$
is an isomorphism of left $\mathcal{E}_{6}$-modules. Thus  $\mathcal{E}_{5}$ is a commutative Frobenius algebra. Actually,
the morphism $\theta_5: \mathcal{E}_6\to k[x]/(x^4)$ of $k$-algebras defined by
$$\theta_6: \quad
\begin{aligned}
\xi_1 & \to 1\\
\xi_2 & \to \bar{x}\\
\xi_3 & \to \bar{x}^2\\
\xi_4 & \to \bar{x}^3.
\end{aligned}
$$
is an isomorphism.  Hence
$E_6$ is a symmetric Frobenius algebra concentrated
in degree $0$. This implies that
$\mathrm{Tor}_{\mathcal{A}_6}(k_{\mathcal{A}_6},{}_{\mathcal{A}_6}k)\cong
E^*_6$ is a symmetric coalgebra. By Remark \ref{methods},
$\mathcal{A}_6$ is a Koszul Calabi-Yau DG algebra.
\end{proof}

Now, we can reach the following conclusion: the DG algebras
$\mathcal{A}_{\mathcal{O}_{-1}(k^3)}(M)$ in  Case $2$, Case $3$ and Case $4$ are Koszul Calabi-Yau DG algebras.

\section{proof of theorem c}
\begin{proof}
First, let us prove the `if' part.  Suppose that there exists some  $C=(c_{ij})_{3\times 3}\in \mathrm{QPL}_3(k)$ satisfying $N=C^{-1}M(c_{ij}^2)_{3\times 3}$,
where $$M=\left(
                                 \begin{array}{ccc}
                                   1 & 1 & 0 \\
                                   1 & 1 & 0 \\
                                   1 & 1 & 0 \\
                                 \end{array}
                               \right)
\,\,\text{or}\,\,M=\left(
                                 \begin{array}{ccc}
                                   m_{11} & m_{12} & m_{13} \\
                                   l_1m_{11} & l_1m_{12} & l_1m_{13} \\
                                   l_2m_{11} & l_2m_{12} & l_2m_{13} \\
                                 \end{array}
                               \right)$$ with $m_{12}l_1^2+m_{13}l_2^2\neq m_{11}, l_1l_2\neq 0$ and $4m_{12}m_{13}l_1^2l_2^2= (m_{12}l_1^2+m_{13}l_2^2-m_{11})^2$. In both cases, $\mathcal{A}_{\mathcal{O}_{-1}(k^3)}(M)\cong \mathcal{A}_{\mathcal{O}_{-1}(k^3)}(N)$ by Theorem \ref{iso}. On the other hand, the DG algebra $\mathcal{A}_{\mathcal{O}_{-1}(k^3)}(M)$  is not Calabi-Yau by Proposition \ref{noncycase} and Proposition \ref{nonregsec}. Thus $\mathcal{A}_{\mathcal{O}_{-1}(k^3)}(N)$ is not Calabi-Yau.

                                It remains to show the `only if' part.
If $\mathcal{A}_{\mathcal{O}_{-1}(k^3)}(N)$ is not Calabi-Yau, then $r(N)\neq 3$ by Proposition \ref{easycases}(1), $r(N)\neq 0$ by \cite[Proposition 3.2]{MYY},  $r(N)\neq 2$ by Proposition \ref{easycases}(2) and Proposition \ref{case1cy}. So $r(N)=1$. By Remark \ref{redtosimp}, we have $\mathcal{A}_{\mathcal{O}_{-1}(k^3)}(N)\cong \mathcal{A}_{\mathcal{O}_{-1}(k^3)}(M)$, where
$$M=\left(
                                 \begin{array}{ccc}
                                   m_{11} & m_{12} & m_{13} \\
                                   l_1m_{11} & l_1m_{12} & l_1m_{13} \\
                                   l_2m_{11} & l_2m_{12} & l_2m_{13} \\
                                 \end{array}
                               \right),$$
 $(0,0,0)\neq (m_{11},m_{12},m_{13})\in k^3$ and  $l_1,l_2\in k$. By Proposition \ref{easycases}(3-5), Remark \ref{simpsix} and Proposition \ref{sixcases},  we have either
  $$l_1l_2\neq 0, m_{12}m_{13}=0\,\,\text{and} \,\, m_{12}l_1^2+m_{13}l_2^2= m_{11}$$
 or $$l_1l_2\neq 0, m_{12}l_1^2+m_{13}l_2^2\neq m_{11}, 4m_{12}m_{13}l_1^2l_2^2=(m_{12}l_1^2+m_{13}l_2^2-m_{11})^2.$$
By Proposition \ref{nonregsec}, there exists $B=(b_{ij})_{3\times 3}\in \mathrm{QPL}_3(k)$ such that
$$ B^{-1}M(b_{ij}^2)_{3\times 3}=\left(
                                 \begin{array}{ccc}
                                   1 & 1 & 0 \\
                                   1 & 1 & 0 \\
                                   1 & 1 & 0 \\
                                 \end{array}
                               \right),$$
if
$l_1l_2\neq 0, m_{12}m_{13}=0\,\,\text{and} \,\, m_{12}l_1^2+m_{13}l_2^2= m_{11}$. In this case,
$$\mathcal{A}_{\mathcal{O}_{-1}(k^3)}(N)\cong \mathcal{A}_{\mathcal{O}_{-1}(k^3)}(M)\cong \mathcal{A}_{\mathcal{O}_{-1}(k^3)}(Q)$$ by Theorem \ref{iso}, where $$Q=\left(
                                 \begin{array}{ccc}
                                   1 & 1 & 0 \\
                                   1 & 1 & 0 \\
                                   1 & 1 & 0 \\
                                 \end{array}
                               \right).$$
\end{proof}
\section*{Acknowledgments}
X.-F. Mao was supported by the National Natural Science Foundation of
China (No. 11871326). X.-T. Wang was supported by Simons Foundation Program: Mathematics and Physical Sciences-Collaboration Grants for Mathematician (Award No.688403).
\section*{Appendix}
{\bf A.1} {\it Proof of Proposition \ref{simpleprop}}.
 divide it into some parts. For simplicity, we write $$\overrightarrow{t^2}=\left(
    \begin{array}{ccc}
      t_1^2 \\
      t_2^2 \\
      t_3^2 \\
    \end{array}
  \right), \overrightarrow{qt}=\left(
    \begin{array}{ccc}
      t^2_1 \\
      t^2_2 \\
      t^2_3 \\
    \end{array}
  \right), \overrightarrow{4rt+q^2}=\left(
    \begin{array}{ccc}
      4r_1t_1+q^2_1 \\
      4r_2t_2+q^2_2\\
      4r_3t_3+q^2_3\\
    \end{array}
  \right).$$
 We show the following lemmas first.
  \begin{lem}\label{part1}
  Assume that $M=(m_{ij})_{3\times 3}$ satisfies the  conditions $(1),(2)$ and $(3)$ in Proposition \ref{simpleprop}. Then there is at least one zero in the set $\{s_1,s_2,s_3,t_1,t_2,t_3\}$.
  \end{lem}
  \begin{proof}
  If the elements in $\{s_1,s_2,s_3,t_1,t_2,t_3\}$ are all non-zero, then $s_1t^2_1+s_2t^2_2\neq 0$ since $s_1t^2_1+s_2t^2_2+s_3t^2_3=0$. Since
  $M\overrightarrow{s}=0$ and $M^T\overrightarrow{t}=0$,
  $M$ can be written by $$\left(
         \begin{array}{ccc}
           a_1 & a_2 &\frac{s_1a_1+s_2a_2}{-s_3} \\
           b_1 & b_2& \frac{s_1b_1+s_2b_2}{-s_3}  \\
           \frac{t_1a_1+t_2b_1}{-t_3} &\frac{t_1a_2+t_2b_2}{-t_3}  &\frac{t_1(s_1a_1+s_2a_2)+t_2(s_1b_1+s_2b_2)}{s_3t_3}  \\
         \end{array}
       \right)$$ By $M^T\overrightarrow{q}=\overrightarrow{t^2}$, we have $2=r(M^T)=r(M^T,\overrightarrow{t^2})$, which implies
       \begin{align}\label{eqone}
       t_3^2=-\frac{s_1}{s_3}t_1^2-\frac{s_2}{s_3}t_2^2.
       \end{align}
       Similarly, $M^T\overrightarrow{r}=\overrightarrow{qt}$ implies $2=r(M^T)=r(M^T,\overrightarrow{rt})$ and hence
       \begin{align}\label{eqtwo}
       q_3t_3=-\frac{s_1}{s_3}q_1t_1-\frac{s_2}{s_3}q_2t_2\Leftrightarrow q_3=-\frac{s_1t_1}{s_3t_3}q_1-\frac{s_2t_2}{s_3t_3}q_2.
       \end{align}

       Since $\overrightarrow{qt}$ and $\overrightarrow{t^2}$ are linearly independent, the vectors $\left(
    \begin{array}{ccc}
      q_1t_1 \\
      q_2t_2
    \end{array}
  \right)$ and $\left(
    \begin{array}{ccc}
      t_1^2 \\
      t_2^2
    \end{array}
  \right)$ are linearly independent.
  Indeed, if $\left(
    \begin{array}{ccc}
      q_1t_1 \\
      q_2t_2
    \end{array}
  \right)$ and $\left(
    \begin{array}{ccc}
      t_1^2 \\
      t_2^2
    \end{array}
  \right)$ are linearly dependent, then there exist $\lambda \in k$ such that $\lambda t_1^2= q_1t_1,\lambda t_2^2=q_2t_2$, which implies $q_1=\lambda t_1,q_2=\lambda t_2$. And hence \begin{align*}
  q_3&=-\frac{s_1t_1}{s_3t_3}q_1-\frac{s_2t_2}{s_3t_3}q_2\\
 &= -\frac{\lambda s_1t_1^2}{s_3t_3} -\frac{\lambda s_2t_2^2}{s_3t_3}\\
 &= \lambda \frac{s_3t_3^2}{s_3t_3}=\lambda t_3.
  \end{align*}
But then $\overrightarrow{qt}=\lambda \overrightarrow{t^2}$ and $\overrightarrow{t^2}$ are linearly dependent. We reach a contradiction. So the vectors $\left(
    \begin{array}{ccc}
      q_1t_1 \\
      q_2t_2
    \end{array}
  \right)$ and $\left(
    \begin{array}{ccc}
      t_1^2 \\
      t_2^2
    \end{array}
  \right)$ are linearly independent and hence
 $q_1t_2\neq q_2t_1$.

  On the other hand, $M^T\overrightarrow{q}=\overrightarrow{t^2}$ implies
  \begin{align}\label{one}
  a_1q_1+b_1q_2+\frac{(t_1a_1+t_2b_1)(s_1t_1q_1+s_2t_2q_2)}{t_3s_3t_3}=t_1^2
  \end{align}
  and
  \begin{align}\label{two}
  a_2q_1+b_2q_2+\frac{(t_1a_2+t_2b_2)(s_1t_1q_1+s_2t_2q_2)}{t_3s_3t_3}=t_2^2,
  \end{align}
  which are respectively equivalent to
  $$
  (a_1q_1+b_1q_2)s_3t_3^2t_2^2+(t_1a_1+t_2b_1)(s_1t_1q_1+s_2t_2q_2)t_2^2=t_1^2s_3t_3^2t_2^2 $$ and
  $$(a_2q_1+b_2q_2)s_3t_3^2t_1^2+(t_1a_2+t_2b_2)(s_1t_1q_1+s_2t_2q_2)t_1^2=t_2^2s_3t_3^2t_1^2.$$
  Then
  \begin{align*}
  & \quad\quad\quad (a_1q_1t_2^2+b_1q_2t_2^2-a_2q_1t_1^2-b_2q_2t_1^2)(s_1t_1^2+s_2t_2^2) \\
  &\quad\quad  =(s_1t_1q_1+s_2t_2q_2)(t_1t_2^2a_1+t_2^3b_1-t_1^3a_2-t_1^2t_2b_2)\\
  &\Rightarrow \quad\quad\quad s_1b_1q_2t_1^2t_2^2-s_1b_2q_2t_1^4+s_2a_1q_1t_2^4-s_2a_2q_1t_1^2t_2^2\\
  &\quad\quad  = s_1b_1q_1t_1t_2^3-s_1b_2q_1t_3^3t_2+s_2a_1q_2t_1t_2^3-s_2a_2q_2t_1^3t_2  \\
  & \Rightarrow \quad\quad\quad q_2t_1[b_1s_1t_1t_2^2-b_2s_1t_1^3-a_1s_2t_2^3+a_2s_2t_1^2t_2]\\
  &\quad\quad  =q_1t_2[b_1s_1t_1t_2^2-a_1s_2t_2^3+a_2s_2t_1^2t_2-b_2s_1t_1^3].
  \end{align*}
If $b_1s_1t_1t_2^2-a_1s_2t_2^3+a_2s_2t_1^2t_2-b_2s_1t_1^3\neq 0$, then we get $q_1t_2=q_2t_1$, which contradicts with $q_1t_2\neq q_2t_1$.

If $b_1s_1t_1t_2^2-a_1s_2t_2^3+a_2s_2t_1^2t_2-b_2s_1t_1^3=0$, then
\begin{align}\label{eqimp}
s_1t_1b_2-s_2t_2a_2=(s_1t_1b_1-s_2t_2a_1)\frac{t_2^2}{t_1^2}
\end{align}
and we can show as follows that (\ref{one}) and (\ref{two}) are equivalent. Indeed,
 (\ref{two}) is equivalent to
\begin{align*}
&\quad\quad\quad \frac{(a_2q_1+b_2q_2)(-s_1t_1^2-s_2t_2^2)+(t_1a_2+t_2b_2)(s_1t_1q_1+s_2t_2q_2)}{s_3t_3^2}=t_2^2 \\
&\quad\quad\quad \Leftrightarrow
\frac{-a_2q_1s_2t_2^2-b_2s_1q_2t_1^2+a_2q_2s_2t_1t_2+b_2q_1s_1t_1t_2}{s_3t_3^2}=t_2^2 \\
&\quad\quad\quad \Leftrightarrow \frac{(q_1t_2-q_2t_1)(s_1t_1b_2-s_2t_2a_2)}{s_3t_3^2}=t_2^2\\
&\quad\quad\quad \Leftrightarrow \frac{(q_1t_2-q_2t_1)(s_1t_1b_1-s_2t_2a_1)\frac{t_2^2}{t_1^2}}{s_3t_3^2}=t_2^2\\
&\quad\quad\quad \Leftrightarrow \frac{(q_1t_2-q_2t_1)(s_1t_1b_1-s_2t_2a_1)}{s_3t_3^2}=t_1^2.
\end{align*}
Similarly,  (\ref{one}) is equivalent to
\begin{align*}
&\quad\quad\quad \frac{(a_1q_1+b_1q_2)(-s_1t_1^2-s_2t_2^2)+(t_1a_1+t_2b_1)(s_1t_1q_1+s_2t_2q_2)}{s_3t_3^2}=t_1^2 \\
&\quad\quad\quad \Leftrightarrow
\frac{-a_1q_1s_2t_2^2-b_1s_1q_2t_1^2+a_1q_2s_2t_1t_2+b_1q_1s_1t_1t_2}{s_3t_3^2}=t_1^2 \\
&\quad\quad\quad \Leftrightarrow \frac{(q_1t_2-q_2t_1)(s_1t_1b_1-s_2t_2a_1)}{s_3t_3^2}=t_1^2.
\end{align*}
So (\ref{one}) and (\ref{two}) are equivalent. Then
$r\left(
         \begin{array}{ccc}
           a_1 & b_1 &  \frac{t_1a_1+t_2b_1}{-t_3} \\
           a_2 & b_2&  \frac{t_1a_2+t_2b_2}{-t_3}
         \end{array}
       \right)=1$ and hence $$r(M^T)=r\left(
         \begin{array}{ccc}
           a_1 & b_1 &  \frac{t_1a_1+t_2b_1}{-t_3} \\
           a_2 & b_2&  \frac{t_1a_2+t_2b_2}{-t_3} \\
           \frac{s_1a_1+s_2a_2}{-s_3}& \frac{s_1b_1+s_2b_2}{-s_3} &\frac{t_1(s_1a_1+s_2a_2)+t_2(s_1b_1+s_2b_2)}{s_3t_3}
         \end{array}
       \right)=1,
       $$
which contradicts with $r(M)=2$.

Then we reach a conclusion that there are at least one zero in $\{s_1,s_2,s_3,t_1,t_2,t_3\}$.
  \end{proof}

\begin{lem}\label{part2}
Assume that $M=(m_{ij})_{3\times 3}$ satisfies all the  conditions in Proposition \ref{simpleprop}. Then there is at least one zero in the set $\{t_1,t_2,t_3\}$. Furthermore, there is exactly one zero in the set $\{t_1,t_2,t_3\}$.
\end{lem}
\begin{proof}
We will give a proof of the first part of the statement by contradiction.
Assume that each $t_i$ is nonzero. Then there exists at least one zero in $\{s_1,s_2,s_3\}$
by Lemma $\ref{part1}$. Furthermore, there is exactly one zero $s_i$ since $s_1t^2_1+s_2t^2_2+s_3t^2_3=0$.
Without the loss of generality, we assume
 $s_1=0$. Then $s_2t^2_2+s_3t^2_3=0$, and $s_2, s_3\in k^{\times}$.
 By
 $Ms=0$, we have
 \begin{align*}
 \begin{cases}
 m_{12}s_2+m_{13}s_3=0\\
 m_{22}s_2+m_{23}s_3=0\\
 m_{32}s_2+m_{33}s_3=0
 \end{cases} \Leftrightarrow  \left( \begin{array}{ccc}
      m_{13}\\
      m_{23}\\
      m_{33}\\
    \end{array}
  \right)=-\frac{s_2}{s_3} \left( \begin{array}{ccc}
      m_{12}\\
      m_{22}\\
      m_{32}\\
    \end{array}
  \right).
 \end{align*}
 Moreover, $M^Tt=0$ and $r(M)=2$ imply that $M$ can be written by
 $$M=\left( \begin{array}{ccc}
    a_1&a_2&\frac{s_2a_2}{-s_3} \\
     b_1&b_2&\frac{s_2b_2}{-s_3}\\
    \frac{a_1t_1+b_1t_2}{-t_3}&\frac{a_2t_1+b_2t_2}{-t_3}&\frac{s_2a_2t_1+s_2b_2t_2}{s_3t_3}\\
    \end{array}
  \right)\quad \text{with}\quad r\left( \begin{array}{cc}
    a_1&a_2 \\
     b_1&b_2
      \end{array}
  \right)=2.$$
 By $M^T\overrightarrow{r}=\overrightarrow{qt}$,  we have $r(M^T)=r(M^T,\overrightarrow{qt})$, which implies
  $-\frac{s_2}{s_3}q_2t_2=q_3t_3$ or equivalently $q_2=\frac{-s_3q_3t_3}{s_2t_2}$. Substitute it into $M^T\overrightarrow{q}=\overrightarrow{t^2}$. We have
\begin{align}\label{one2}
a_1q_1-\frac{s_3t_3b_1q_3}{s_2t_2}-\frac{a_1t_1q_3+b_1t_2q_3}{t_3}=t^2_1
\end{align}
and
\begin{align}\label{two2}
 a_2q_1-\frac{s_3t_3b_2q_3}{s_2t_2}-\frac{a_2t_1q_3+b_2t_2q_3}{t_3}=t^2_2.
 \end{align}
By computations, (\ref{one2}) and (\ref{two2}) are respectively equivalent to
\begin{align*}
&\quad\quad\quad a_1q_1s_2t_2t_3-b_1q_3s_3t_3^2-s_2t_2(a_1t_1q_3+b_1t_2q_3)=s_2t_1^2t_2t_3\\
&\Leftrightarrow a_1q_1s_2t_2t_3-b_1q_3s_3t_3^2-a_1q_3s_2t_1t_2-b_1q_3s_2t_2^2=s_2t_1^2t_2t_3\\
&\Leftrightarrow a_1q_1t_3-a_1q_3t_1=t_1^2t_3 \Leftrightarrow a_1(q_1t_3-q_3t_1)=t_1^2t_3
\end{align*}
and
\begin{align*}
&\quad\quad\quad a_2q_1s_2t_2t_3-s_3t_3^2b_2q_3-(a_2t_1q_3+b_2t_2q_3)s_2t_2=s_2t^3_2t_3\\
&\Leftrightarrow  a_2q_1s_2t_2t_3-b_2q_3s_3t_3^2-a_2q_3s_2t_1t_2-b_2q_3s_2t_2^2=s_2t_2^3t_3\\
&\Leftrightarrow  a_2q_1t_3-a_2q_3t_1=t_2^2t_3 \Leftrightarrow a_2(q_1t_3-q_3t_1)=t_2^2t_3.
\end{align*}
Since each $t_i$ is non-zero, we have  $a_1,a_2$ and $q_1t_3-q_3t_1$ are all non-zeros. Hence
 $a_2=\frac{t^2_2}{t^2_1}a_1$.
  Then
 $$M^T=\left(
         \begin{array}{ccc}
           a_1 & b_1 & \frac{a_1t_1+b_1t_2}{-t_3} \\
           \frac{t^2_2}{t^2_1}a_1 &b_2 & \frac{a_1t^2_2+b_2t_1t_2}{-t_1t_3}  \\
           \frac{t^2_3}{t^2_1}a_1 &-\frac{s_2}{s_3}b_2 & \frac{s_2b_2t_2}{s_3t_3}-\frac{t_3}{t_1}a_1 \\
         \end{array}
       \right).$$
 Since $s_2t_2^2=-s_3t_3^2$ and $s_2q_2t_2=-s_3q_3t_3$, we have $$s_2q_2^2s_2t_2^2=s_2^2q_2^2t_2^2=s_3^2q_3^2t_3^2=s_3q_3^2s_3t_3^2,$$
 which implies $s_2q_2^2=-s_3q_3^2$.
 On the other hand, $r(M^T)=r(M^T,4rt+q^2)$
 since $M^T\overrightarrow{u}=\overrightarrow{4rt+q^2}$. Then
 we have
 \begin{align*}
-\frac{s_2}{s_3}(4r_2t_2+q^2_2)=4r_3t_3+q^2_3 &\Leftrightarrow -4s_2r_2t_2-s_2q_2^2=4s_3r_3t_3+s_3q_3^2\\
& \Leftrightarrow -s_2r_2t_2= s_3r_3t_3\\
& \Leftrightarrow r_2=\frac{-s_3r_3t_3}{s_2t_2}.
 \end{align*}
 Substitute it into $M^Tr=qt$.  We get
\begin{align}\label{one3}
a_1r_1-b_1\frac{s_3r_3t_3}{s_2t_2}-\frac{a_1t_1+b_1t_2}{t_3}r_3=q_1t_1
 \end{align}
 and
 \begin{align}\label{two3}
  \frac{t^2_2}{t^2_1}a_1r_1-b_2\frac{s_3r_3t_3}{s_2t_2}-\frac{a_1t^2_2+b_2t_1t_2}{t_1t_3}r_3=q_2t_2.
  \end{align}
 By computations, (\ref{one3}) and (\ref{two3}) are respectively equivalent to
  \begin{align*}
 &\quad\quad a_1s_2r_1t_2t_3-b_1s_3r_3t_3^2-(a_1t_1+b_1t_2)s_2t_2r_3=q_1s_2t_1t_2t_3\\
  &\Leftrightarrow  a_1s_2r_1t_2t_3-b_1r_3(s_3t_3^2+s_2t_2^2)-a_1r_3s_2t_1t_2=q_1s_2t_1t_2t_3\\
  &\Leftrightarrow  a_1r_1t_3-a_1r_3t_1=q_1t_1t_3 \Leftrightarrow q_1t_2=\frac{a_1r_1t_2t_3-a_1r_3t_1t_2}{t_1t_3}
  \end{align*}
  and
   \begin{align*}
  &\quad\quad  a_1r_1s_2t^3_2t_3-b_2s_3r_3t_1^2t_3^2-(a_1t^2_2+b_2t_1t_2)s_2t_1t_2r_3=q_2s_2t_1^2t_2^2t_3\\
  &\Leftrightarrow a_1r_1s_2t^3_2t_3-b_2r_3t_1^2(s_3t_3^2+s_2t_2^2)-a_1s_2r_3t_1t_2^3=q_2s_2t_1^2t_2^2t_3\\
  &\Leftrightarrow a_1r_1t_2t_3-a_1r_3t_1t_2=q_2t_1^2t_3 \Leftrightarrow q_2t_1=\frac{a_1r_1t_2t_3-a_1r_3t_1t_2}{t_1t_3}.
  \end{align*}
Then $q_1t_2=q_2t_1,$ which contradicts with the assumption that $\overrightarrow{qt}$ and $\overrightarrow{t^2}$ are linearly independent (
One can see why  $q_1t_2\neq q_2t_1$ in the proof of Lemma \ref{part1}).

By the proof above, we can conclude that there is at least one zero in $\{t_1,t_2,t_3\}$. If there are two zeros in  $\{t_1,t_2,t_3\}$, then  $\overrightarrow{qt}$ and $\overrightarrow{t^2}$ are obviously linearly dependent, which contradicts with the condition (2) in Proposition \ref{simpleprop}.
\end{proof}

\begin{lem}\label{part3}
Assume that $M=(m_{ij})_{3\times 3}$ satisfies all the  conditions in Proposition \ref{simpleprop}. Then any two non-zero columns of $M$ are linearly independent.
\end{lem}
\begin{proof}
We will give a proof by contradiction. Suppose that $M$ admits two nonzero linearly dependent columns. Without
the loss of generality, $M$ can be written as
$$\left(
         \begin{array}{ccc}
           a & ua & d \\
           b & ub & e  \\
           c & uc & f \\
         \end{array}
       \right), u\in k^{\times}.$$
Since $Ms=0$ and $r(M)=2$, we have $s_1=-us_2$, $s_3=0$. Then
 $s_1t^2_1+s_2t^2_2=0$, and $s_1,s_2\in k^{\times}$ since $u\neq 0$.
 Since $M^T\overrightarrow{q}=\overrightarrow{t^2}$ and $\overrightarrow{q}\neq 0$,
 we have $r(M^T)=r(M^T,\overrightarrow{t^2})$, which implies
$t^2_2=ut^2_1$. Similarly, $M^T\overrightarrow{r}=\overrightarrow{qt}$ implies
$r(M^T)=r(M^T,\overrightarrow{qt})$, and hence
$q_2t_2=uq_1t_1$. We claim $t_3\neq 0$. Indeed, if $t_3=0$,then $\overrightarrow{qt}=\left(
         \begin{array}{ccc}
           q_1t_1 \\
          uq_1t_1  \\
           0 \\
         \end{array}
       \right)$ and $\overrightarrow{t^2}=\left(
         \begin{array}{ccc}
           t^2_1 \\
           ut^2_1  \\
           0\\
         \end{array}
       \right)$ are linearly dependent, which contradicts with the assumption.
Since $t^2_2=ut^2_1$, we have $t_1=t_2=0$ or $t_1\neq 0,t_2\neq 0$. However, both cases contradicts with the statement of Lemma \ref{part2}.
Then we complete our proof.
\end{proof}

\begin{lem}\label{part4}
Assume that $M=(m_{ij})_{3\times 3}$ satisfies all the  conditions in Proposition \ref{simpleprop}. Then there is at lest one zero in $\{s_1,s_2,s_3\}$. Furthermore, there are exactly two zeros in the set $\{s_1,s_2,s_3\}$.
\end{lem}
\begin{proof}
Assume that each $s_i$ is non-zero. It suffices to reach a contradiction. By Lemma \ref{part2}, there is exactly one zero in $\{t_1,t_2,t_3\}$. Without the loss of generality, we let $t_1=0$ and $t_2,t_3\in k^{\times}$. Then $M$ can be written as
$$M=\left( \begin{array}{ccc}
    a_1&a_2&\frac{s_1a_1+s_2a_2}{-s_3} \\
     b_1&b_2&\frac{s_1b_1+s_2b_2}{-s_3}\\
    \frac{b_1t_2}{-t_3}&\frac{b_2t_2}{-t_3}&\frac{t_2s_1b_1+t_2s_2b_2}{s_3t_3}\\
    \end{array}
  \right).$$
  Since $M^T\overrightarrow{q}=\overrightarrow{t^2}$ and $M^T\overrightarrow{r}=\overrightarrow{qt}$, we have $r(M^T,\overrightarrow{t^2})=r(M^T)$ and
$r(M^T,\overrightarrow{qt})=r(M^T)$ respectively. Then we get $t_3^2=-\frac{s_2}{s_3}t_2^2$ and $q_3t_3=-\frac{s_2}{s_3}q_2t_2$. This implies that
$\overrightarrow{qt}= \left( \begin{array}{ccc}
      0\\
      q_2t_2\\
      -\frac{s_2}{s_3}q_2t_2\\
    \end{array}
  \right)$ and $\overrightarrow{t^2}= \left( \begin{array}{ccc}
      0\\
      t_2^2\\
      -\frac{s_2}{s_3}t_2^2\\
    \end{array}
  \right)$ are linearly dependent.  This contradicts with the condition (2) in Proposition \ref{simpleprop}. So there is at lest one zero in $\{s_1,s_2,s_3\}$.
Since $r(M)=2$ and $M\overrightarrow{s}=0$, we can conclude that there are exactly two zeros in $\{s_1,s_2,s_3\}$.
\end{proof}

Now, let us come to the proof of Proposition \ref{simpleprop}.

\begin{proof}
By Lemma \ref{part3} and Lemma \ref{part4}, one sees that $M$ admits one zero column.
More precisely, $s_i\neq 0$ if and only if the $i$-th column of $M$ is zero. Without the loss of generality,  we may let $s_2\neq 0$. Then we can write
$$M=\left(
         \begin{array}{ccc}
          a&0&d \\
          b&0&e  \\
         c&0&f\\
         \end{array}
       \right).$$
We have $\overrightarrow{s}=\left(
         \begin{array}{ccc}
          0 \\
          s_2  \\
        0\\
         \end{array}
       \right)$, $\overrightarrow{t}=\left(
         \begin{array}{ccc}
          t_1 \\
          0  \\
        t_3\\
         \end{array}
       \right)$, with $t_1,t_3\in k^{\times}$.
       Since $M^T\overrightarrow{t}=0$, we have
\begin{align}\label{lineareq}
\begin{cases}
at_1+ct_3=0\\
dt_1+ft_3=0,
\end{cases}
\end{align}
which implies $$\left|
         \begin{array}{ccc}
           a& c \\
           d & f \\
         \end{array}
       \right|=0 \Leftrightarrow af=cd.$$
There is at least one non-zero element in $\{a, f,c,d\}$. Otherwise, $r(M)=1$.
On the other hand, $t_1$ and $t_3$ are both non-zeros.  Hence (\ref{lineareq}) implies that
$$\begin{cases}
a=0,c=0\\
d\neq 0, f\neq 0
\end{cases}\,\, \text{or}\,\,\begin{cases}
a\neq 0,c\neq 0\\
d=0, f=0
\end{cases}\,\, \text{or}\,\,\begin{cases}
a\neq 0,c\neq 0\\
d\neq 0, f\neq 0.
\end{cases} $$

If $\begin{cases}
a=0,c=0\\
d\neq 0, f\neq 0,
\end{cases}$ then $M=\left(
         \begin{array}{ccc}
          0&0&d \\
          b&0&e  \\
         0&0&f\\
         \end{array}
       \right).$ We have $b\neq 0$ since $r(M)=2$. We can take $\overrightarrow{t}=\left(
         \begin{array}{ccc}
          f \\
          0  \\
          -d\\
         \end{array}
       \right), \overrightarrow{q}=
\left(
         \begin{array}{ccc}
          d-\frac{f^2e}{db} \\
          \frac{f^2}{b}  \\
          0\\
         \end{array}
       \right)$. Then $\overrightarrow{qt}=\left(
         \begin{array}{ccc}
          fd-\frac{f^3e}{db} \\
          0  \\
          0\\
         \end{array}
       \right)$. Since $\overrightarrow{qt}$ and $\overrightarrow{t^2}$ are linearly independent, we have $d^2b\neq f^2e$.
By $M^T\overrightarrow{r}=\overrightarrow{qt}$, we can take $\overrightarrow{r}=\left(
         \begin{array}{ccc}
          \frac{e^2f^3-d^2bef}{d^2b^2} \\
          \frac{fd^2b-f^3e}{db^2}  \\
          0\\
         \end{array}
       \right).$ Then $\overrightarrow{4rt+q^2}=\left(
         \begin{array}{ccc}
          \frac{(5f^2e-d^2b)(f^2e-d^2b)}{d^2b^2} \\
          \frac{f^4}{b^2}  \\
          0\\
         \end{array}
       \right)$ and hence $3=r(M^T,\overrightarrow{4rt+q^2})\neq r(M^T)=2$, which contradicts with $M^T\overrightarrow{u}=\overrightarrow{4rt+q^2}$. So this case is impossible to occur.

       If $\begin{cases}
a\neq 0,c\neq 0\\
d=0, f=0
\end{cases}$, then $M=\left(
         \begin{array}{ccc}
          a&0&0 \\
          b&0&e  \\
         c&0&0\\
         \end{array}
       \right).$ We have $e\neq 0$ since $r(M)=2$. We can take $\overrightarrow{t}=\left(
         \begin{array}{ccc}
          c \\
          0  \\
          -a\\
         \end{array}
       \right), \overrightarrow{q}=
\left(
         \begin{array}{ccc}
          \frac{c^2}{a}-\frac{ba}{e} \\
          \frac{a^2}{e}  \\
          0\\
         \end{array}
       \right)$. Then $\overrightarrow{qt}=\left(
         \begin{array}{ccc}
          \frac{c^3}{a} -\frac{abc}{e}\\
          0  \\
          0\\
         \end{array}
       \right)$. Since $\overrightarrow{qt}$ and $\overrightarrow{t^2}$ are linearly independent, we have $a^2b\neq c^2e$.
By $M^T\overrightarrow{r}=\overrightarrow{qt}$, we can take $\overrightarrow{r}=\left(
         \begin{array}{ccc}
           \frac{c^3}{a^2}-\frac{bc}{e}\\
          0 \\
          0\\
         \end{array}
       \right).$ Then $\overrightarrow{4rt+q^2}=\left(
         \begin{array}{ccc}
          \frac{(5c^2e-a^2b)(c^2e-a^2b)}{a^2e^2} \\
          \frac{a^4}{e^2}  \\
          0\\
         \end{array}
       \right)$ and hence $3=r(M^T,\overrightarrow{4rt+q^2})\neq r(M^T)=2$, which contradicts with $M^T\overrightarrow{u}=\overrightarrow{4rt+q^2}$. So this case is impossible to occur.

Now, lets consider the cases  $\begin{cases}
a\neq 0,c\neq 0\\
d\neq 0, f\neq 0.
\end{cases} $
Since $af=cd$, there exists $\lambda\in k^{\times}$ such that $d=\lambda a, f=\lambda c$.
We have $M=\left(
         \begin{array}{ccc}
          a&0&\lambda a \\
          b&0&e  \\
          c&0& \lambda c\\
         \end{array}
       \right)$ with $e\neq \lambda b$ since $r(M)=2$. By computations, we can choose
       $$\overrightarrow{t}=\left(
         \begin{array}{ccc}
          c \\
          0  \\
          -a\\
         \end{array}
       \right), \overrightarrow{q}=\left(
         \begin{array}{ccc}
           \frac{c^2e-a^2b}{ea-\lambda ab} \\
           \frac{a^2-\lambda c^2}{e-\lambda b}  \\
          0 \\
         \end{array}
       \right).$$
Then $\overrightarrow{qt}=\left(
         \begin{array}{ccc}
          \frac{c^3e-a^2bc}{ea-\lambda ab} \\
          0  \\
          0\\
         \end{array}
       \right)$.
Since $\overrightarrow{qt}$ and $\overrightarrow{t^2}$ are linearly independent, we have $a^2b\neq c^2e$. By $M^T\overrightarrow{r}=\overrightarrow{qt}$, we can choose$\overrightarrow{r}=\left(
         \begin{array}{ccc}
          \frac{c^3e^2-a^2bce}{a^2(e-\lambda b)^2} \\
          \frac{\lambda (a^2bc-c^3e)}{a(e-\lambda b)^2}  \\
          0\\
         \end{array}
       \right).$ Then $$\overrightarrow{4rt+q^2}=\left(
         \begin{array}{ccc}
          \frac{(5c^2e-a^2b)(c^2e-a^2b)}{a^2(e-\lambda b)^2} \\
          \frac{(a^2-\lambda c^2)^2}{(e-\lambda b)^2}  \\
          0\\
         \end{array}
       \right).$$
Since $M^T\overrightarrow{u}=\overrightarrow{4rt+q^2}$, we have $r(M^T,\overrightarrow{4rt+q^2})=r(M^T)=2$, which implies $a^2=\lambda c^2$.
Note that one get $a^2b\neq c^2e$ when $\lambda b \neq e$ and $a^2=\lambda c^2$. Therefore,
$M=\left(
         \begin{array}{ccc}
          a&0&\lambda a \\
          b&0&e  \\
          c&0&\lambda c\\
         \end{array}
       \right)$ with $a,c, \lambda \in k^{\times}, e\neq \lambda b$ and $a^2=\lambda c^2$.

       Conversely, if $M=\left(
         \begin{array}{ccc}
          a&0&\lambda a \\
          b&0&e  \\
          c&0&\lambda c\\
         \end{array}
       \right)$ with $a,c, \lambda \in k^{\times}, e\neq \lambda b$ and $a^2=\lambda c^2$, then it is straight forward to show that $M$ satisfies the conditions (1),(2),(3),(4) in Proposition \ref{simpleprop} and $s_2\neq 0$.

Similarly,  $M=\left(
         \begin{array}{ccc}
          0&b&e \\
          0&a&\lambda a \\
          0&c&\lambda c\\
         \end{array}
       \right)$ (resp. $M=\left(
         \begin{array}{ccc}
          a&\lambda a &0 \\
          c&\lambda c &0 \\
          b&e         &0\\
         \end{array}
       \right))$ with $a,c, \lambda \in k^{\times}, e\neq \lambda b$ and $a^2=\lambda c^2$ if and only if $M$ satisfies the conditions (1),(2),(3),(4) in Proposition \ref{simpleprop} and $s_1\neq 0$ (resp. $s_3\neq 0$).
\end{proof}

{\bf A.2} {\it Proof of the claims on the minimal semifree resolutions in Section \ref{caseone}}. Since the proofs for different subcases are similar to each other, we only need to prove the most complicated one: Case $1.2.4$. In this case,
\begin{align*}
\begin{cases}
t_1=t_2=0,t_3\neq 0\\
\exists \sigma =q_2x_2, q_2\neq 0, \partial_{\mathcal{A}}(\sigma)=t_3^2x_3^2\\
\partial_{\mathcal{A}}(x_3)=0\\
\exists \lambda=u_1x_1+u_2x_2, u_1\neq 0\\
\partial_{\mathcal{A}}(\lambda)=q_2^2x_2^2\\
\exists \eta=w_1x_1+w_2x_2, w_1=\frac{2u_2u_1}{q_2},w_2=\frac{2u_2^2}{q_2},
\partial_{\mathcal{A}}(\eta)=2q_2u_2x_2^2 \\
\end{cases}
\end{align*}
by the constructing process in Section \ref{caseone}. We have
\begin{align*}
H(\mathcal{A})&=k[\lceil t_1x_1 +t_2x_2+t_3x_3\rceil,\lceil s_1x_1^2+s_2x_2^2+s_3x_3^2\rceil ]/(\lceil t_1x_1 +t_2x_2+t_3x_3\rceil^2)\\
              &=k[\lceil x_3\rceil, \lceil s_1x_1^2\rceil]/(\lceil x_3^2))
\end{align*}

and
$F=F_7$ with
$$F^{\#}=\mathcal{A}^{\#}\oplus \mathcal{A}^{\#}e_1\oplus \mathcal{A}^{\#}e_2\oplus \mathcal{A}^{\#}e_3\oplus \mathcal{A}^{\#}e_4\oplus \mathcal{A}^{\#}e_5\oplus \mathcal{A}^{\#}e_6\oplus \mathcal{A}^{\#}e_7$$ and
$$\left(
                         \begin{array}{c}
                          \partial_{F}(1) \\
                          \partial_{F}(e_1)\\
                          \partial_{F}(e_2)\\
                          \partial_{F}(e_3)\\
                          \partial_{F}(e_4)\\
                          \partial_{F}(e_5)\\
                          \partial_{F}(e_6)\\
                          \partial_{F}(e_7)
                         \end{array}
                       \right) =\left(
            \begin{array}{cccccccc}
              0 & 0& 0 & 0 &0 & 0 &0 & 0\\
               t_3x_3  & 0 & 0 & 0 &0 &0&0 & 0 \\
              \sigma &  t_3x_3  & 0 & 0 &0 &0&0 & 0 \\
              0 & \sigma & t_3x_3   & 0 &0 &0&0 & 0\\
              \lambda & 0 & \sigma &  t_3x_3  & 0 & 0&0 & 0\\
              0   & \lambda & 0       & \sigma &  t_3x_3  & 0      &0 & 0\\
             \eta & 0       & \lambda & 0      & \sigma   & t_3x_3 &0 & 0\\
             0    & \eta    & 0       &\lambda &  0       &\sigma  &t_3x_3 & 0
            \end{array}
          \right)\left(
                         \begin{array}{c}
                          1\\
                          e_1\\
                          e_2\\
                          e_3\\
                          e_4\\
                          e_5\\
                          e_6\\
                          e_7
                         \end{array}
                       \right).$$
          We already have $H^1(F)=0$.
In order to show that $F$ is a minimal semi-free resolution of ${}_{\mathcal{A}}k$,
we should prove $H^{2i}(F)=0$ and $H^{2i+1}(F)=0$ for any $i\ge 1$.
Let $z=a_0+\sum\limits_{j=1}^7a_je_j\in Z^{2i}$. We have
\begin{align*}
0&=\partial_F(z)=\partial_{\mathcal{A}}(a_0)+\sum\limits_{j=1}^7[\partial_{\mathcal{A}}(a_j)e_j+a_j\partial_F(e_j)]\\
&=\partial_{\mathcal{A}}(a_7)e_7+[a_7t_3x_3+\partial_{\mathcal{A}}(a_6)]e_6+[\partial_{\mathcal{A}}(a_5)+a_6t_3x_3+a_7\sigma]e_5\\
&+[\partial_{\mathcal{A}}(a_4)+a_5t_3x_3+a_6\sigma]e_4+[\partial_{\mathcal{A}}(a_3)+a_4t_3x_3+a_5\sigma+a_7\lambda]e_3\\
&+[\partial_{\mathcal{A}}(a_2)+a_3t_3x_3+a_4\sigma +a_6\lambda]e_2+[\partial_{A}(a_1)+a_2t_3x_3+a_3\sigma+a_5\lambda+a_7\eta]e_1\\
&+\partial_{\mathcal{A}}(a_0)+a_1t_3x_3+a_2\sigma+a_4\lambda+a_5\eta +a_6\eta.
\end{align*}
Hence
\begin{align*}
\begin{cases}
\partial_{\mathcal{A}}(a_7)=0\\
a_7t_3x_3+\partial_{\mathcal{A}}(a_6)=0\\
\partial_{\mathcal{A}}(a_5)+a_6t_3x_3+a_7\sigma=0\\
\partial_{\mathcal{A}}(a_4)+a_5t_3x_3+a_6\sigma=0\\
\partial_{\mathcal{A}}(a_3)+a_4t_3x_3+a_5\sigma+a_7\lambda=0\\
\partial_{\mathcal{A}}(a_2)+a_3t_3x_3+a_4\sigma +a_6\lambda=0\\
\partial_{A}(a_1)+a_2t_3x_3+a_3\sigma+a_5\lambda+a_7\eta=0\\
\partial_{\mathcal{A}}(a_0)+a_1t_3x_3+a_2\sigma+a_4\lambda +a_6\eta=0.
\end{cases}
\end{align*}
By $\partial_{\mathcal{A}}(a_7)=0$, we have $a_7=c_7(x_1^2)^i+\partial_{\mathcal{A}}(\lambda_7)$, for some $c_7\in k,\lambda_7\in \mathcal{A}^{2i-1}$.
Since $a_7t_3x_3+\partial_{\mathcal{A}}(a_6)=0$, we conclude that $c_7=0$ and $a_6=-t_3\lambda_7 x_3+c_6(x_1^2)^i+\partial_{\mathcal{A}}(\lambda_6)$ for some $c_6\in k, \lambda_6\in \mathcal{A}^{2i-1}$. Then $\partial_{\mathcal{A}}(a_5)+a_6t_3x_3+a_7\sigma=0$ implies that $c_6=0$ and $a_5=-\lambda_7\sigma-\lambda_6t_3x_3 +c_5(x_1^2)^i+\partial_{\mathcal{A}}(\lambda_5)$ for some $c_5\in k, \lambda_5\in \mathcal{A}^{2i-1}$. By
$\partial_{\mathcal{A}}(a_4)+a_5t_3x_3+a_6\sigma=0$, we have
\begin{align*}
\partial_{\mathcal{A}}(a_4)&=-[-\lambda_7\sigma-\lambda_6t_3x_3+c_5(x_1^2)^i+\partial_{\mathcal{A}}(\lambda_5)]t_3x_3-[-t_3\lambda_7 x_3+\partial_{\mathcal{A}}(\lambda_6)]\sigma \\
&=t_3\lambda_7(\sigma x_3+x_3\sigma)+\lambda_6t_3^2x_3^2-\partial_{\mathcal{A}}(\lambda_6)\sigma-\partial_{\mathcal{A}}(\lambda_5)t_3x_3-c_5(x_1^2)^it_3x_3\\
&=q_2t_3\lambda_7(x_2x_3+x_3x_2)-\partial_{\mathcal{A}}(\lambda_6\sigma)-\partial_{\mathcal{A}}(\lambda_5t_3x_3)-c_5(x_1^2)^it_3x_3\\
&=-\partial_{\mathcal{A}}(\lambda_6\sigma)-\partial_{\mathcal{A}}(\lambda_5t_3x_3)-c_5(x_1^2)^it_3x_3,
\end{align*}
which implies $c_5=0$ and $a_4=-\lambda_6\sigma-\lambda_5t_3x_3 +c_4(x_1^2)^i+\partial_{\mathcal{A}}(\lambda_4)$ for some $c_4\in k$ and $\lambda_4\in \mathcal{A}^{2i-1}$. By $\partial_{\mathcal{A}}(a_3)+a_4t_3x_3+a_5\sigma+a_7\lambda=0$, we have
\begin{align*}
&\partial_{\mathcal{A}}(a_3)=-a_4t_3x_3-a_5\sigma-a_7\lambda \\
&=[\lambda_6\sigma+\lambda_5t_3x_3-c_4(x_1^2)^i-\partial_{\mathcal{A}}(\lambda_4)]t_3x_3+[\lambda_7\sigma+\lambda_6t_3x_3 -\partial_{\mathcal{A}}(\lambda_5)]\sigma-\partial_{\mathcal{A}}(\lambda_7)\lambda \\
&=-\partial_{\mathcal{A}}(\lambda_5\sigma+\lambda_4t_3x_3+\lambda_7\lambda)+t_3\lambda_6(\sigma x_3+x_3\sigma)-c_4(x_1^2)^it_3x_3\\
&=-\partial_{\mathcal{A}}(\lambda_5\sigma+\lambda_4t_3x_3+\lambda_7\lambda)+t_3q_2\lambda_6(x_2x_3+x_3x_2)-c_4(x_1^2)^it_3x_3\\
&=-\partial_{\mathcal{A}}(\lambda_5\sigma+\lambda_4t_3x_3+\lambda_7\lambda)-c_4(x_1^2)^it_3x_3,
\end{align*}
which implies that $c_4=0$ and $a_3=-(\lambda_5\sigma+\lambda_4t_3x_3+\lambda_7\lambda)+c_3(x_1^2)^i+\partial_{\mathcal{A}}(\lambda_3)$ for some $c_5\in k, \lambda_5\in \mathcal{A}^{2i-1}$. Then
\begin{align*}
& \partial_{\mathcal{A}}(a_2)=-(a_3t_3x_3+a_4\sigma +a_6\lambda)\\
&=[(\lambda_5\sigma+\lambda_4t_3x_3+\lambda_7\lambda)-c_3(x_1^2)^i-\partial_{\mathcal{A}}(\lambda_3)]t_3x_3+[\lambda_6\sigma+\lambda_5t_3x_3 -\partial_{\mathcal{A}}(\lambda_4)]\sigma\\
&\quad +[t_3\lambda_7 x_3-\partial_{\mathcal{A}}(\lambda_6)]\lambda \\
&=-\partial_{\mathcal{A}}(\lambda_6\lambda +\lambda_4\sigma +\lambda_3t_3x_3)+t_3\lambda_5(\sigma x_3+x_3\sigma)+t_3\lambda_7(\lambda x_3+x_3\lambda)-c_3(x_1^2)^it_3x_3\\
&=-\partial_{\mathcal{A}}(\lambda_6\lambda +\lambda_4\sigma +\lambda_3t_3x_3)-c_3(x_1^2)^it_3x_3,
\end{align*}
which implies that $c_3=0$ and $a_2=-(\lambda_6\lambda +\lambda_4\sigma +\lambda_3t_3x_3)+c_2(x_1^2)^i+\partial_{\mathcal{A}}(\lambda_2)$.
We have
\begin{align*}
\partial_{A}(a_1)&=-a_2t_3x_3-a_3\sigma-a_5\lambda-a_7\eta \\
                 &=[(\lambda_6\lambda +\lambda_4\sigma +\lambda_3t_3x_3)-c_2(x_1^2)^i-\partial_{\mathcal{A}}(\lambda_2)]t_3x_3 \\ &+[(\lambda_5\sigma+\lambda_4t_3x_3+\lambda_7\lambda)-\partial_{\mathcal{A}}(\lambda_3)]\sigma\\
                 &+[\lambda_7\sigma+\lambda_6t_3x_3 -\partial_{\mathcal{A}}(\lambda_5)]\lambda-\partial_{\mathcal{A}}(\lambda_7)\eta\\
                 &=-\partial_{\mathcal{A}}(\lambda_5\lambda +\lambda_3\sigma+\lambda_2t_3x_3)+t_3\lambda_6(\lambda x_3+x_3\lambda)+t_3\lambda_4(\sigma x_3+x_3\sigma)\\
                 &+\lambda_5\sigma^2+\lambda_7(\lambda\sigma +\sigma \lambda)-\partial_{\mathcal{A}}(\lambda_7)\eta -c_2(x_1^2)^it_3x_3\\
                 &=-\partial_{\mathcal{A}}(\lambda_5\lambda +\lambda_3\sigma+\lambda_2t_3x_3)+\lambda_7[(u_1x_1+u_2x_2)q_2x_2+q_2x_2(u_1x_1+u_2x_2)]\\
                 &-\partial_{\mathcal{A}}(\lambda_7)\eta -c_2(x_1^2)^it_3x_3\\
                 &=-\partial_{\mathcal{A}}(\lambda_5\lambda +\lambda_3\sigma+\lambda_2t_3x_3+\lambda_7\eta)-c_2(x_1^2)^it_3x_3
\end{align*}
which implies that $c_2=0$ and $a_1=-(\lambda_5\lambda +\lambda_3\sigma+\lambda_2t_3x_3+\lambda_7\eta)+c_1(x_1^2)^i+\partial_{\mathcal{A}}(\lambda_1)$.
Then \begin{align*}
\partial_{\mathcal{A}}(a_0)&=-a_1t_3x_3-a_2\sigma-a_4\lambda -a_6\eta\\
&=(\lambda_5\lambda +\lambda_3\sigma+\lambda_2t_3x_3+\lambda_7\eta)t_3x_3 -c_1(x_1^2)^it_3x_3 -\partial_{\mathcal{A}}(\lambda_1)t_3x_3\\
&+[\lambda_6\lambda +\lambda_4\sigma +\lambda_3t_3x_3-\partial_{\mathcal{A}}(\lambda_2)]\sigma +[\lambda_6\sigma+\lambda_5t_3x_3 -\partial_{\mathcal{A}}(\lambda_4)]\lambda\\
&+[t_3\lambda_7 x_3-\partial_{\mathcal{A}}(\lambda_6)]\eta \\
&=-\partial_{\mathcal{A}}[\lambda_2\sigma+\lambda_4\lambda+\lambda_1t_3x_3]-\partial_{\mathcal{A}}(\lambda_6)\eta +\lambda_6(\lambda\sigma +\sigma\lambda)+t_3\lambda_5(\lambda x_3+x_3\lambda)\\
&+t_3\lambda_3(\sigma x_3+x_3\sigma)+t_3\lambda_7(\eta x_3+x_3\eta)-c_1(x_1^2)^it_3x_3\\
&=-\partial_{\mathcal{A}}[\lambda_2\sigma+\lambda_4\lambda+\lambda_1t_3x_3+\lambda_6\eta]-c_1(x_1^2)^it_3x_3,
\end{align*}
which implies that $c_1=0$ and $a_0=-(\lambda_2\sigma+\lambda_4\lambda+\lambda_1t_3x_3+\lambda_6\eta)+c_0(x_1^2)^i+\partial_{\mathcal{A}}(\lambda_0)$ for some $c_0\in k$ and $\lambda_0\in \mathcal{A}^{2i-1}$. Therefore,
\begin{align*}
z &=a_0+\sum\limits_{i=1}^7a_ie_i =-(\lambda_2\sigma+\lambda_4\lambda+\lambda_1t_3x_3+\lambda_6\eta)+c_0(x_1^2)^i+\partial_{\mathcal{A}}(\lambda_0)\\
                               &-(\lambda_5\lambda +\lambda_3\sigma+\lambda_2t_3x_3+\lambda_7\eta)e_1+\partial_{\mathcal{A}}(\lambda_1)e_1 -(\lambda_6\lambda +\lambda_4\sigma +\lambda_3t_3x_3)e_2\\
                               &+\partial_{\mathcal{A}}(\lambda_2)e_2 -(\lambda_7\sigma+\lambda_4t_3x_3+\lambda_7\lambda)e_3+\partial_{\mathcal{A}}(\lambda_3)e_3 -(\lambda_6\sigma+\lambda_5t_3x_3)e_4 +\partial_{\mathcal{A}}(\lambda_4)e_4\\
                               &-(\lambda_7\sigma+\lambda_6t_3x_3)e_5 +\partial_{\mathcal{A}}(\lambda_5)e_5-t_3\lambda_7 x_3e_6+\partial_{\mathcal{A}}(\lambda_6)e_6+\partial_{\mathcal{A}}(\lambda_7)e_7\\
                               &=\partial_F[\lambda_0+\sum\limits_{i=1}^7\lambda_ie_i]+c_0(x_1^2)^i\\
                               &=\partial_F[\lambda_0+\sum\limits_{i=1}^7\lambda_ie_i]-\partial_F[\frac{c_0(x_1^2)^{i-1}}{u_1^2}(t_3x_3e_7+q_2x_2e_6+\lambda e_4+\eta e_2+\frac{5u_2^2\lambda}{q_2^2})]\in B^{2i}(F).
\end{align*}
Thus $H^{2i}(F)=0$ for any $i\ge 1$.

Let $z=a_0+\sum\limits_{j=1}^7a_je_j\in Z^{2i+1}$, $i\ge 1$. Then
\begin{align*}
0&=\partial_F(z)=\partial_{\mathcal{A}}(a_0)+\sum\limits_{j=1}^7[\partial_{\mathcal{A}}(a_j)e_j-a_j\partial_F(e_j)]\\
&=\partial_{\mathcal{A}}(a_7)e_7+[\partial_{\mathcal{A}}(a_6)-a_7t_3x_3]e_6+[\partial_{\mathcal{A}}(a_5)-a_6t_3x_3-a_7\sigma]e_5\\
&+[\partial_{\mathcal{A}}(a_4)-a_5t_3x_3-a_6\sigma]e_4+[\partial_{\mathcal{A}}(a_3)-a_4t_3x_3-a_5\sigma-a_7\lambda]e_3\\
&+[\partial_{\mathcal{A}}(a_2)-a_3t_3x_3-a_4\sigma -a_6\lambda]e_2+[\partial_{A}(a_1)-a_2t_3x_3-a_3\sigma-a_5\lambda-a_7\eta]e_1\\
&+\partial_{\mathcal{A}}(a_0)-a_1t_3x_3-a_2\sigma-a_4\lambda-a_5\eta -a_6\eta.
\end{align*}
Hence
\begin{align*}
\begin{cases}
\partial_{\mathcal{A}}(a_7)=0\\
\partial_{\mathcal{A}}(a_6)-a_7t_3x_3=0\\
\partial_{\mathcal{A}}(a_5)-a_6t_3x_3-a_7\sigma=0\\
\partial_{\mathcal{A}}(a_4)-a_5t_3x_3-a_6\sigma=0\\
\partial_{\mathcal{A}}(a_3)-a_4t_3x_3-a_5\sigma-a_7\lambda=0\\
\partial_{\mathcal{A}}(a_2)-a_3t_3x_3-a_4\sigma -a_6\lambda=0\\
\partial_{A}(a_1)-a_2t_3x_3-a_3\sigma-a_5\lambda-a_7\eta=0\\
\partial_{\mathcal{A}}(a_0)-a_1t_3x_3-a_2\sigma-a_4\lambda-a_6\eta=0.
\end{cases}
\end{align*}
By $\partial_{\mathcal{A}}(a_7)=0$, we have $a_7=c_7(x_1^2)^it_3x_3+\partial_{\mathcal{A}}(\lambda_7)$, for some $c_7\in k,\lambda_7\in \mathcal{A}^{2i}$.
Since $\partial_{\mathcal{A}}(a_6)-a_7t_3x_3=0$, we get  $c_7=0$ and $a_6=\lambda_7 t_3x_3+c_6(x_1^2)^it_3x_3+\partial_{\mathcal{A}}(\lambda_6)$ for some $c_6\in k, \lambda_6\in \mathcal{A}^{2i}$. Then $\partial_{\mathcal{A}}(a_5)-a_6t_3x_3-a_7\sigma=0$ implies that $c_6=0$ and $a_5=\lambda_7\sigma+\lambda_6t_3x_3 +c_5(x_1^2)^it_3x_3+\partial_{\mathcal{A}}(\lambda_5)$ for some $c_5\in k, \lambda_5\in \mathcal{A}^{2i}$. By
$\partial_{\mathcal{A}}(a_4)-a_5t_3x_3-a_6\sigma=0$, we have
\begin{align*}
\partial_{\mathcal{A}}(a_4)&=[\lambda_7\sigma+\lambda_6t_3x_3+c_5(x_1^2)^it_3x_3+\partial_{\mathcal{A}}(\lambda_5)]t_3x_3+[t_3\lambda_7 x_3+\partial_{\mathcal{A}}(\lambda_6)]\sigma \\
&=t_3\lambda_7(\sigma x_3+x_3\sigma)+\lambda_6t_3^2x_3^2+\partial_{\mathcal{A}}(\lambda_6)\sigma+\partial_{\mathcal{A}}(\lambda_5)t_3x_3+c_5(x_1^2)^it_3^2x_3^2\\
&=q_2t_3\lambda_7(x_2x_3+x_3x_2)+\partial_{\mathcal{A}}(\lambda_6\sigma)+\partial_{\mathcal{A}}(\lambda_5t_3x_3)+c_5(x_1^2)^it_3^2x_3^2\\
&=\partial_{\mathcal{A}}(\lambda_6\sigma)+\partial_{\mathcal{A}}(\lambda_5t_3x_3)+c_5(x_1^2)^it_3^2x_3^2,
\end{align*}
which implies $c_5=0$ and $a_4=\lambda_6\sigma+\lambda_5t_3x_3 +c_4(x_1^2)^it_3x_3+\partial_{\mathcal{A}}(\lambda_4)$ for some $c_4\in k$ and $\lambda_4\in \mathcal{A}^{2i}$. By $\partial_{\mathcal{A}}(a_3)-a_4t_3x_3-a_5\sigma-a_7\lambda=0$, we have
\begin{align*}
\partial_{\mathcal{A}}(a_3)&=a_4t_3x_3+a_5\sigma+a_7\lambda \\
&=[\lambda_6\sigma+\lambda_5t_3x_3+c_4(x_1^2)^it_3x_3+\partial_{\mathcal{A}}(\lambda_4)]t_3x_3+[\lambda_7\sigma+\lambda_6t_3x_3 +\partial_{\mathcal{A}}(\lambda_5)]\sigma \\
& +\partial_{\mathcal{A}}(\lambda_7)\lambda \\
&=\partial_{\mathcal{A}}(\lambda_5\sigma+\lambda_4t_3x_3+\lambda_7\lambda)+t_3\lambda_6(\sigma x_3+x_3\sigma)+c_4(x_1^2)^it_3^2x_3^2\\
&=\partial_{\mathcal{A}}(\lambda_5\sigma+\lambda_4t_3x_3+\lambda_7\lambda)+t_3q_2\lambda_6(x_2x_3+x_3x_2)+c_4(x_1^2)^it_3^2x_3^2\\
&=\partial_{\mathcal{A}}(\lambda_5\sigma+\lambda_4t_3x_3+\lambda_7\lambda)+c_4(x_1^2)^it_3^2x_3^2,
\end{align*}
which implies that $c_4=0$ and $a_3=(\lambda_5\sigma+\lambda_4t_3x_3+\lambda_7\lambda)+c_3(x_1^2)^it_3x_3+\partial_{\mathcal{A}}(\lambda_3)$ for some $c_5\in k, \lambda_5\in \mathcal{A}^{2i}$. Then
\begin{align*}
& \partial_{\mathcal{A}}(a_2)=a_3t_3x_3+a_4\sigma +a_6\lambda\\
&=[(\lambda_5\sigma+\lambda_4t_3x_3+\lambda_7\lambda)+c_3(x_1^2)^it_3x_3+\partial_{\mathcal{A}}(\lambda_3)]t_3x_3+[\lambda_6\sigma+\lambda_5t_3x_3 +\partial_{\mathcal{A}}(\lambda_4)]\sigma\\
&\quad +[t_3\lambda_7 x_3+\partial_{\mathcal{A}}(\lambda_6)]\lambda \\
&=\partial_{\mathcal{A}}(\lambda_6\lambda +\lambda_4\sigma +\lambda_3t_3x_3)+t_3\lambda_5(\sigma x_3+x_3\sigma)+t_3\lambda_7(\lambda x_3+x_3\lambda)+c_3(x_1^2)^it_3^2x_3^2\\
&=\partial_{\mathcal{A}}(\lambda_6\lambda +\lambda_4\sigma +\lambda_3t_3x_3)+c_3(x_1^2)^it_3^2x_3^2,
\end{align*}
which implies that $c_3=0$ and $a_2=(\lambda_6\lambda +\lambda_4\sigma +\lambda_3t_3x_3)+c_2(x_1^2)^it_3x_3+\partial_{\mathcal{A}}(\lambda_2)$.
We have
\begin{align*}
\partial_{A}(a_1)&=a_2t_3x_3+a_3\sigma+a_5\lambda+a_7\eta \\
                 &=[(\lambda_6\lambda +\lambda_4\sigma +\lambda_3t_3x_3)+c_2(x_1^2)^it_3x_3+\partial_{\mathcal{A}}(\lambda_2)]t_3x_3 \\ &+[(\lambda_5\sigma+\lambda_4t_3x_3+\lambda_7\lambda)+\partial_{\mathcal{A}}(\lambda_3)]\sigma\\
                 &+[\lambda_7\sigma+\lambda_6t_3x_3 +\partial_{\mathcal{A}}(\lambda_5)]\lambda+\partial_{\mathcal{A}}(\lambda_7)\eta\\
                 &=\partial_{\mathcal{A}}(\lambda_5\lambda +\lambda_3\sigma+\lambda_2t_3x_3)+t_3\lambda_6(\lambda x_3+x_3\lambda)+t_3\lambda_4(\sigma x_3+x_3\sigma)\\
                 &+\lambda_5\sigma^2+\lambda_7(\lambda\sigma +\sigma \lambda)+\partial_{\mathcal{A}}(\lambda_7)\eta +c_2(x_1^2)^it_3^2x_3^2\\
                 &=\partial_{\mathcal{A}}(\lambda_5\lambda +\lambda_3\sigma+\lambda_2t_3x_3)+\lambda_7[(u_1x_1+u_2x_2)q_2x_2+q_2x_2(u_1x_1+u_2x_2)]\\
                 &+\partial_{\mathcal{A}}(\lambda_7)\eta +c_2(x_1^2)^it_3^2x_3^2\\
                 &=\partial_{\mathcal{A}}(\lambda_5\lambda +\lambda_3\sigma+\lambda_2t_3x_3+\lambda_7\eta)+c_2(x_1^2)^it_3^2x_3^2
\end{align*}
which implies that $c_2=0$ and $a_1=(\lambda_5\lambda +\lambda_3\sigma+\lambda_2t_3x_3+\lambda_7\eta)+c_1(x_1^2)^it_3x_3+\partial_{\mathcal{A}}(\lambda_1)$ for some $c_1\in k$ and $\lambda_1\in \mathcal{A}^{2i}$.
Then \begin{align*}
\partial_{\mathcal{A}}(a_0)&=a_1t_3x_3+a_2\sigma+a_4\lambda+a_6\eta\\
&=(\lambda_5\lambda +\lambda_3\sigma+\lambda_2t_3x_3+\lambda_7\eta)t_3x_3 +c_1(x_1^2)^it_3x_3 +\partial_{\mathcal{A}}(\lambda_1)t_3x_3\\
&+[\lambda_6\lambda +\lambda_4\sigma +\lambda_3t_3x_3+\partial_{\mathcal{A}}(\lambda_2)]\sigma +[\lambda_6\sigma+\lambda_5t_3x_3 +\partial_{\mathcal{A}}(\lambda_4)]\lambda\\
&+[t_3\lambda_7 x_3-\partial_{\mathcal{A}}(\lambda_6)]\eta \\
&=\partial_{\mathcal{A}}[\lambda_2\sigma+\lambda_4\lambda+\lambda_1t_3x_3]+\partial_{\mathcal{A}}(\lambda_6)\eta +\lambda_6(\lambda\sigma +\sigma\lambda)+t_3\lambda_5(\lambda x_3+x_3\lambda)\\
&+t_3\lambda_3(\sigma x_3+x_3\sigma)+t_3\lambda_7(\eta x_3+x_3\eta)+c_1(x_1^2)^it_3^2x_3^2\\
&=\partial_{\mathcal{A}}[\lambda_2\sigma+\lambda_4\lambda+\lambda_1t_3x_3+\lambda_6\eta]+c_1(x_1^2)^it_3^2x_3^2,
\end{align*}
which implies that $c_1=0$ and $a_0=\lambda_2\sigma+\lambda_4\lambda+\lambda_1t_3x_3+\lambda_6\eta+c_0(x_1^2)^it_3x_3+\partial_{\mathcal{A}}(\lambda_0)$ for some $c_0\in k$ and $\lambda_0\in \mathcal{A}^{2i}$. Therefore,
\begin{align*}
z &=a_0+\sum\limits_{i=1}^7a_ie_i =(\lambda_2\sigma+\lambda_4\lambda+\lambda_1t_3x_3+\lambda_6\eta)+c_0(x_1^2)^it_3x_3+\partial_{\mathcal{A}}(\lambda_0)\\
                               &+(\lambda_5\lambda +\lambda_3\sigma+\lambda_2t_3x_3+\lambda_7\eta)e_1+\partial_{\mathcal{A}}(\lambda_1)e_1 +(\lambda_6\lambda +\lambda_4\sigma +\lambda_3t_3x_3)e_2\\
                               &+\partial_{\mathcal{A}}(\lambda_2)e_2 +(\lambda_7\sigma+\lambda_4t_3x_3+\lambda_7\lambda)e_3+\partial_{\mathcal{A}}(\lambda_3)e_3 +(\lambda_6\sigma+\lambda_5t_3x_3)e_4 +\partial_{\mathcal{A}}(\lambda_4)e_4\\
                               &+(\lambda_7\sigma+\lambda_6t_3x_3)e_5 +\partial_{\mathcal{A}}(\lambda_5)e_5+t_3\lambda_7 x_3e_6+\partial_{\mathcal{A}}(\lambda_6)e_6+\partial_{\mathcal{A}}(\lambda_7)e_7\\
                               &=\partial_F[\lambda_0+\sum\limits_{i=1}^7\lambda_ie_i]+c_0(x_1^2)^it_3x_3\\
                               &=\partial_F[\lambda_0+\sum\limits_{i=1}^7\lambda_ie_i]+\partial_F[\frac{c_0(x_1^2)^{i-1}t_3x_3}{u_1^2}(t_3x_3e_7+q_2x_2e_6+\lambda e_4+\eta e_2+\frac{5u_2^2\lambda}{q_2^2})].
\end{align*}
So $z\in B^{2i+1}(F)$ and $H^{2i+1}(F)=0$ for any $i\ge 1$.

\def\refname{References}

\end{document}